\documentclass[11pt,a4paper]{article}

\usepackage[T1]{fontenc}
\usepackage[utf8]{inputenc}

\usepackage[left=2.45cm, top=2.45cm,bottom=2.45cm,right=2.45cm]{geometry}
\usepackage{amsmath}
\usepackage{amssymb,amsthm,graphicx,tikz}
\usepackage[british]{babel}
\usepackage{bbm}

\usepackage{tikz}
\usetikzlibrary{calc, patterns}
\tikzset{
  schraffiert/.style={pattern=horizontal lines, pattern color=#1},
  schraffiert/.default=black
}

\usepackage[
	bookmarks=true,
	bookmarksnumbered=true,
	bookmarksopen=true,
	unicode=true,
	pdftoolbar=true,
	pdfmenubar=true,
	pdffitwindow=false,
	pdfstartview={FitH},
	pdftitle={},
	pdfauthor={},
	pdfsubject={},
	pdfcreator={},
	pdfproducer={},
	pdfkeywords={},
	pdfnewwindow=true,
	colorlinks=true,
	linkcolor=black,
	citecolor=black,
	filecolor=black,
	urlcolor=black]{hyperref}

\makeatletter
	\def\@cite#1#2{[\textbf{#1}\if@tempswa , #2\fi]}	%text
	\def\@biblabel#1{[#1]}								%bibliography
\makeatother
\numberwithin{equation}{section}
\numberwithin{figure}{section}

\newcommand{\ind}{\mathbbm{1}}

\makeatletter
\let\save@mathaccent\mathaccent
\newcommand*\if@single[3]{%
  \setbox0\hbox{${\mathaccent"0362{#1}}^H$}%
  \setbox2\hbox{${\mathaccent"0362{\kern0pt#1}}^H$}%
  \ifdim\ht0=\ht2 #3\else #2\fi
  }
\newcommand*\rel@kern[1]{\kern#1\dimexpr\macc@kerna}
\newcommand*\widebar[1]{\@ifnextchar^{{\wide@bar{#1}{0}}}{\wide@bar{#1}{1}}}
\newcommand*\wide@bar[2]{\if@single{#1}{\wide@bar@{#1}{#2}{1}}{\wide@bar@{#1}{#2}{2}}}
\newcommand*\wide@bar@[3]{%
  \begingroup
  \def\mathaccent##1##2{%
    \let\mathaccent\save@mathaccent
    \if#32 \let\macc@nucleus\first@char \fi
    \setbox\z@\hbox{$\macc@style{\macc@nucleus}_{}$}%
    \setbox\tw@\hbox{$\macc@style{\macc@nucleus}{}_{}$}%
    \dimen@\wd\tw@
    \advance\dimen@-\wd\z@
    \divide\dimen@ 3
    \@tempdima\wd\tw@
    \advance\@tempdima-\scriptspace
    \divide\@tempdima 10
    \advance\dimen@-\@tempdima
    \ifdim\dimen@>\z@ \dimen@0pt\fi
    \rel@kern{0.6}\kern-\dimen@
    \if#31
      \overline{\rel@kern{-0.6}\kern\dimen@\macc@nucleus\rel@kern{0.4}\kern\dimen@}%
      \advance\dimen@0.4\dimexpr\macc@kerna
      \let\final@kern#2%
      \ifdim\dimen@<\z@ \let\final@kern1\fi
      \if\final@kern1 \kern-\dimen@\fi
    \else
      \overline{\rel@kern{-0.6}\kern\dimen@#1}%
    \fi
  }%
  \macc@depth\@ne
  \let\math@bgroup\@empty \let\math@egroup\macc@set@skewchar
  \mathsurround\z@ \frozen@everymath{\mathgroup\macc@group\relax}%
  \macc@set@skewchar\relax
  \let\mathaccentV\macc@nested@a
  \if#31
    \macc@nested@a\relax111{#1}%
  \else
    \def\gobble@till@marker##1\endmarker{}%
    \futurelet\first@char\gobble@till@marker#1\endmarker
    \ifcat\noexpand\first@char A\else
      \def\first@char{}%
    \fi
    \macc@nested@a\relax111{\first@char}%
  \fi
  \endgroup
}
\makeatother

\newtheorem {theorem}{Theorem}[section]
\newtheorem {proposition}[theorem]{Proposition}
\newtheorem {lemma}[theorem]{Lemma}
\newtheorem {corollary}[theorem]{Corollary}

\theoremstyle{definition}
\newtheorem{definition}[theorem]{Definition}
\newtheorem {remark}[theorem]{Remark}

\newcommand{\Mod}[1]{\ (\mathrm{mod}\ #1)}

\newcommand{\BB}{\mathbb{B}}
\newcommand{\EE}{\mathbb{E}}
\newcommand{\HH}{\mathbb{H}}
\newcommand{\NN}{\mathbb{N}}
\newcommand{\ZZ}{\mathbb{Z}}
\newcommand{\PP}{\mathbb{P}}
\newcommand{\MM}{\mathbb{M}}
\newcommand{\RR}{\mathbb{R}}
\newcommand{\LL}{\mathbb{L}}
\newcommand{\RRd}{\mathbb{R}^{d+1}}
\renewcommand{\SS}{\mathbb{S}}
\newcommand{\XX}{\mathbb{X}}

\newcommand{\cF}{\mathcal{F}}

\newcommand{\cP}{\mathcal{P}}
\newcommand\cV{\mathcal{V}}

\DeclareMathOperator{\artanh}{artanh}
\DeclareMathOperator{\Vol}{Vol}
\DeclareMathOperator{\pos}{pos}
\DeclareMathOperator{\aff}{aff}
\DeclareMathOperator{\lin}{lin}

\DeclareMathOperator{\conv}{conv}

\DeclareMathOperator{\dist}{dist}

\DeclareMathOperator{\cotanh}{cotanh}

\newcommand{\dint}{\mathrm{d}}
\usepackage{accents}
\DeclareMathSymbol{\tildesym}{\mathord}{largesymbols}{"65}

\DeclareMathOperator{\arccosh}{arcosh}

\makeatletter
\let\@fnsymbol\@alph
\makeatother

\begin{document}

\title{\bfseries Beta-star polytopes and\\ hyperbolic stochastic geometry}

\author{%
 Thomas Godland\footnotemark[1]
\and
Zakhar Kabluchko\footnotemark[2]%
\and
Christoph Th\"ale\footnotemark[3]%
}

\date{}
\renewcommand{\thefootnote}{\fnsymbol{footnote}}
\footnotetext[1]{%
    University of M\"unster, Germany. Email: thomas.godland@uni-muenster.de
}

\footnotetext[2]{%
	University of M\"unster, Germany. Email: zakhar.kabluchko@uni-muenster.de
}

\footnotetext[3]{%
    Ruhr University Bochum, Germany. Email: christoph.thaele@rub.de
}

\maketitle

\begin{abstract}\noindent
%In this paper the class of beta-star polytopes is introduced and studied.
Motivated by problems of hyperbolic stochastic geometry we introduce and study the class of beta-star polytopes.
A beta-star polytope is defined as the convex hull of an inhomogeneous Poisson processes on the complement of the unit ball in $\mathbb{R}^d$ with density proportional to $(\|x\|^2-1)^{-\beta}$, where $\|x\|>1$ and $\beta>d/2$. Explicit formulas for various geometric and combinatorial functionals associated with beta-star polytopes are provided, including the expected number of $k$-dimensional faces, the expected external angle sums and the expected intrinsic volumes.  Beta-star polytopes are relevant in the context of hyperbolic stochastic geometry, since they are tightly connected to the typical cell of a Poisson-Voronoi tessellation as well as the zero cell of a Poisson hyperplane tessellation in hyperbolic space. The general results for beta-star polytopes are used to provide explicit formulas for the expected $f$-vector of the typical hyperbolic Poisson-Voronoi  cell and the hyperbolic Poisson zero cell. Their asymptotics for large intensities and their monotonicity behaviour  is discussed as well. Finally, stochastic geometry in the de Sitter half-space is studied as the hyperbolic analogue to recent investigations about random cones generated by random points on half-spheres in spherical or conical stochastic geometry.


\end{abstract}

\tableofcontents

\section{Introduction}\label{sec:Motivation}

%\subsection{Motivation}\label
Stochastic geometry is concerned with the analysis of complex spatial random structures. While traditionally research has been centred around models in Euclidean space, the focus has partially moved to stochastic geometry models in non-Euclidean spaces in the last years, most notably to spaces of constant curvature $+1$ and $-1$. As examples which are most relevant to our situation we mention here the studies on hyperbolic random geometric graphs~\cite{BodeFountoulakisMuller,FountulakisHoornMuller,FountulakisYukich}, hyperbolic random polytopes~\cite{BesauLudwigWerner,BesauRosenThaele,BesauThaele}, hyperbolic Poisson line or hyperplane tessellations~\cite{HeroldHugThaele,PorretBlanc,SantaloYanez,TykessonCalka} and hyperbolic Poisson-Voronoi (or Poisson-Delaunay) tessellations~\cite{BenjaminiEtAl,BenjaminiEtAl21,CalkaChapron,HansenMuller,Isokawa,IsokawaPlane,NielsenNock}.

\medskip
In the focus of the present paper are the typical cell of a Poisson-Voronoi tessellation and the zero cell of a Poisson hyperplane tessellation in a $d$-dimensional hyperbolic space.
We start by a description of these two objects. For this, let
$$
\HH^d:= \{(x_0,x_1,\ldots,x_d)\in\RRd:x_0^2- x_1^2 - \ldots - x_d^2 = 1, x_0>0\}
$$
be the hyperboloid model for a $d$-dimensional hyperbolic space, $d\in \NN$. It is well known that via the so-called gnomonic projection $\Pi_{\text{gn}}:\HH^d\to\BB^d$ it can be identified with the Klein model in the $d$-dimensional open unit ball $\BB^d$ and via the stereographic projection $\Pi_{\text{st}}:\HH^d\to\BB^d$ with the Poincar\'e model in $\BB^d$; see Section~\ref{subsec:HyperbolicGeo} for details and further explanations.

%\medskip

\paragraph{Poisson-Voronoi tessellation of the hyperbolic space.}
For $\lambda>0$ let $\eta_{d,\lambda}$ be a stationary Poisson process in $\HH^d$ with intensity $\lambda>0$. More precisely, this means that $\eta_{d,\lambda}$ is a Poisson process whose intensity measure is a multiple $\lambda$ of the $d$-dimensional hyperbolic volume measure on $\HH^d$. In particular, this implies that the law of $\eta_{d,\lambda}$ is invariant under all isometries of $\HH^d$. If $d_{\text{hyp}}(\,\cdot\,,\,\cdot\,)$ denotes the hyperbolic distance on $\HH^d$, we can associate with each point $x\in \eta_{d,\lambda}$ its \textit{Voronoi cell}
$$
V(x;\eta_{d,\lambda}) := \{y\in\HH^d: d_{\text{hyp}}(x,y)\leq d_{\text{hyp}}(z,y)\text{ for all }z\in\eta_{d,\lambda}\}.
$$
In other words, $V(x;\eta_{d,\lambda})$ is the set of all points of $\HH^d$ that are closer to $x$ than to any other point of $\eta_{d,\lambda}$, in the sense of the hyperbolic distance. As in the Euclidean case, it is not hard to verify that each Voronoi cell is a hyperbolic random polytope. The collection of all such Voronoi cells is the \textit{hyperbolic Poisson-Voronoi tessellation} $\cV_{d,\lambda}$ with intensity $\lambda$; see the left panel of Figure~\ref{fig:Voronoi} showing it in the Poincar\'e model. We are interested in what is known as the \textit{typical cell} $V^{\text{typ}}_{d,\lambda}$ of $\cV_{d,\lambda}$. Intuitively, one can think of $V^{\text{typ}}_{d,\lambda}$ as a randomly selected cell of $\cV_{d,\lambda}$, where each cell has the same chance of being selected, independently of size and shape. Formally, the distribution of $V^{\text{typ}}_{d,\lambda}$ can be defined using Palm calculus. However, there is an alternative way based on Slivnyak's theorem for Poisson processes. In fact, if $e=(1,0,\ldots,0)$ denotes the apex of the hyperboloid $\HH^d$ we may define $V^{\text{typ}}_{d,\lambda}$ as the Voronoi cell of $e$ if the point $e$ is added to the Poisson process $\eta_{d,\lambda}$:
\begin{equation}\label{eq:DefTypicalCell}
V^{\text{typ}}_{d,\lambda} := \{y\in\HH^d: d_{\text{hyp}}(y,e)\leq d_{\text{hyp}}(y,z)\text{ for all } z\in\eta_{d,\lambda}\}.
\end{equation}
In Theorem~\ref{thm:TypicalVoronoiCell} below we shall identify the distribution (of the gnomonic projection) of $V^{\text{typ}}_{d,\lambda}$.

\begin{figure}[t]
	\centering
	\includegraphics[width=0.49\columnwidth]{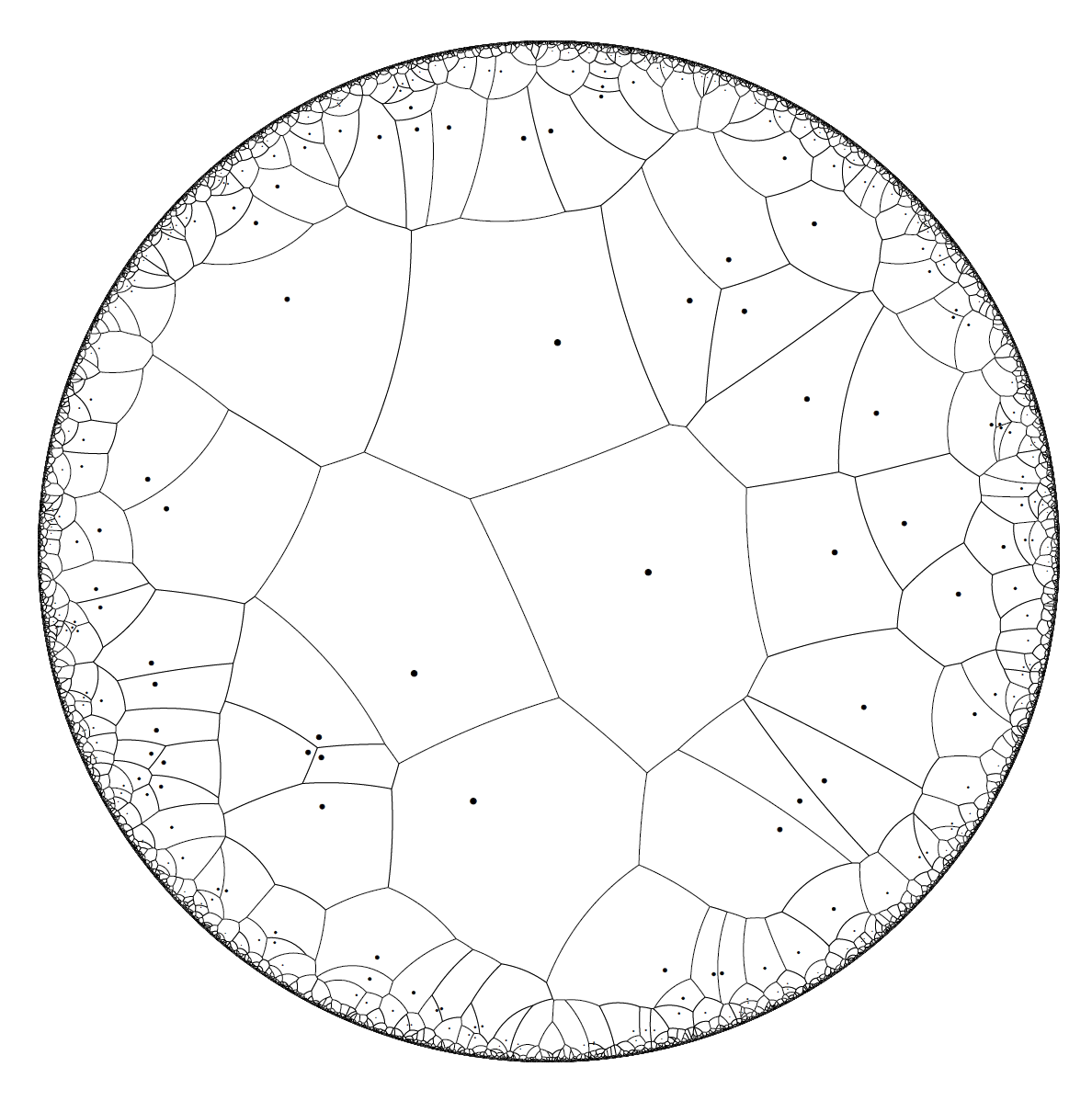}
	\includegraphics[width=0.49\columnwidth]{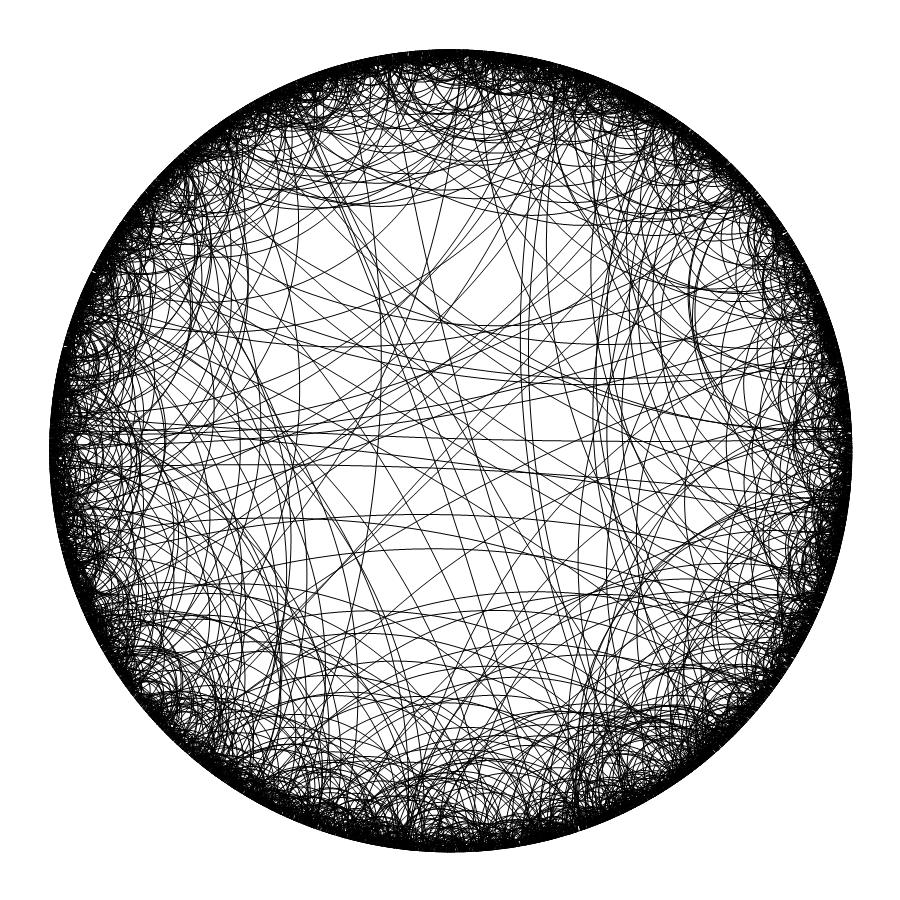}
	\caption{Simulation of a Poisson-Voronoi tessellation (left panel) and a Poisson hyperplane tessellation (right panel) in the Poincar\'e disk model of the hyperbolic plane.}
\label{fig:Voronoi}
\end{figure}

\paragraph{Poisson hyperplane tessellation of the hyperbolic space.}
To describe the next model,  we fix $\lambda>0$ and  consider a Poisson process $\xi_{d, \lambda}$ on the space $A_h(d,d-1)$ of hyperbolic hyperplanes in $\HH^d$  (which are defined as intersections of linear hyperplanes in $\RR^{d+1}$ with $\HH^d$).
The intensity measure of $\xi_{d, \lambda}$ is chosen to be the following infinite measure on $A_h(d,d-1)$:
$$
\mu_{d,\lambda}(\,\cdot\,):= \lambda \int\limits_{\SS_e^{d-1}}\int\limits_{0}^\infty (\cosh \theta)^{d-1}\, \ind_{\{H_e(u,\theta)\in\,\cdot\,\}} \,
\dint \theta \, \sigma_{d-1;e}(\dint u).
$$
Here, $e=(1,0,\ldots,0)$ is the apex (or the origin) of $\HH^d$, $\SS_e^{d-1}$ stands for the (hyperbolic) unit sphere in $\HH^d$ centred at $e$, $\sigma_{d-1;e}$ is the normalized spherical Lebesgue measure on $\SS_e^{d-1}$,  and $H_e(u,\theta)$ denotes the unique hyperbolic hyperplane orthogonal to $u$ having hyperbolic distance $\theta$ to $e$.  It is known that the measure $\mu_{d,\lambda}$ is the unique (up to a multiplicative constant) measure on $A_h(d,d-1)$ which is invariant under all isometries of $\HH^d$; see~\cite{santalo_book} and, in particular,  Equation~(17.54) on p.~309 there.  Equivalently, we can describe $\xi_{d,\lambda}$ as follows: the hyperbolic distances from $e$ to the hyperplanes from $\xi_{d,\lambda}$ form a Poisson process on $(0,\infty)$ with Lebesgue intensity $\theta \mapsto \cosh^{d-1}\theta$, while their normal directions are independent and uniformly distributed  on the unit sphere. The hyperplanes from $\xi_{d,\lambda}$ decompose the space $\HH^d$ into random subsets with pairwise disjoint interiors; see the right panel of Figure~\ref{fig:Voronoi} for a realization in the Poincar\'e model.  With probability one, there is a unique such set containing the origin $e$. We refer to this set as the \textit{hyperbolic Poisson zero cell} and denote it by $Z_{d,\lambda}^0$. Remarkably and in contrast to the case of the typical Voronoi cell described above, it is known for $d=2$ from~\cite{BenjaminiJonassonEtAL,PorretBlanc,TykessonCalka} (and can also be concluded from a paper of Hoffmann-J\o rgensen~\cite{HoffmannJoergensen} for general $d$, as we shall demonstrate) that $Z_{d,\lambda}^0$ is a \textit{bounded} hyperbolic random polytope only if the intensity parameter $\lambda$ lies above a certain critical value, while below this value $Z_{d,\lambda}^0$ is hyperbolically unbounded with positive probability. In the forthcoming Theorem~\ref{thm:ZeroCell} we shall identify the precise distribution (again, of the gnomonic projection) of $Z_{d,\lambda}^0$.

\paragraph{Beta$^*$ sets.}
The connection between the typical cell $V^{\text{typ}}_{d,\lambda}$ and the zero cell $Z_{d,\lambda}^0$ comes from the fact that, under gnomonic projection and after application of convex duality/polarity, both random polytopes can be identified with the convex hull of an inhomogeneous Poisson process on $\RR^d\backslash\widebar \BB^d$ whose Lebesgue intensity is proportional to $(\|x\|^2-1)^{-\beta}$, for $\|x\|>1$; see Figure~\ref{fig:beta_star_polys}.  As it turns out, in the Voronoi case one has that $\beta=d$, whereas one has to choose $\beta=(d+1)/ 2$ in case of the Poisson zero cell. This motivates an independent study of what we call a beta$^*$ set (or a beta-star set), which is defined as the closed convex hull of the atoms of the Poisson process mentioned above. In general, this Poisson process has a countably infinite number of atoms and the beta$^*$ set is a closed convex set. However, for some values of parameters it turns out that a \textit{finite} subset of these atoms suffices to generate this closed convex hull, in which case we speak of a beta$^*$ polytope.

\begin{figure}[t]
	\centering
	\includegraphics[width=0.48\columnwidth]{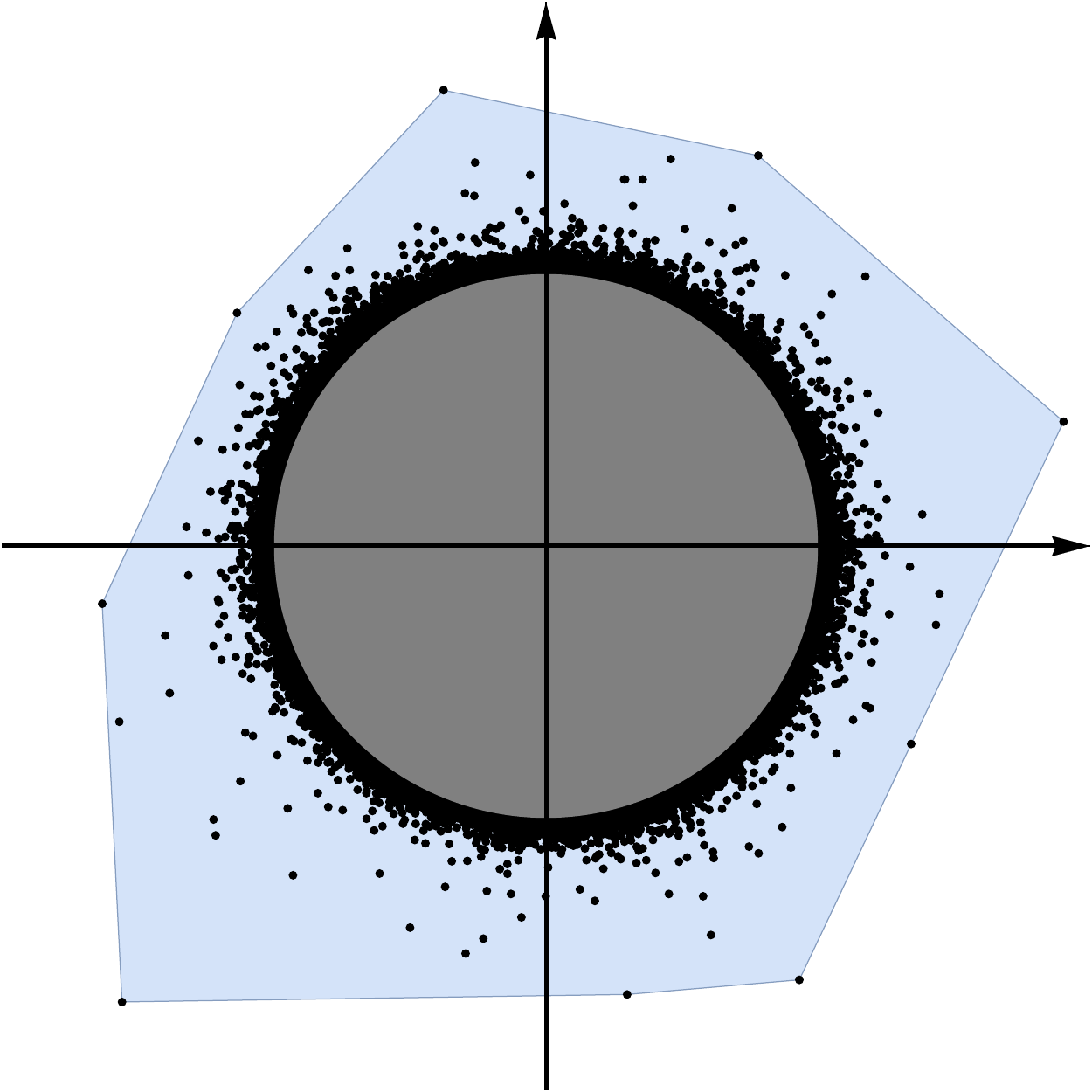}
	\includegraphics[width=0.48\columnwidth]{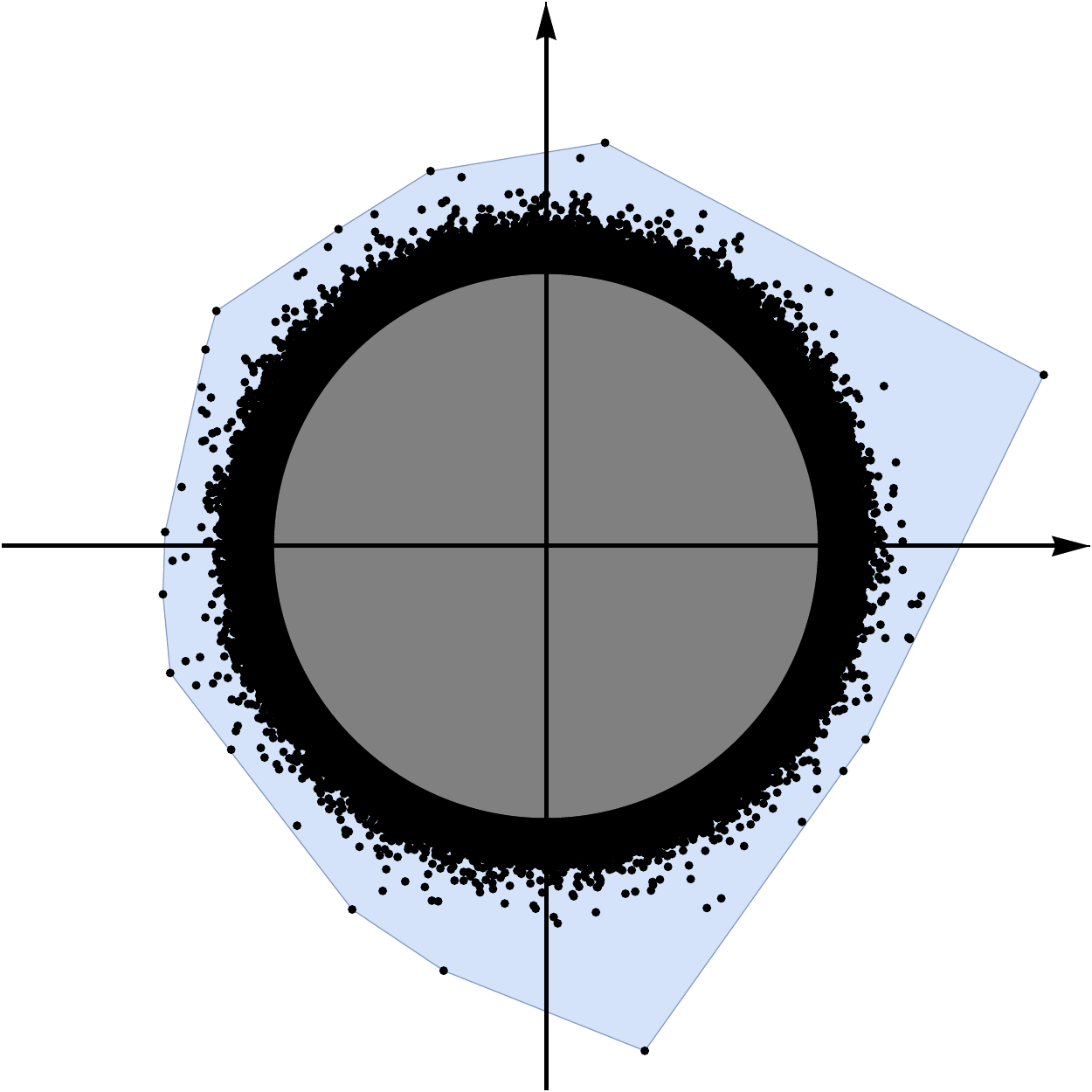}
	\caption{Beta$^*$ polytopes in dimension $d=2$ with $\beta=3$ and $\beta=6$.}
\label{fig:beta_star_polys}
\end{figure}

\paragraph{Summary of the main results.}
Our main results can roughly be summarized as follows:
\begin{itemize}
\item[(a)] We provide conditions on the  parameters $\beta$, $d$, and an additional intensity parameter which we call $\alpha$, under which a beta$^*$ set is a polytope; see Theorem~\ref{thm:PolytopeOrNot}.
\item[(b)] We compute the expectation of the so-called $T$-functional of beta$^*$ sets; see Theorem~\ref{thm:Tfunctional}. As a consequence, we derive a formula for the expected intrinsic volumes of the beta$^*$ set; see Proposition~\ref{prop:IntrinsicVolume}.
\item[(c)] We compute the expected $f$-vector (i.e., the expected number of faces of any  given dimension) of a beta$^*$ polytope; see Theorem~\ref{thm:f-vector_beta_star}. The main tool is the so-called canonical decomposition of the beta$^*$ intensity which we state and prove in Theorem~\ref{thm:canonical_decomp}. We also compute the expected external angle sums of beta$^*$ polytopes in Theorem~\ref{thm:exp_ext_anlge_sums}.
\item[(d)] We show that the typical Voronoi cell in the hyperbolic space $\HH^d$ is related to the beta$^*$ polytope with $\beta=d$ and compute its expected $f$-vector; see Theorems~\ref{thm:TypicalVoronoiCell} and~\ref{thm:TypicalVoronoiCellExpectedFaceNums}.
\item[(e)] We introduce a class of (in general, non-stationary) Poisson hyperplane tessellations in the hyperbolic space and show that the zero cells of these tessellations are related to beta$^*$ polytopes by convex duality; see Theorem~\ref{thm:ZeroCell}. This fact can be used to compute their expected $f$-vectors. The stationary hyperplane tessellation is included as a special case; see Theorem~\ref{thm:f-vector,beta=(d+1)/2}.
\item[(f)] We analyse the behaviour of beta$^*$ polytopes as the intensity parameter goes to $\infty$ and show that they converge to their Euclidean counterparts; see Theorem~\ref{thm:convergence_beta_star}. We also prove that the convergence of the expected $f$-vectors takes place in a strictly decreasing way; see Theorem~\ref{thm:Monotonicity}.
\end{itemize}

\paragraph{Related polytopes.}
The results listed above complement our findings for beta and beta' polytopes in~\cite{KabluchkoAngles,KabluchkoZeroPolytope,KabluchkoMarynychTemesvariThaele,KabluchkoTemesvariThaele,KabluchkoThaeleVoronoiSphere,KabluchkoThaeleZaporozhets}, which in turn have found application to Voronoi and hyperplane tessellations in Euclidean and spherical spaces. The beta polytopes are defined as convex hulls of i.i.d.\ samples in the unit ball $\BB^d$ with density proportional to $(1-\|x\|^2)^\beta$, $\|x\|<1$. Similarly, beta' polytopes are defined as convex hulls of i.i.d.\ samples in $\RR^d$ with density proportional to $(1+\|x\|^2)^{-\beta}$, $x\in \RR^d$. As we have shown in the above mentioned papers, the beta' polytopes are related to  Voronoi and hyperplane tessellations on the sphere and to random polytopes in the half-sphere. In the present paper we will show that the beta$^*$ polytopes are related to similar objects in the hyperbolic space and the de Sitter half-space.
Moreover, the infinite intensity limits of both, the beta' and beta$^*$ polytopes, which are called Poisson polytopes and defined as convex hulls of Poisson processes with intensities proportional to $\|x\|^{-2\beta}$ on $\RR^d\backslash \{0\}$, are related to Voronoi and hyperplane tessellations of the Euclidean space. The beta and beta' distributions with densities proportional to  $(1-\|x\|^2)^\beta$, $\|x\|<1$, and $(1+\|x\|^2)^{-\beta}$, $x\in \RR^d$, (together with the Gaussian distribution and the uniform distribution on the sphere, which are their limit cases) were characterized by Ruben and Miles~\cite{Ruben_Miles} as the only probability distributions satisfying the so-called canonical decomposition property. The beta$^*$ intensity proportional to $(\|x\|^2-1)^{-\beta}$, $\|x\|>1$, defines an \textit{infinite} measure (for $d\geq 2$) and does not appear in the classification of Ruben and Miles, although it also satisfies a variant of canonical decomposition, as we shall show in Theorem~\ref{thm:canonical_decomp}.

\paragraph{Structure of the paper.}
Our main results on beta$^*$ polytopes will be stated in Section~\ref{sec:beta_star_polytope_results}. Applications to  hyperbolic stochastic geometry will be presented in Section~\ref{sec:SpecialCases}. The remaining sections contain proofs.

\paragraph{Notation.} We write  $\|\,\cdot\,\|$ for the Euclidean norm and $\langle\,\cdot\,,\,\cdot\,\rangle$ for the Euclidean scalar product in $\RR^d$.  We denote by $\BB^d :=\{x\in \RR^d: \|x\|<1\}$ the open unit ball and by $\widebar\BB^d:=\{x\in \RR^d: \|x\|\leq 1\}$ its closure.
Following the book of Schneider and Weil~\cite[p.~13]{SW} we write $\kappa_d$ for the volume of $\BB^d$ and $\omega_d$ for the surface area (that is, the $(d-1)$-dimensional Hausdorff measure) of the unit sphere $\SS^{d-1} := \{x\in \RR^d: \|x\|=1\}$ in $\RR^d$. It is well known that
$$
\kappa_d = \frac{\pi^{d/2}}{\Gamma(\frac d2 + 1)}
\quad
\text{ and }
\quad
\omega_d = d \kappa_d = \frac{d\pi^{d/2}}{\Gamma(\frac d2 + 1)}= \frac{2\pi^{d/2}}{\Gamma(\frac d2)}.
$$
By $\conv(A)$, $\aff (A)$ and $\lin (A)$ we indicate the convex, the affine and the linear hull of a set $A\subset\RR^d$, respectively.

\section{Statement of the main results}\label{sec:beta_star_polytope_results}

\subsection{Definition and existence of \texorpdfstring{beta$^*$}{beta*} polytopes}\label{subsec:def_and_existence_of_beta_star}
Let us define the objects we are interested in.  Fix some space dimension $d\in \NN$. For $\alpha>0$ and $\beta > d/2$ we denote by $\zeta_{d,\alpha,\beta}$ a Poisson process on $\RR^d\backslash \widebar\BB^d = \{x\in \RR^d: \|x\|>1\}$ whose intensity $f_{d,\alpha,\beta}$ (with respect to the Lebesgue measure)  is given by
\begin{equation}\label{eq:DensityZeta}
f_{d,\alpha,\beta}(x) = \frac{\alpha\,\tilde c_{d,\beta}}{(\|x\|^2-1)^{\beta}},
\qquad
\|x\|>1,
\qquad
\tilde c_{d,\beta} := {\Gamma(\beta)\over\pi^{d\over 2}\Gamma(\beta-{d\over 2})}.
\end{equation}
It should be observed that although $\zeta_{d,\alpha,\beta}$ does not have points inside the unit ball $\widebar\BB^d$, the points of $\zeta_{d,\alpha,\beta}$ accumulate in an outside neighbourhood of the boundary of $\BB^d$, as we shall see in a moment (see also Figure~\ref{fig:beta_star_polys}). Introducing the constant $\tilde c_{d,\beta}$, which may look unnatural at a first sight, will prove convenient at many places, for example in Lemma~\ref{lem:Projection}. For background information and properties of Poisson processes we refer to~\cite{LPbook} and~\cite[Chapter~3] {resnick_book}.

\begin{definition}[beta$^*$ sets]
The closed convex hull of the atoms of $\zeta_{d,\alpha,\beta}$ is denoted by $P_{d,\alpha,\beta}$ and called a \emph{beta$^*$ set}  with parameters $\alpha>0$ and $\beta>d/2$.
\end{definition}

Realizations of beta$^*$ sets are shown on Figure~\ref{fig:beta_star_polys}. We record some basic properties of the Poisson process $\zeta_{d,\alpha,\beta}$.

\begin{proposition}\label{prop:poi_process_properties}
Let $d\in \NN$, $\beta > \max(d/2,1)$ and $\alpha>0$. Then, the total number of atoms of $\zeta_{d,\alpha,\beta}$ is infinite a.s., while the number of atoms having norm $\geq r$ is finite a.s.\ for every $r>1$. Also, with probability $1$, all points on the unit sphere $\SS^{d-1}$ are accumulation points for the atoms of $\zeta_{d,\alpha,\beta}$.
\end{proposition}
\begin{proof}
Condition $\beta > 1$ ensures that  $\int_{\|x\| > 1}f_{d,\alpha,\beta}(x) \dint x  = +\infty$ meaning that the total number of atoms of $\zeta_{d,\alpha,\beta}$ is infinite a.s. On the other hand, condition $\beta>d/2$  ensures  that $\int_{\|x\|>r} f_{d,\alpha,\beta}(x) \dint x < \infty$ for all $r>1$ meaning that the number of atoms outside any ball $r \BB^d$ of radius $r>1$ is finite a.s. To prove the last claim of the proposition, represent the atoms as $R_1 U_1, R_2U_2,\ldots$, where $R_1>R_2>\ldots>1$ are the distances from the atoms to the origin, while $U_1,U_2,\ldots$ are independent random vectors whose distribution is  uniform on the unit sphere due to the isotropy of the intensity. Since $R_n\to 1$ as $n\to\infty$ and the sequence $U_1,U_2,\ldots$ is dense on the unit sphere with probability $1$, we conclude that the unit sphere belongs to the closure of the set of atoms of $\zeta_{d,\alpha,\beta}$.
\end{proof}

\begin{remark}
Although in the following we are interested in the case $d\geq 2$, let us mention that for $d=1$ and $\beta\in (1/2,1)$ the number of atoms of $\zeta_{1,\alpha,\beta}$ is finite a.s.
For $d=\beta=1$ the number of atoms is infinite and they cluster at $\pm 1$  a.s.
\end{remark}

By Proposition~\ref{prop:poi_process_properties}, the beta$^*$ set $P_{d,\alpha,\beta}$ is a compact convex set containing $\widebar\BB^d$ with probability $1$.  However, $P_{d,\alpha,\beta}$ need not be a polytope a.s.\  since, as we shall see in a moment,  with positive probability its boundary  might touch the boundary of $\widebar\BB^d$ at infinitely many points that become exposed points~\cite[p.~18]{SchneiderBook} of $P_{d,\alpha,\beta}$; see the right panel of Figure~\ref{fig:regime_and_cantor}. In order to characterize the cases when this does not happen, let us denote by $R$ the radius of the largest ball centred at the origin and contained in $P_{d,\alpha,\beta}$. That is, we put
$$
R := \sup\{r>0: r \widebar\BB^d \subset P_{d,\alpha, \beta}\}.
$$
Note that if $R>1$, then $P_{d,\alpha,\beta}$ is a convex hull of the atoms of $\zeta_{d,\alpha,\beta}$ with norm exceeding $R$, and, since their number is a.s.\ finite, a polytope.

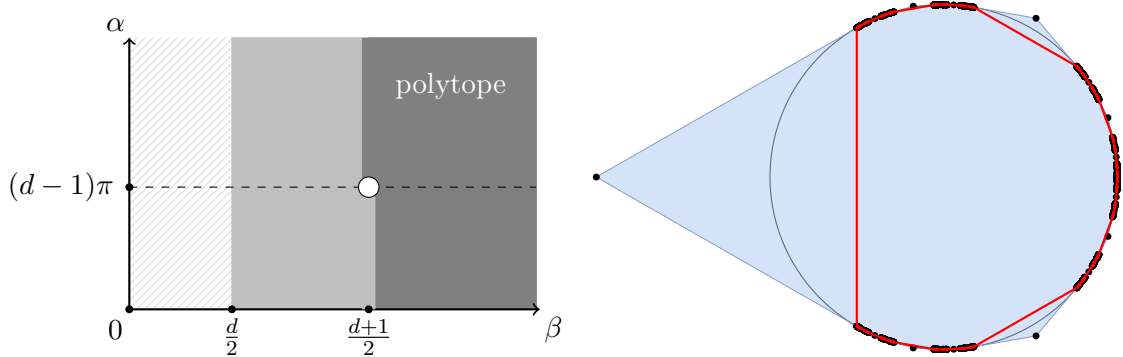
\begin{figure}[t]
%		\begin{center}
		\begin{tikzpicture}[scale=0.9]
	\begin{scope}
%\draw [thick,domain=-2:2] plot ({sinh(\x)}, {cosh(\x)});
%\draw[thick] (-4,1) -- (4,1);
%\draw[schraffiert=black] (0,0) rectangle (1.5,4);
\fill[pattern=north east lines, pattern color=lightgray!50] (0,0) rectangle (1.5,4);
\filldraw [fill = gray, draw = gray] (3.5,0) rectangle (5.95,4) ;
%\shade[left color=red!50,right color=green!50](3.5,0) rectangle (6,4) ;
%\shade[left color=gray,right color=white](6.2,0) rectangle (9,4) ;
\filldraw [fill = lightgray, draw = lightgray] (1.5,0) rectangle (3.5,4) ;
\draw[->,thick] (0,0) -- (6,0);
\draw[->,thick] (0,0) -- (0,4);
\draw[fill=black] (0,0) circle(0.05);

\draw[fill=black] (0,1.8) circle(0.05);
\node at (-1,1.8) {$(d-1)\pi$};

\draw[line width=5pt, lightgray] (3.5,0) -- (3.5,1.8);
\draw[line width=5pt, gray] (3.5,1.8) -- (3.5,4);
\draw[dashed] (0,1.8) -- (5.95,1.8);
\draw[fill=white] (3.5,1.8) circle(0.15);
%	
%\draw[dashed] (2.3,2.5) -- (0,-1);	
%\draw[fill=black] (0,2.5) circle(0.05);

\draw[fill=black] (3.5,0) circle(0.05);
\node at (3.5,-0.4) {$\frac{d+1}{2}$};

\draw[fill=black] (1.5,0) circle(0.05);
\node at (1.5,-0.4) {$\frac{d}{2}$};

%\draw[-,line width = 0.05mm] (3.65,0) -- (3.65,4);
%\draw[-,line width = 0.05mm] (3.35,0) -- (3.35,4);

\node at (6.2,-0.3) {$\beta$};
\node at (-0.2,4.2) {$\alpha$};

\node at (-0.2,-0.3) {$0$};

\node at (4.7,3.3) {$\textup{\color{white}{polytope}}$};
\end{scope}
	\end{tikzpicture}
%\end{center}
\includegraphics[width=0.45\columnwidth]{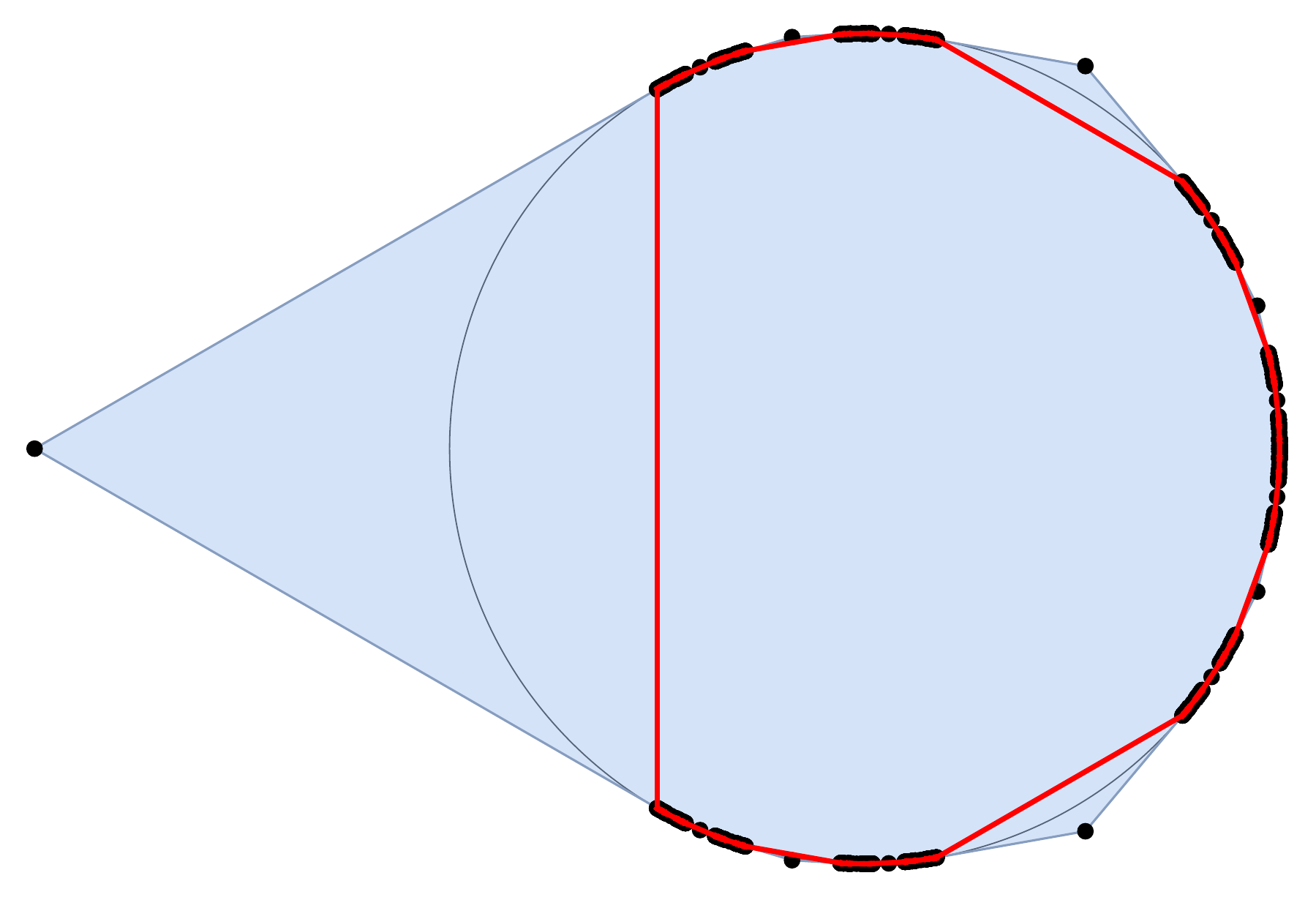}
\caption{
Left panel: Phase diagram illustrating Theorem~\ref{thm:PolytopeOrNot} for fixed dimension $d\ge 2$. Darkgray region: the beta$^*$ set $P_{d,\alpha,\beta}$ is a.s.~a polytope.  Lightgray region: $P_{d,\alpha,\beta}$ is (with positive probability) not a polytope. Hatched lightgray: the beta$^*$ set is not defined. The white point with $\beta=(d+1)/2$ and $\alpha=(d-1)\pi$ represents the doubly critical case which remains an open problem except for $d=2$.
Right panel: A convex body (in blue) touching the unit circle at infinitely many points that form a Cantor set on the circle. The convex dual of this body is shown in red. The edges of the convex dual correspond to the arcs which are removed during the construction of the Cantor set. For $d=2$, the beta$^*$ set is conjectured to have a similar structure if it is not a polytope.
}
\label{fig:regime_and_cantor}
\end{figure}

%\begin{figure}[t]
%	\centering
%	\includegraphics[width=0.48\columnwidth]{Cantor_Polytope_Dual.pdf}
%	\caption{{\color{red}}}
%	\label{fig:cantor_sphere}
%\end{figure}

\begin{theorem}[Existence of beta$^*$ polytopes]\label{thm:PolytopeOrNot}
Fix a dimension $d\geq 2$.
\begin{itemize}
\item[(i)] If $\beta>(d+1)/2$ and $\alpha>0$ is arbitrary, or if $\beta=(d+1)/2$ and $\alpha > (d-1)\pi$, then $P_{d,\alpha,\beta}$ is a.s.\ a $d$-dimensional polytope with $R>1$.
\item[(ii)] If $d/2 < \beta < (d+1)/2$ and $\alpha>0$ is arbitrary, or if  $\beta=(d+1)/2$ and $0 < \alpha < (d-1)\pi$, then
    $$
    \PP[\text{$P_{d,\alpha,\beta}$ is not a polytope and $R=1$}] >0.
    $$
\end{itemize}
\end{theorem}
Observe that the probability in case (ii) is not equal to one. Indeed, on an event of positive probability the Poisson process $\zeta_{d,\alpha,\beta}$ has at least one atom in a sufficiently small ball around each vertex of the cube $[-2,2]^d$, which guarantees that $R>1$ and, as a consequence,  $P_{d,\alpha,\beta}$ is a polytope.
It is an interesting problem to investigate the (probably fractal) structure of  the boundary of a beta$^*$-set when it touches the unit sphere in the case (ii),   $R=1$.
The doubly critical case when $\beta=(d+1)/2$ and $\alpha = (d-1)\pi$ remains open except when $d=2$, where we shall show the following result.  The parameter regimes are illustrated in Figure~\ref{fig:regime_and_cantor}, left panel.

\begin{theorem}[Existence of beta$^*$ polygons in the doubly critical case]\label{thm:PolytopeOrNot_critical}
Almost surely, $P_{2,\pi,3/2}$ is a polygon with $R>1$.
\end{theorem}

\subsection{The \texorpdfstring{$T$}{T}-functional}

In this section we start the investigation of geometric parameters of beta$^*$ polytopes and more general beta$^*$ sets $P_{d,\alpha,\beta}$ for $\alpha>0$ and $\beta>d/2$,  as introduced in the previous section. For that purpose we extend the definition of the so-called $T$-functional for polytopes, as introduced by Wieacker \cite{Wieacker}, to general convex sets. Namely, consider a non-empty closed convex set $K\subset\RR^d$, $d\geq 2$, and let $H$ be a supporting hyperplane of $K$. We call the intersection $H\cap K$ a $k$-face of $K$, $k\in\{0,1,\ldots,d-1\}$, provided that the affine hull of  $H\cap K$ is a $k$-dimensional affine subspace of $\RR^d$ (in the literature such faces are known as exposed faces, see \cite[Section 2.1, page 75]{SchneiderBook}, but since this is the only type of faces we consider we do not follow this terminology). By $\cF_k(K)$ we denote the (possibly empty) set of all $k$-faces of $K$. We remark that the (exposed) $(d-1)$-skeleton $\bigcup_{F\in\cF_{d-1}(K)}$ of $K$ may in general be a proper subset of the boundary of $K$, whereas equality holds for $d$-dimensional polytopes, for example. Now, for  parameters $a,b\geq 0$ we define the $T$-functional of $K$ as
$$
T_{a,b}(K) := \sum_{F\in\cF_{d-1}(K)}\dist(F)^a\,V_{d-1}(F)^b,
$$
where $\dist(F):=\dist(\rm aff(F))$ denotes the distance from the origin to the affine hyperplane spanned by $F$, and $V_{d-1}(F)$ is the $(d-1)$-dimensional volume (Hausdorff measure) of $F$.
In the special case $d=1$ and for intervals $K=[m,M]$ with $m < M$ we define $T_{a,b}(K) :=|m|^a+|M|^a$ for $a\geq 0$ (the functional does not depend on $b$).
The next theorem provides a formula for the expected $T$-functional of a beta$^*$ set.

\begin{theorem}[Expected $T$-functional] \label{thm:Tfunctional}
Let $d\geq 2$ and $a,b\geq 0$, $\alpha>0$, $\beta > d/2$ be parameters satisfying the following constraints:
\begin{itemize}
	\item if $\beta=(d+1)/2$ then $0\leq b<1$, $0\leq a<2d-(d-1)(b+1)-1$ and $\alpha>\pi(d-1)(1-b)$,
	\item if $\beta>(d+1)/2$ then $0\leq b<2\beta-d$, $0\leq a<d(2\beta-d+1)-(d-1)(b+1)-1$.
\end{itemize}
Only in these cases $\EE T_{a,b}(P_{d,\alpha,\beta})<\infty$ and its value is given by
$$
\EE T_{a,b}(P_{d,\alpha,\beta}) = {(\tilde{c}_{d,\beta}\alpha)^d\over d}{\omega_d}S_{d,\beta}(b)\int\limits_1^\infty p_{d,\alpha,\beta}(h)\,h^a\,(h^2-1)^{{(d-1)(b+1)\over 2}-d(\beta-{d-1\over 2})}\,\dint h,
$$
where the function $p_{d,\alpha,\beta}(h)$, $h>1$, and the constant $S_{d,\beta}(b)$ are given by
\begin{align}
	p_{d,\alpha,\beta}(h) &:= \exp\Bigg\{-\alpha \tilde{c}_{1,{\beta-{d-1\over 2}}}\int\limits_h^\infty(r^2-1)^{-\beta+{d-1\over 2}}\dint r\Bigg\},\notag\\
	S_{d,\beta}(b) &:= {\tilde{c}_{d-1,\beta}^{-d}\over ((d-1)!)^{b+1}}{\Gamma(d(\beta-{d-1\over 2})-{(d-1)(b+1)\over 2})\over\Gamma(d(\beta-{d+b\over 2}))}\bigg({\Gamma(\beta-{d+b\over 2})\over\Gamma(\beta-{d-1\over 2})}\bigg)^{d}\prod_{i=1}^{d-1}{\Gamma({i+b+1\over 2})\over\Gamma({i\over 2})}.\label{eq:def_Sdbeta}
\end{align}
Let $d=1$ and $a>0$,  $\alpha>0$, $\beta>1$. 	Then, $\EE T_{a,b}(P_{1,\alpha,\beta})<\infty$ if and only if $a<2\beta-1$ and its value is given by the same formula as above (with $S_{1,\beta}(b) = 1$).
%	Only in this case  and its value is given by
%	$$
%	\EE T_{a}(P_{1,\alpha,\beta}) = a2^a\int\limits_1^\infty %h^{a-1}\Big(1-\exp\Big\{-\alpha\tilde{c}_{1,\beta}\int\limits_{h}^{\infty}{\dint r\over(r^2-1)^\beta}\Big\}\Big)\,\dint h.
%	$$
\end{theorem}

Let us discuss some special cases of this result.  By choosing $a=b=0$, Theorem~\ref{thm:Tfunctional} yields a formula for the expected number $\EE T_{0,0}(P_{d,\alpha,\beta})=\EE f_{d-1}(P_{d,\alpha,\beta})$ of facets (i.e.\ $(d-1)$-faces) of a beta$^*$ set $P_{d,\alpha,\beta}$. In particular $\EE f_{d-1}(P_{d,\alpha,\beta})$ is finite for all $\alpha>0$ if $\beta>(d+1)/2$, while for $\beta=(d+1)/ 2$ this is the case if  and only if $\alpha>\pi(d-1)$. Note that in the doubly critical case $\alpha=\pi(d-1)$  and $\beta=(d+1)/2$, we have that $\EE f_{d-1}(P_{d,\alpha,\beta})=\infty$ although, at least for $d=2$, $P_{d,\alpha,\beta}$ is almost surely a polygon according to  Theorem~\ref{thm:PolytopeOrNot_critical}.

Choosing $a=0$ and $b=1$ in Theorem~\ref{thm:Tfunctional} an explicit formula for the expected surface area of beta$^*$ polytopes $P_{d,\alpha,\beta}$ can be derived, provided $\alpha>0$ and $\beta>(d+1)/2$. Moreover, since in these cases $P_{d,\alpha,\beta}$ almost surely contains the unit ball according to Theorem \ref{thm:PolytopeOrNot} and in particular the origin of $\RR^d$ in its interior, decomposing $P_{d,\alpha,\beta}$ into pyramids $\conv(F\cup\{0\})$ spanned by the facets $F\in\cF_{d-1}(P_{d,\alpha,\beta})$ of $P_{d,\alpha,\beta}$ meeting at the origin, one can use the base-times-height-formula for the volume of such pyramids to express the expected volume of $P_{d,\alpha,\beta}$ as ${1\over d}\EE T_{1,1}(P_{d,\alpha,\beta})$.  In addition, we observe that this term is a lower bound of the volume of $P_{d,\alpha,\beta}$, even in the cases where the beta$^*$ set is not a polytope.

More generally,  all expected intrinsic volumes $\EE V_k(P_{d,\alpha,\beta})$, $k\in\{1,\ldots,d\}$, can be expressed by means of the $T$-functional as follows (note that the case $k=0$ is trivial in this context, since $V_0(P_{d,\alpha,\beta})=1$ almost surely).

\begin{proposition}[Expected intrinsic volumes]\label{prop:IntrinsicVolume}
Suppose that $d\geq 2$, $\beta>(d+1)/ 2$ and $\alpha>0$. Then, for $k\in\{1,\ldots,d\}$ it holds that $\EE V_k(P_{d,\alpha,\beta})<\infty$ and
\begin{equation}\label{eq:14-2-22}
\EE V_k(P_{d,\alpha,\beta}) = {1\over d}{d\choose k}{\kappa_d\over\kappa_{k}\kappa_{d-k}}\,\EE T_{1,1}(P_{k,\alpha,\beta-{d-k\over 2}}).
\end{equation}
\end{proposition}

\begin{remark}
The proof of Proposition \ref{prop:IntrinsicVolume} together with the argument using the base-times-height-formula above shows that for $d\geq 2$ the right-hand side in \eqref{eq:14-2-22} is a lower bound for $\EE V_k(P_{d,\alpha,\beta})$ for all parameters $\alpha>0$, $\beta>d/2$ and $k\in\{1,\ldots,d\}$. If $\beta\leq{d+1\over 2}$ and $k\geq 2$, then $\beta':=\beta-{d-k\over 2}\in({k\over 2},{k+1\over 2}]$ and we have from Theorem \ref{thm:Tfunctional} (case $d\geq 2$) that $ \EE T_{1,1}(P_{k,\alpha,\beta'})=\infty$. If, on the other hand, $\beta\leq{d+1\over 2}$ and $k=1$, the same conclusion follows from Theorem \ref{thm:Tfunctional} (case $d=1$). In particular, this shows that for all $\beta\in (\frac{d}{2},\frac{d+1}{2}]$ the expected intrinsic volumes $\EE V_k(P_{d,\alpha,\beta})$, $k\in\{1,\dots,d\}$, are infinite.
\end{remark}

\subsection{Expected \texorpdfstring{$f$}{f}-vector}
Our next target is an explicit formula for the number of $k$-dimensional faces of $P_{d,\alpha,\beta}$ for all $k\in\{0,1,\ldots,d-1\}$. To state it, we need to introduce some further notation.

\medskip
\noindent
\textit{Angles of polytopes.}
For a $d$-dimensional polytope $P$,  a face $F\in\cF_k(P)$ and a point $x$ in the relative interior of $F$ let $N_F(P)$ be the cone of normal vectors of $P$ at $x$ (note that $N_F(P)$ is independent of the precise choice of $x$). The solid angle of $N_F(P)$ is called the \textit{external angle} $\gamma(F,P)$ of $P$ at its face $F$. Next, define the tangent cone $T_F(P)$ as the positive hull of $P-x$ (again, this definition is independent of the choice of $x$).
The \textit{internal angle} $\beta(F,P)$ of $P$ at its face $F$ is the solid angle of the cone $T_F(P)$. Solid angles are normalized such that the full-space angle is $1$. We refer to~\cite[Chapter 14]{Grunbaum} for further details.

\medskip
\noindent
\textit{Internal angle sums of beta' simplices.}
To be consistent with the notation used in the works on beta and beta' polytopes~\cite{KabluchkoThaeleZaporozhets,KabluchkoAngles}, let $\tilde J_{m,\ell}(\beta)$ denote the expected internal angle at some $(\ell-1)$-dimensional face of the simplex $\conv(Z_1,\dots,Z_m)\subset\RR^{m-1}$, where $Z_1,\dots,Z_m$ are $m$ independent random points in $\RR^{m-1}$ each having the so-called beta' probability density
\begin{equation}\label{eq:BetaprimeDensity}
\tilde f_{m-1,\beta}(x):= \tilde c_{m-1,\beta}\big(1+\|x\|^2\big)^{-\beta},\qquad x\in\RR^{m-1}.
\end{equation}
That is,
\begin{align}\label{eq:def_J_tilde}
\tilde J_{m,\ell}(\beta)
&	=\EE\beta(\conv(Z_1,\dots,Z_\ell),\conv(Z_1,\dots,Z_m)).
\end{align}
%\TG{Inconvenient that $\beta$ is in the argument of $\tilde J_{m,\ell}$ and also the notation for internal angle. Can we work around that? Do we need to? ZK: Yes, it is inconvenient... We could use \measuredangle T_{face}(polytope), but this is difficult.}
By exchangeability, the expected sum of internal angles at all $(\ell-1)$-dimensional faces of the beta' simplex $\conv(Z_1,\dots,Z_m)$ is then
\begin{align*}
\tilde{\mathbb J}_{m,\ell}(\beta):=\binom m\ell \tilde J_{m,\ell}(\beta).
\end{align*}
The admissible values of the  parameters are $m\in\NN$, $\ell\in\{1,\dots,m\}$ and $\beta >(m-1)/2$. By convention, $\tilde{\mathbb J}_{1,1}(\beta) := 1$ and $\tilde{\mathbb J}_{m,\ell}(\beta):=0$ for $m<\ell$.

\medskip
\noindent
\textit{External angle sums of beta$^*$ polytopes.}
For $m\in\NN$ and $\lambda>1, \alpha>0$ or $\lambda=1,\alpha > (m-1)\pi$ we introduce the quantities
\begin{align}\label{eq:def_I_star}
I^*_{\alpha,m}(\lambda)
=
\int\limits_1^\infty \tilde c_{1,\frac{\lambda m+1}{2}}(y^2-1)^{-\frac{\lambda m+1}{2}} \exp\Bigg\{-\alpha\int\limits_{y}^\infty\tilde c_{1,\frac{\lambda+1}{2}}(t^2-1)^{-\frac{\lambda+1}{2}}\,\dint t\Bigg\}\,\dint y
\end{align}
and
\begin{align}\label{eq:def_I_star_sum}
\mathbb I^*_{\alpha,m}(\lambda)
	=\frac{\alpha^m}{m!}\, I^*_{\alpha,m}(\lambda).
\end{align}
The latter term $\mathbb I^*_{\alpha,m}(\lambda)$ turns out to be the expected sum of external angles at all $(m-1)$-dimensional faces of the beta$^*$ polytope $P_{d,\alpha,\frac{\lambda+d}2}$ in $\RR^d$, as we will see in Theorem~\ref{thm:exp_ext_anlge_sums} below.

\medskip
The next result provides an explicit formula for the expected $f$-vector of $P_{d,\alpha,\beta}$ in terms of $\mathbb I^*_{\alpha,m}(\lambda)$ and $\tilde{\mathbb J}_{m,\ell}(\beta)$.

\begin{theorem}[Expected $f$-vector]\label{thm:f-vector_beta_star}
Suppose that either $\beta>(d+1)/ 2$ and $\alpha>0$ or $\beta=(d+1)/2$ and $\alpha > (d-1)\pi$. Then, for $k\in\{0,\dots,d-1\}$, it holds that
\begin{align*}%\label{eq:exp_f-vector}
\EE f_k(P_ {d,\alpha,\beta})=2\sum_{s=0}^{\lfloor (d-k-1)/2\rfloor}   \mathbb I^*_{\alpha,d-2s}(2\beta-d)\cdot\tilde {\mathbb J}_{d-2s,k+1}\Big(\beta-s-\frac 12\Big).
\end{align*}
\end{theorem}

Theorem~\ref{thm:f-vector_beta_star} can be seen as the analogue for beta$^*$ polytopes of the corresponding results for beta polytopes~\cite[Theorem 1.2]{KabluchkoThaeleZaporozhets} and beta' polytopes~\cite[Theorem 1.14]{KabluchkoThaeleZaporozhets}.

Our next result provides an interpretation of $\mathbb I^*_{\alpha,m}(\lambda)$ as an expected external angle sum and is the analogue for beta$^*$ polytopes of~\cite[Theorem 1.13 and Theorem 1.16]{KabluchkoThaeleZaporozhets}.

\begin{theorem}[Expected external angle sums]\label{thm:exp_ext_anlge_sums}
Suppose that either $\beta>(d+1)/2$ and $\alpha>0$, or $\beta=(d+1)/2$ and $\alpha > (d-1)\pi$. Then, for $k\in\{0,\dots,d-1\}$ it holds that
\begin{align*}%\label{eq:exp_external_anglesums}
\EE\Bigg[\sum_{G\in\cF_k(P_{d,\alpha,\beta})}\gamma(G,P_{d,\alpha,\beta})\Bigg]=\mathbb I^*_{\alpha,k+1}(2\beta-d).
\end{align*}
\end{theorem}

\begin{remark}
We mention two alternative expressions for $I^*_{\alpha,m}(\lambda)$. They are obtained from~\eqref{eq:def_I_star} by the change of variables $y=\cosh \varphi$, $t=\cosh\theta$ with $\varphi,\theta\in (0,\infty)$, and, respectively, $y=\coth \varphi$, $t=\coth\theta$ with $\varphi,\theta\in (0,\infty)$:
\begin{align}
I^*_{\alpha,m}(\lambda)
&	=\int\limits_0^\infty \tilde c_{1,\frac{\lambda m+1}{2}} (\sinh \varphi)^{-\lambda n}\,\exp\Bigg\{-\alpha\int\limits_\varphi^\infty \tilde c_{1,\frac{\lambda+1}{2}}(\sinh \theta)^{-\lambda}\,\dint \theta\Bigg\}\,\dint \varphi\label{eq:I*_alt_expression_sinh_negative_powers}\\
&	=\int\limits_0^\infty\tilde c_{1,\frac{\lambda m+1}{2}} (\sinh \varphi)^{\lambda m-1}\,\exp\Bigg\{-\alpha\int\limits_0^\varphi \tilde c_{1,\frac{\lambda+1}{2}}(\sinh \theta)^{\lambda-1}\,\dint \theta\Bigg\}\,\dint \varphi.\notag%\label{eq:I*_alt_expression_sinh}
\end{align}
We also mention that the quantity $\tilde {\mathbb J}_{m,\ell}(\beta)$ has been explicitly computed in~\cite{KabluchkoAngles}. In particular, from that paper it is known that
\begin{align} \label{eq:internal_angle_sum_J_formula}
\tilde {\mathbb J}_{m,\ell}\bigg(\frac{\lambda+m-1}{2}\bigg)
=
\binom m\ell
\int\limits_{-\pi/2}^{\pi/2}\tilde c_{1,\frac{\lambda m}{2}}(\cos\varphi)^{\lambda m-2}\left(\:\int\limits_{-\infty}^{\textup i\varphi}\tilde c_{1,\frac{\lambda+1}{2}}(\cosh \theta)^{-\lambda} \,\dint \theta\right)^{m-\ell}\dint\varphi,
\end{align}
for all $\lambda>0$, $m\in\NN$ and $\ell\in\{1,\dots,m\}$ such that $\lambda m>1$, and where $\textup{i}=\sqrt{-1}$ stands for the imaginary unit.
This makes the formula in Theorem~\ref{thm:f-vector_beta_star} for the expected $f$-vector of a beta$^*$ polytope fully explicit.
\end{remark}

\section{Applications and special cases}\label{sec:SpecialCases}

In this section we come back to the two hyperbolic stochastic geometry models discussed earlier. We show how they are related to beta$^*$ polytopes and then discuss several consequences.

\subsection{Basics from hyperbolic geometry}\label{subsec:HyperbolicGeo}
The hyperbolic space can be defined as a simply connected, geodesically complete $d$-dimensional Riemannian manifold of constant negative curvature $\kappa=-1$. In the following, we shall briefly describe three classical realizations of the hyperbolic space we shall work with.  We refer the reader to~\cite{CannonEtAl,AlekseevskijVinbergSolodovnikov,LoustauBook,Ratcliffe} for further background material on hyperbolic geometry.

\medskip
\noindent
\textit{The hyperboloid model.} Fix some dimension $d\in \NN$ and consider the space $\RR^{d+1}$. The \textit{Minkowski product} of two vectors $x= (x_0,x_1,\ldots, x_d)$ and $y = (y_0,y_1,\ldots,y_d)$ in $\RR^{d+1}$ is defined by
$$
B(x,y) := x_0 y_0 - x_1 y_1 - \ldots - x_d y_d.
$$
The hyperboloid model for a $d$-dimensional hyperbolic space is defined on the \textit{upper hyperboloid}
$$
\HH^d:= \{(x_0,x_1,\ldots,x_d)\in\RRd:x_0^2- x_1^2 - \ldots - x_d^2 = 1, x_0>0\}.
$$
The restriction of $B(x,y)$ to any tangent space of $\HH^d$ turns out to be positive definite and defines a Riemannian metric on $\HH^d$ with constant curvature $\kappa= -1$. The Riemannian metric, in turn, determines a volume measure $\nu_{\text{hyp}}$ on $\HH^d$.
The geodesic distance $d_{\text{hyp}} (x,y)$ between any two points $x\in\HH^d$ and $y\in\HH^d$ is given by
\begin{equation}\label{eq:HyperbolicDistance}
d_{\text{hyp}} (x,y) := \arccosh B(x,y).
\end{equation}
The hyperbolic hyperplanes (i.e.\ totally geodesic subspaces of $\HH^d$ of dimension $d-1$) can be represented as $\HH^d\cap H$ where $H\subset\RR^{d+1}$ is a $d$-dimensional linear subspace (especially, $H$ passes through the origin $0$ of $\RR^{d+1}$).

Any point $x\in\HH^d$ can be parametrized as $x= (\cosh\theta,u\sinh\theta)$, where $u\in\SS^{d-1}\subset\RR^d$ and $\theta\geq 0$ is the hyperbolic distance between $x$ and the \textit{apex} $e := (1,0,\ldots,0)$ of the hyperboloid $\HH^d$. We observe that the apex  can be parametrized as $e=(\cosh 0,u\sinh 0)$ for any $u\in\SS^{d-1}$, while the parametrization of any other point $x\in \HH^d$ is unique.

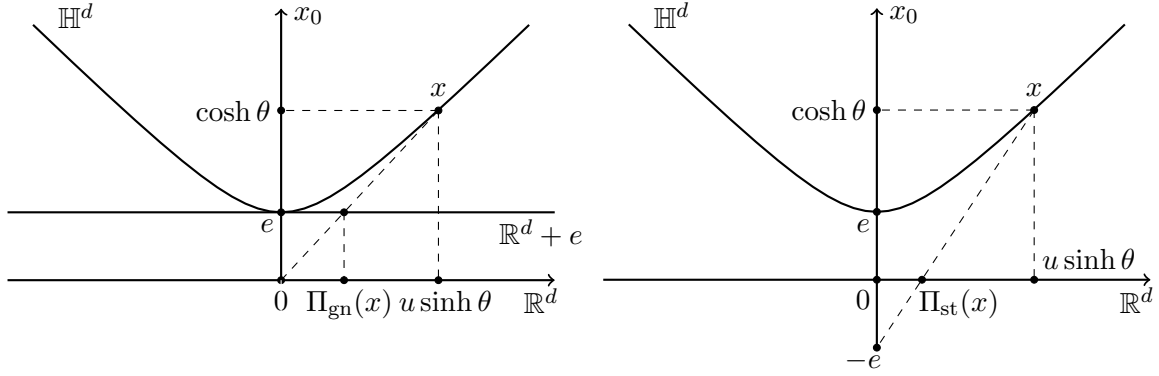
\begin{figure}[t]
		\begin{center}
	\begin{tikzpicture}[scale=0.9]
	\begin{scope}
	\draw [thick,domain=-2:2] plot ({sinh(\x)}, {cosh(\x)});
	\draw[thick] (-4,1) -- (4,1);
	\draw[->,thick] (-4,0) -- (4,0);
	\draw[->,thick] (0,0) -- (0,4);
	\draw[fill=black] (0,0) circle(0.05);
	\draw[fill=black] (0,1) circle(0.05);
	\draw[fill=black] (2.3,2.5) circle(0.05);
	\draw[dashed] (2.3,2.5) -- (0,2.5);
	\draw[dashed] (2.3,2.5) -- (2.3,0);	
	\draw[dashed] (2.3,2.5) -- (0,0);	
	\draw[fill=black] (0,2.5) circle(0.05);
	\draw[fill=black] (2.3,0) circle(0.05);
	\draw[fill=black] (0.92,1) circle(0.05);
	\draw[fill=black] (0.92,0) circle(0.05);
	\draw[dashed] (0.92,1) -- (0.92,0);	
	\node at (0.4,3.9) {$x_0$};
	\node at (-3,3.9) {$\HH^d$};
	\node at (3.8,-0.3) {$\RR^d$};
	\node at (3.8,-0.3+1) {$\RR^d+e$};
	\node at (-0.2,0.8) {$e$};
	\node at (-0,0.7-1) {$0$};
	\node at (1,0.65-1) {$\Pi_{\text{gn}}(x)$};
	\node at (1+1+0.4,0.7-1) {$u\sinh\theta$};
	\node at (1+1+0.3,0.7-1+3.1) {$x$};
	\node at (-0.7,0.7-1+2.8) {$\cosh\theta$};
	\draw[white] (0,0) -- (0,-1.47);
	\end{scope}
	\end{tikzpicture}
		\begin{tikzpicture}[scale=0.9]
	\begin{scope}
\draw [thick,domain=-2:2] plot ({sinh(\x)}, {cosh(\x)});
%\draw[thick] (-4,1) -- (4,1);
\draw[->,thick] (-4,0) -- (4,0);
\draw[->,thick] (0,-1) -- (0,4);
\draw[fill=black] (0,0) circle(0.05);
\draw[fill=black] (0,1) circle(0.05);
\draw[fill=black] (0,-1) circle(0.05);
\draw[fill=black] (2.3,2.5) circle(0.05);
\draw[dashed] (2.3,2.5) -- (0,2.5);
\draw[dashed] (2.3,2.5) -- (2.3,0);	
\draw[dashed] (2.3,2.5) -- (0,-1);	
\draw[fill=black] (0,2.5) circle(0.05);
\draw[fill=black] (2.3,0) circle(0.05);
\draw[fill=black] (0.66,0) circle(0.05);
%\draw[fill=black] (0.92,1) circle(0.05);
%\draw[fill=black] (0.92,0) circle(0.05);
%\draw[dashed] (0.92,1) -- (0.92,0);	
\node at (0.4,3.9) {$x_0$};
\node at (-3,3.9) {$\HH^d$};
\node at (3.8,-0.3) {$\RR^d$};
%\node at (3.8,-0.3+1) {$\RR^d+e$};
\node at (-0.2,0.8) {$e$};
\node at (-0.2,-1.2) {$-e$};
\node at (-0.2,0.7-1) {$0$};
\node at (1+0.2,0.65-1) {$\Pi_{\text{st}}(x)$};
\node at (1+1+0.8+0.3,0.7-0.6+0.2) {$u\sinh\theta$};
\node at (1+1+0.3,0.7-1+3.1) {$x$};
\node at (-0.7,0.7-1+2.8) {$\cosh\theta$};
\end{scope}
	\end{tikzpicture}
\end{center}
\caption{Illustration of the gnomonic projection which identifies the hyperboloid model with the Klein model (left), and the stereographic projection (right) which identifies the hyperboloid model with the Poincar\'e model of hyperbolic space.}
\label{fig:Projection}
\end{figure}

\medskip
\noindent
\textit{The Klein model.} The hyperboloid model on $\HH^d$ is connected to the Klein model on the open unit ball $\BB^d = \{x\in \RR^d: \|x\| < 1\}$ via the \textit{gnomonic projection} $\Pi_{\text{gn}}: \HH^d \to \BB^d$; see the left panel of Figure~\ref{fig:Projection}. To define this projection, consider a segment joining $0$ and some point $x\in \HH^d$. This segment intersects the hyperplane $\{x_0=1\}$ at some point denoted by $(1,v)$ with $v\in \BB^d$. Then the gnomonic projection of $x$ is defined as $\Pi_{\text{gn}}(x) := v$. For a point $x\in\HH^d$ with parametrization $(\cosh\theta,u\sinh\theta)$ the gnomonic projection is given by
\begin{equation}\label{eq:GnomonicProjection}
\Pi_{\text{gn}}(x) := u\tanh\theta\in\BB^d\subset\RR^d.
\end{equation}
In the Klein model, the hyperbolic hyperplanes are represented by (non-empty) intersections of usual affine hyperplanes in $\RR^d$ with $\BB^d$. The Riemannian metric $d_{\text{Kl}}$ (i.e.\ the distance between a point $v\in \BB^d$ and an infinitesimally close point $v+\dint v$) and the volume element $\nu_{\text{Kl}}$ in the Klein model are given by
\begin{equation}\label{eq:metric_volume_klein}
d_{\text{Kl}}^2(v, v+\dint v) =  \frac{\|\dint v\|^2}{1 - \| v \|^2} +  \frac{\langle v, \dint v \rangle^2}{(1 - \| v \|^2)^2},
\qquad
\nu_{\text{Kl}}(\dint v_1, \ldots, \dint v_d) = \frac{\dint v_1\ldots \dint v_d}{(1-\|v\|^2)^{\frac{d+1}2}}.
%\qquad
%\|x\| <1.
\end{equation}
The distance between $v\in \BB^d$ and the origin $0$ is given by
\begin{equation}\label{eq:dist_0_x_klein}
d_{\text{Kl}} (0, v) = \frac 12 \log \frac{1+\|v\|}{1 - \|v\|} = \artanh \|v\|.
%\qquad
%\|v\| <1.
\end{equation}

\medskip
\noindent
\textit{The Poincar\'e model.}
To pass from the hyperboloid model to the Poincar\'e model which is defined on the unit ball $\BB^d$ (as is the Klein model), we need the \textit{stereographic projection} $\Pi_{\text{st}}:\HH^d\to \BB^d$; see the right panel of Figure~\ref{fig:Projection}. To define the projection, consider a segment joining $-e = (-1,0,\ldots,0)$ and some point $x\in \HH^d$. This segment intersects the hyperplane $\{x_0=0\}$ at some point denoted by $(0,w)$ with $w\in \BB^d$. Then the stereographic  projection of $x$ is defined as $\Pi_{\text{gn}}(x) := w$. For a point $x\in\HH^d$ with parametrization $(\cosh\theta,u\sinh\theta)$ we have
\begin{equation}\label{eq:StereographicProjection}
\Pi_{\text{st}}(x) := u\tanh{\theta\over 2}\in\BB^d.
\end{equation}
In the Poincar\'e model, the hyperbolic hyperplanes have the form $S\cap \BB^d$, where $S$ is a $(d-1)$-dimensional sphere in $\RR^d$ intersecting $\SS^{d-1}$ orthogonally or a hyperplane in $\RR^d$ passing through the origin.
The Riemannian metric and the volume element in the Poincar\'e model are given by
\begin{equation}\label{eq:metric_volume_poincare}
d_{\text{Poi}}^2(w, w+\dint w) =  \frac{4 \|\dint w\|^2}{(1 - \| w \|^2)^2},
\qquad
\nu_{\text{Poi}}(\dint w_1,\ldots, \dint w_d) = \frac{2^d \dint w_1\ldots \dint w_d}{(1-\|w\|^2)^{d}}.
\end{equation}
The distance between $0$ and $w\in \BB^d$ is given by
\begin{equation}\label{eq:dist_0_x_poincare}
d_{\text{Poi}} (0, w) = \log \frac{1+\|w\|}{1 - \|w\|} = 2 \artanh \|w\| = 2 d_{\text{Kl}} (0, w).
\end{equation}
This relation between the distances in the Klein and Poincar\'e models will be crucial in our treatment of the typical hyperbolic Voronoi cell.

\medskip
\noindent
\textit{Isomorphism between the models.}
Let us finally describe the map $\varphi_{\text{Poi$\to$Kl}}:\BB^d\to \BB^d$ between the Poincar\'e and the Klein model.  If some point $x\in \HH^d$ is represented by $w = \Pi_{\text{st}}(x) \in \BB^d$ in the Poincar\'e model, then the corresponding point in the Klein model is given by
$$
\varphi_{\text{Poi$\to$Kl}} (w) = \frac{2w}{1+\|w\|^2} = \Pi_{\text{gn}}(x).
$$

\subsection{Poisson hyperplane tessellations in the hyperbolic space}\label{sec:Hyperplane}
We are now going to describe the connection between beta$^*$ polytopes and certain hyperplane tessellations of the hyperbolic space. It is known from~\cite{santalo_book} that there is a unique (up to a multiplicative constant $\lambda>0$) infinite  measure on the space of hyperbolic hyperplanes that is invariant under the isometries of the $d$-dimensional hyperbolic space. Consider a Poisson process on the space of hyperbolic hyperplanes whose intensity is given by that measure. The atoms of this process give rise to countably many hyperplanes dissecting the hyperbolic space into random cells with disjoint interiors.   We shall be interested in the  hyperbolic Poisson zero cell, i.e.\ the a.s.\ unique cell containing some fixed point (the ``origin'').  In fact, we shall define a family of hyperbolic hyperplane tessellations indexed by two parameters $\lambda$ and $\beta$ which reduces to the isometry-invariant tessellation mentioned above in the special case $\beta= (d+1)/2$.

\subsubsection{A family of hyperbolic hyperplane tessellations}\label{subsubsec:hyperbolic_hyperplane_tess_def}
We start by defining our tessellations in the Klein model on $\BB^d$, $d\in\NN$. Fix some parameters $\beta>\max(d/2,1)$ and $\lambda>0$.  Consider  a Poisson process on $(0,1)$ with Lebesgue intensity
\begin{align}\label{eq:def_f(r)}
f(r)
%=
%\lambda \left(\frac{1}{\sqrt{1-r^2}}\right)^{d-1} \left(\frac r {\sqrt{1-r^2}}\right)^{2\beta-d-1} \frac 1{1-r^2}
:=
\lambda \frac{r^{2\beta - d - 1}}{(1-r^2)^{\beta}},
\qquad
0<r<1.
\end{align}
By our assumption on $\beta$, the function $f$ is integrable at $0$ but not at $1$. Hence, the atoms of this Poisson process can be ordered increasingly as $r_1<r_2<\ldots$ and satisfy $r_n\to 1$, as $n\to\infty$.
Independently, let $u_1,u_2,\ldots$ be i.i.d.\ points drawn uniformly at random from the unit sphere $\SS^{d-1}$. Then, in the Klein model, our hyperplane tessellation consists of the hyperplanes
$$
\{x\in \BB^d : \langle x, u_n\rangle = r_n\} = (u_n^\bot  +  r_n u_n)\cap \BB^d
\qquad
n=1,2,\ldots;
$$
see the left panel of Figure~\ref{fig:hyperplane_tess} for a realization.
%(or more precisely, by the intersection of these hyperplanes with the unit ball).
Note that $r_1,r_2,\ldots \in (0,1)$ are the \textit{Euclidean} distances from the hyperplanes to the origin, while $u_1,u_2,\ldots\in \SS^{d-1}$ are their normals.

Let us now give an equivalent definition of the tessellation which
%, although it is stated in the hyperboloid model,
is model-independent in the sense that it involves only quantities intrinsic to  hyperbolic geometry. Consider a $d$-dimensional hyperbolic space $\MM^d$ endowed with the hyperbolic metric $d_{\text{h}}(\cdot, \cdot)$. Fix an arbitrary point $p\in \MM^d$, referred to as the origin, and let $\SS_p^{d-1} = \{x\in \MM^d: d_{\text{h}}(x,p) = 1\}$ be the unit hyperbolic sphere centred at $p$ and endowed with the normalized spherical Lebesgue measure $\sigma_{d-1;p}$.
Let again $\lambda>0$ and $\beta > \max(d/2,1)$ be parameters and consider a Poisson process $\xi_{d,\lambda,\beta}$ on the space $A_{\text{h}}(d,d-1)$ of hyperbolic hyperplanes whose intensity measure is chosen to be
\begin{equation}\label{eq:intensity_hyperplane_tess}
\mu_{d,\lambda,\beta}(\,\cdot\,)
:=
\lambda \int\limits_{\SS_p^{d-1}}\int\limits_{0}^\infty
(\cosh \theta)^{d-1}  \, (\sinh \theta)^{2\beta-d-1} \,
\ind_{\{H_p(u,\theta)\in\,\cdot\,\}}
\,\dint \theta \, \sigma_{d-1;p}(\dint u),
\end{equation}
where $H_p(u,\theta)\in A_{\text{h}}(d,d-1)$ is the unique hyperbolic hyperplane orthogonal to the hyperbolic line connecting $p$ to $u$ and having hyperbolic distance $\theta>0$ to $p$.
In the special case when $\beta= (d+1)/2$, the $\sinh$-term disappears and it is known~\cite{santalo_book} (see, in particular, Equation~(17.54) on p.~309 there) that the above definition does not depend on the choice of the origin $p$ (which is not true for other values of $\beta$). Moreover,  $\mu_{d,\lambda, (d+1)/2}$ is invariant under isometries of the hyperbolic space; we shall give another proof of invariance in Remark~\ref{rem:invariance_hyperplane_measure}.

Returning to the Klein model, we choose $p=0$, write $\theta = \artanh r$ for $r\in (0,1)$ and observe that  for any unit vector $u\in\SS^{d-1}$ we have $H_0(u, \theta) = (u^\bot  + r u) \cap \BB^d$ since $d_{\text{Kl}}(0,x) = \artanh \|x\|$ by~\eqref{eq:dist_0_x_klein}. Hence, we can rewrite the formula for $\mu_{d,\lambda,\beta}$ as follows:
$$
\mu_{d,\lambda,\beta}(\,\cdot\,)
=
\lambda
\int\limits_{\SS^{d-1}} \int\limits_{0}^1
\frac{r^{2\beta - d - 1}}{(1-r^2)^{\beta}}
\,\ind_{\{ (u^\bot + r u) \cap \BB^d \in \cdot\}}
\,\dint r \, \sigma_{d-1}(\dint u),
$$
where $\sigma_{d-1}$ is the uniform distribution on $\SS^{d-1}$ and we used the identities
\begin{equation}\label{eq:cosh_artanh}
\cosh (\artanh r) = \frac 1 {\sqrt{1-r^2}},
\qquad
\sinh (\artanh r) = \frac r {\sqrt{1-r^2}},
\qquad
\dint (\artanh r) = \frac{\dint r}{1-r^2}.
%\qquad
%-1 < r < 1.
\end{equation}
It follows that the definition of $\mu_{d,\lambda, \beta}$ given in~\eqref{eq:intensity_hyperplane_tess} is equivalent to the construction we presented at the beginning of this section.

\subsubsection{Poisson zero cell and the \texorpdfstring{beta$^*$}{beta*} polytope}
We are interested in the \textit{zero cell} of the hyperbolic Poisson hyperplane tessellation with intensity measure $\mu_{d,\lambda,\beta}$, as defined in Section~\ref{subsubsec:hyperbolic_hyperplane_tess_def}, which, in the Klein model, is the random closed convex set given by
\begin{equation}\label{eq:zero_cell_def}
Z_{d,\lambda,\beta} := \{x\in \widebar\BB^d: \langle x, u_n\rangle \leq r_n \text{ for all } n\in \NN\} \subset \widebar\BB^d.
\end{equation}
As we shall see in Theorem~\ref{thm:PolytopeOrNot_Tessellation}, the zero cell need not be bounded in the sense of hyperbolic geometry, that is it may touch the unit sphere.
The next theorem relates $Z_{d,\lambda, \beta}$ to the convex dual of a beta$^*$ set. Recall that the convex dual of a set $K\subset \RR^d$ is defined as
\begin{equation}\label{eq:convex_duality_def}
K^\circ : = \{x\in \RR^d: \langle x, y\rangle \leq 1 \text{ for all }y \in K\}.
\end{equation}

\begin{figure}[t]
\centering
\includegraphics[width=0.48\columnwidth]{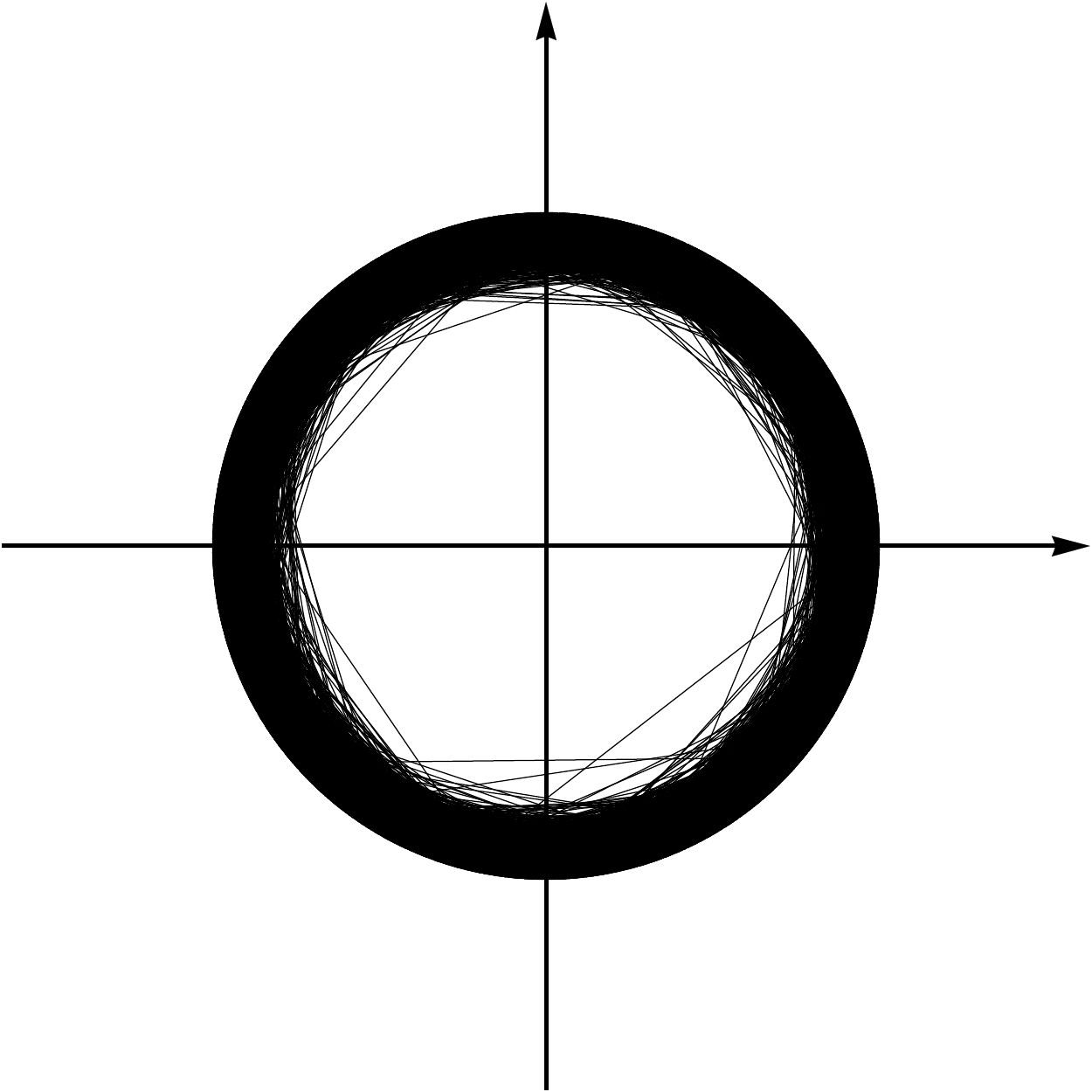}
\includegraphics[width=0.48\columnwidth]{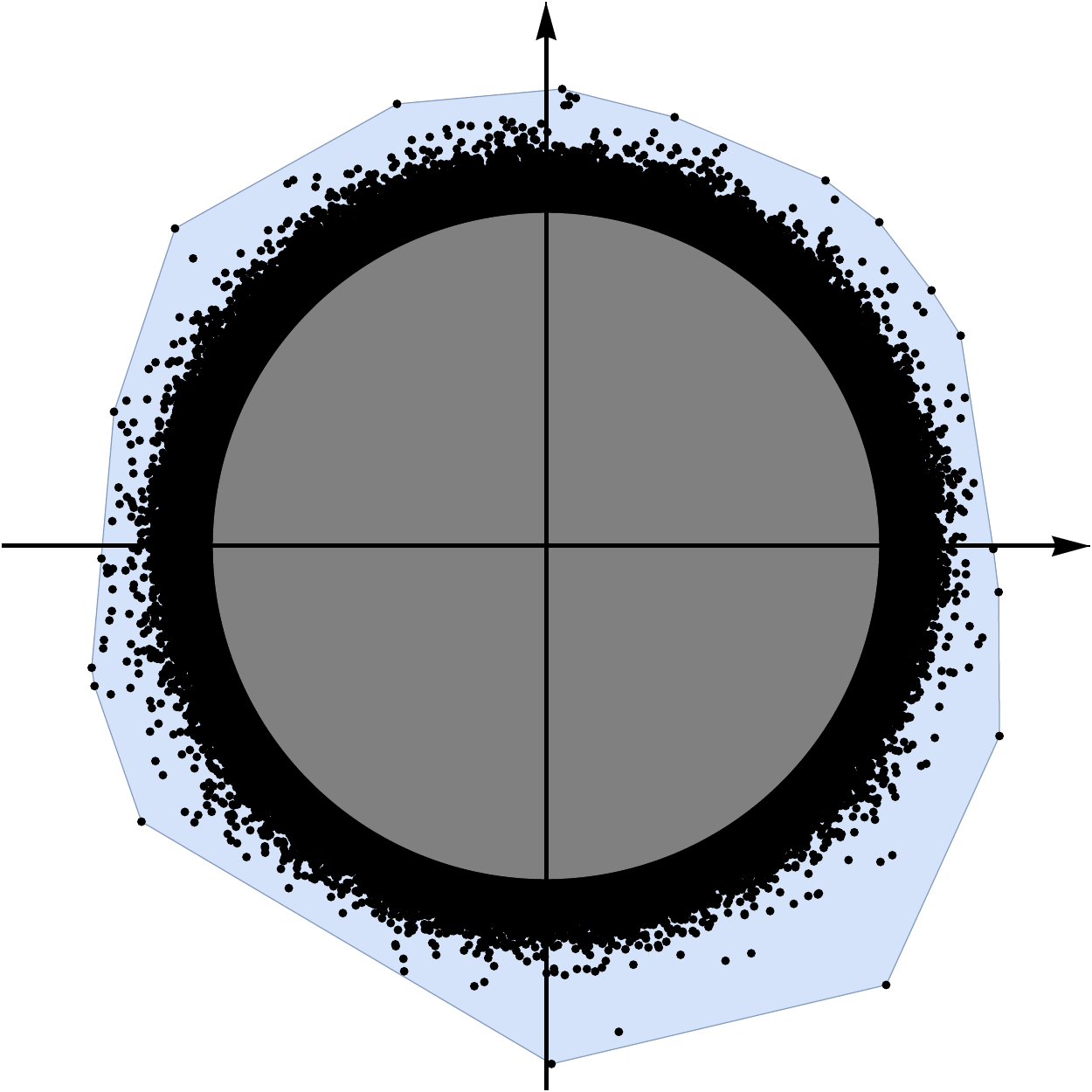}
\caption{A line tessellation in the Klein model of the hyperbolic plane  (left panel) and the convex dual of its zero cell, a beta$^*$ polytope (right panel). In this simulation we took $\beta=6$.}
\label{fig:hyperplane_tess}
\end{figure}

\begin{theorem}[Reduction of the zero cell to beta$^*$ sets]\label{thm:ZeroCell}
Let $d\in \NN$, $\lambda>0$ and $\beta>\max(d/2,1)$.
Then the zero cell $Z_{d,\lambda, \beta}$ has the same distribution as the convex dual $P_{d,\alpha,\beta}^\circ$ of the beta$^*$ set $P_{d,\alpha,\beta}$ with $\alpha = \lambda/ (\tilde c_{d,\beta}\omega_d)$.
\end{theorem}
%For the proof we need a lemma which will be useful on other occasions.
%\begin{lemma}
%%Let $d\geq 2$, $\lambda>0$ and $\beta>d/2$ and
%The points $r_1,r_2,\ldots$ form a Poisson process on $(0,1)$ with intensity $f(r)$ if and only if the points $R_1: %=1/r_1,R_2:=1/r_2,\ldots$ form a Poisson process on $(1,\infty)$ with intensity $g(R)$, where
%$$
%f(r)
%=
%\lambda \frac{r^{2\beta - d - 1}}{(1-r^2)^{\beta}},
%\qquad
%g(R)
%=
%\lambda \frac{R^{d-1}}{(R^2-1)^{\beta}},
%\qquad
%0<r<1, \;\; R>1.
%$$
%\end{lemma}
%\begin{proof}
%The claim follows from the mapping theorem for Poisson processes. Assume that $r_1,r_2,\ldots$ form a Poisson process %with intensity $f(r)$ on $(0,1)$. A point $R_n = 1/r_n$ belongs to an infinitesimal interval $[R, R+\dint R]$ with $R>1$ if and %only if $r_n$ belongs to the interval $[1/(R+\dint R), 1/R]$ whose length is $R^{-2}\dint R +  o(\dint R)$. Hence, the intensity of %the Poisson process formed by the points $R_1,R_2,\ldots$ is
%$$
%R^{-2} f(1/R)
%=
%\lambda R^{-2} \frac{R^{-2\beta + d + 1}}{(1-R^{-2})^{\beta}}
%=
%\lambda \frac{R^{d-1}}{(R^{2}-1)^{\beta}}
%\qquad
%R>1,
%$$
%and the proof is complete.
%\end{proof}
\begin{proof}%[Proof of Theorem~\ref{thm:ZeroCell}]
By definition, the beta$^*$ set $P_{d,\alpha, \beta}$ is the closed convex hull of the atoms of the Poisson process with intensity  $\alpha\,\tilde c_{d,\beta}(\|x\|^2-1)^{-\beta}$, where $\|x\|>1$. We represent these atoms as $R_1 U_1, R_2U_2,\ldots$, where $R_1 > R_2 > \ldots >1$ denote the distances from the atoms  to $0$, and $U_1,U_2,\ldots\in \SS^{d-1}$ are vectors of unit length. By the transformation property of the Poisson processes, the distances $R_1,R_2,\ldots$ form a Poisson process on $(1,\infty)$  with intensity
$$
\psi(R) := \alpha\,\tilde c_{d,\beta} \frac{\omega_d R^{d-1}}{(R^2-1)^{\beta}},
\qquad
R>1,
$$
where the term $\omega_d R^{d-1}$ comes from the polar integration formula. Moreover, it follows from the isotropy of the intensity that $U_1,U_2,\ldots$ are independent, uniformly distributed on $\SS^{d-1}$, and independent from $R_1,R_2,\ldots$

By definition of convex duality~\eqref{eq:convex_duality_def},
%the dual of the closed convex hull of a union of sets is the intersection of the duals of these sets. Hence,
the dual of the beta$^*$ set $P_{d,\alpha, \beta}$ (which is the closed convex hull of the points $R_1 U_1, R_2 U_2,\ldots$ and the open unit ball $\BB^d$) is given by
\begin{equation}\label{eq:beta_star_set_dual}
P_{d,\alpha,\beta}^\circ
=
\{x\in \widebar\BB^d: \langle x, R_n U_n\rangle \leq 1\; \forall n\in\NN\}
=
\{x\in \widebar\BB^d: \langle x, U_n\rangle \leq 1/R_n\; \forall n\in\NN\}.
\end{equation}
Again by the transformation property of Poisson processes, the inverse distances $1/R_1$, $1/R_2, \ldots$ form a Poisson process on $(0,1)$. To compute its intensity, observe that a point $1/R_n$ belongs to an infinitesimal interval $[r, r+\dint r]$ with $0<r<1$ if and only if $R_n$ belongs to the interval $[1/(r+\dint r), 1/r]$ whose length is $r^{-2}\dint r +  o(\dint r)$. Hence, the intensity of the Poisson process formed by the points $1/R_1,1/R_2,\ldots$ is
$$
g(r)
=
r^{-2} \psi(1/r)
=
\alpha\,\tilde c_{d,\beta} r^{-2} \frac{\omega_d r^{1-d}}{(r^{-2}-1)^{\beta}}
=
\alpha\,\tilde c_{d,\beta} \omega_d  \frac{ r^{2\beta - d - 1}}{(1-r^2)^{\beta}}
,
\qquad
0<r<1.
$$
If $\alpha$ and $\lambda$ are such that  $\alpha\,\tilde c_{d,\beta} \omega_d = \lambda$, then $f(r) =g(r)$ (as defined in~\eqref{eq:def_f(r)}) and the Poisson processes formed by the points $1/R_1, 1/R_2, \ldots$ and $r_1, r_2, \ldots$ have the same law. Comparing~\eqref{eq:zero_cell_def} and~\eqref{eq:beta_star_set_dual} we conclude that $Z_{d,\lambda, \beta}$ and   $P_{d,\alpha,\beta}^\circ$ have the same law.
\end{proof}
\begin{remark}
As a corollary of Theorem~\ref{thm:ZeroCell}, we record the identity  $\EE f_k(Z_{d,\lambda,\beta}) = \EE f_{d-k-1} (P_{d,\alpha, \beta})$, for all $k = 0,\ldots, d-1$. Combined with Theorem~\ref{thm:f-vector_beta_star}, it yields an explicit formula for the expected $f$-vector of $Z_{d,\lambda,\beta}$.
\end{remark}
%\begin{equation}\label{eq:duality_zero_cell_beta_star_nonstat}
%f_k(Z_{d,\lambda,\beta}) \stackrel{d}{=} f_{d-k-1} (P_{d,\alpha, \beta}),
%\qquad
%k = 0,\ldots, d-1.
%\end{equation}

The next theorem characterizes the cases when the zero cell is bounded in the hyperbolic sense (that is, when it does not touch the unit sphere).

\begin{theorem}[(Un-)boundedness of the zero cells]\label{thm:PolytopeOrNot_Tessellation}
Fix a dimension $d\geq 2$ and let $\lambda^{\text{crit}}_d := (d-1)^2 \sqrt \pi \, \Gamma (\frac{d-1}{2}) / \Gamma(\frac d2)$.
\begin{itemize}
\item[(i)] If $\beta>(d+1)/ 2$ and $\lambda > 0$ is arbitrary, or if $\beta=(d+1)/2$ and $\lambda > \lambda^{\text{crit}}_d$,
then $Z_{d,\lambda,\beta}$ is a  (Euclidean) polytope contained in $\BB^d$ (meaning that the zero cell is hyperbolically bounded), with probability $1$.
%\item[(ii)] If $\beta=(d+1)/2$ and $\lambda > \lambda_{\text{crit}}$, then $Z_{d,\lambda,\beta}$ is a  polytope contained in $\BB^d$, with probability $1$.
\item[(ii)] If $d/2<\beta<(d+1)/2$ and $\lambda>0$ is arbitrary, or if $\beta=(d+1)/2$ and $0 < \lambda < \lambda^{\text{crit}}_d$, then with non-vanishing probability  the set $Z_{d,\lambda,\beta}$ intersects $\SS^{d-1}$ (meaning that it is hyperbolically unbounded) and is not a (Euclidean) polytope.
%\item[(iii)] , then with non-vanishing probability the set $Z_{d,\lambda,\beta}$ intersects $\SS^{d-1}$ and is not a polytope.
\end{itemize}
\end{theorem}

The second claim of Part~(i) was known from~\cite[Lemma~5.3.3]{HeroldDiss}. It remains open to investigate the presumably fractal structure of the set $\SS^{d-1}\cap Z_{d,\lambda,\beta}$ in case (ii), for example to compute its Hausdorff dimension on the event that this set is non-empty (see~\cite{FengEtAl} for general results in this direction).
The doubly critical case when $\beta=(d+1)/2$ and $\lambda  = \lambda^{\text{crit}}_d$ remains open except in dimension $d=2$, where the following result is known; see~\cite{PorretBlanc}, \cite[Section~6]{BenjaminiJonassonEtAL} and~\cite{TykessonCalka}.

\begin{theorem}[Boundedness of the planar zero cell in the doubly critical case]\label{thm:PolytopeOrNot_critical_tessellation}
%Let $d=2$, $\beta = 3/2$ and $\lambda = \pi$, then
With probability $1$, $Z_{2,\pi, 3/2}$ is a (Euclidean) polygon contained in the open unit disk $\BB^2$.
\end{theorem}

\subsubsection{Expected \texorpdfstring{$f$}{f}-vector of the hyperbolic Poisson zero cell}
In the rest of the present Section~\ref{sec:Hyperplane} we concentrate on the case $\beta = (d+1)/2$  (which corresponds to the stationary hyperplane tessellation mentioned in Section~\ref{sec:Motivation}) and state explicit formulas for the expected $f$-vector of the corresponding beta$^*$ polytope. These can be translated to yield the expected $f$-vector of the hyperbolic Poisson zero cell $Z_{d,\lambda}^0 = Z_{d,\lambda, (d+1)/2}$ via the distributional identity
\begin{equation}\label{eq:duality_zero_cell_beta_star}
f_k(Z_{d,\lambda}^0) \stackrel{d}{=} f_{d-k-1} (P_{d,\alpha, \frac{d+1}{2}}), \qquad k = 0,\ldots, d-1,
\qquad
\alpha = \lambda \, \frac{\sqrt \pi \, \Gamma (\frac d2)}{2 \Gamma(\frac {d+1}{2})}.
\end{equation}
%The next corollary gives a formula for the expected $f$-vector of $P_{d,\alpha,\frac{d+1}{2}}$.
The formulas involve an array of numbers $A[n,k]$ which appeared in~\cite{KabluchkoZeroPolytope} and are defined in terms of the polynomials $Q_n(x)$, $n\in \NN_0$,  given by $Q_0(x) := Q_1(x) := 1$ and
\begin{align*}
	Q_n(x) :=
	\begin{cases}
		(1+(n-1)^2x^2)(1+(n-3)^2x^2)\cdot\ldots\cdot (1+3^2x^2)(1+1^2x^2)	&:\text{$n=2,4,6,\ldots$},\\
		(1+(n-1)^2x^2)(1+(n-3)^2x^2)\cdot\ldots\cdot (1+4^2x^2)(1+2^2x^2)	&:\text{$n=3,5,7,\ldots$}.
	\end{cases}
\end{align*}
The numbers $A[n,k]$, indexed by $n\in\NN_0$ and $k\in\ZZ$ with $n\geq k$, are then  defined by
\begin{equation}\label{eq:A-Terme}
	A[n,k]:=
	\begin{cases}
		[x^k]Q_n(x)		&:\text{$k$ even},\\
		[x^k]\big( \tanh(\frac{\pi}{2x})\cdot Q_n(x)\big)		&:\text{$k$ odd and $n$ even},\\
		[x^k]\big( \cotanh(\frac\pi {2x})\cdot Q_n(x)\big)		&:\text{$k$ odd and $n$ odd},
	\end{cases}
\end{equation}
where $[x^k]H(x)$ denotes the coefficient of $x^k$ in a formal power series $H(x)$.  Terms of the form $A[n,k]$ with $n<k$, whenever they appear, should be interpreted as $0$.

\begin{theorem}[Expected $f$-vector of the zero cell] \label{thm:f-vector,beta=(d+1)/2}
Suppose that $d\geq 2$,  $\alpha>(d-1)\pi$ and let $\ell\in\{1,\ldots,d\}$. Then
\begin{equation}\label{eq:thm:f-vector,beta=(d+1)/2}
\EE f_{\ell-1}\big(P_{d,\alpha,\frac{d+1}{2}}\big)
=
\frac{\pi^{\ell}}{\ell!} \sum_{\substack{m \in \{\ell,\ldots, d\}\\ m \equiv d \Mod 2}}
\bigg(\frac{\alpha}{2\pi}\bigg)^{m}
\frac{\Gamma(\frac{\alpha}{2\pi}-\frac{m-1}{2})}{\Gamma(\frac{\alpha}{2\pi}+\frac{m+1}{2})}
(A[m,\ell]-A[m-2,\ell]).
\end{equation}
\end{theorem}
%ZK: Both formulas in the above theorem and the following remark give numerical values that are consistent with the general formula for the expected f-vector of beta$^*$ polytopes.  %Beta Star Polytopes.nb
\begin{remark}
It is known that $A[n,n]=2^{-n}(n!)^2/\Gamma(\frac n2 +1)^2$ for all $n\in\NN_0$, see~\cite[Proposition~1.2]{KabluchkoZeroPolytope}. Hence, the expected number of facets of $P_{d,\alpha,(d+1)/2}$ is given by
	$$
\EE f_{d-1}(P_{d,\alpha,{d+1\over 2}}) = {d!\,\alpha^d\over 4^d\Gamma(1+{d\over 2})^2}{\Gamma({\alpha\over 2\pi}-{d-1\over 2})\over\Gamma({\alpha\over 2\pi}+{d+1\over 2})}.
$$
\end{remark}
It is straightforward to check that the right-hand side of~\eqref{eq:thm:f-vector,beta=(d+1)/2} is always a rational function of $\alpha$. In the next corollary we specialize~\eqref{eq:thm:f-vector,beta=(d+1)/2} to small space dimensions $d\in\{2,3,4\}$.
%Taking $\alpha=\lambda$ yields the the expected face numbers of $Z_0(\lambda)$.
Numerical values are shown in Figure~\ref{fig:VertexNumberHyperplaneMosaic} (recall~\eqref{eq:duality_zero_cell_beta_star}).

\begin{figure}[t]
\centering
\includegraphics[width=0.3\columnwidth]{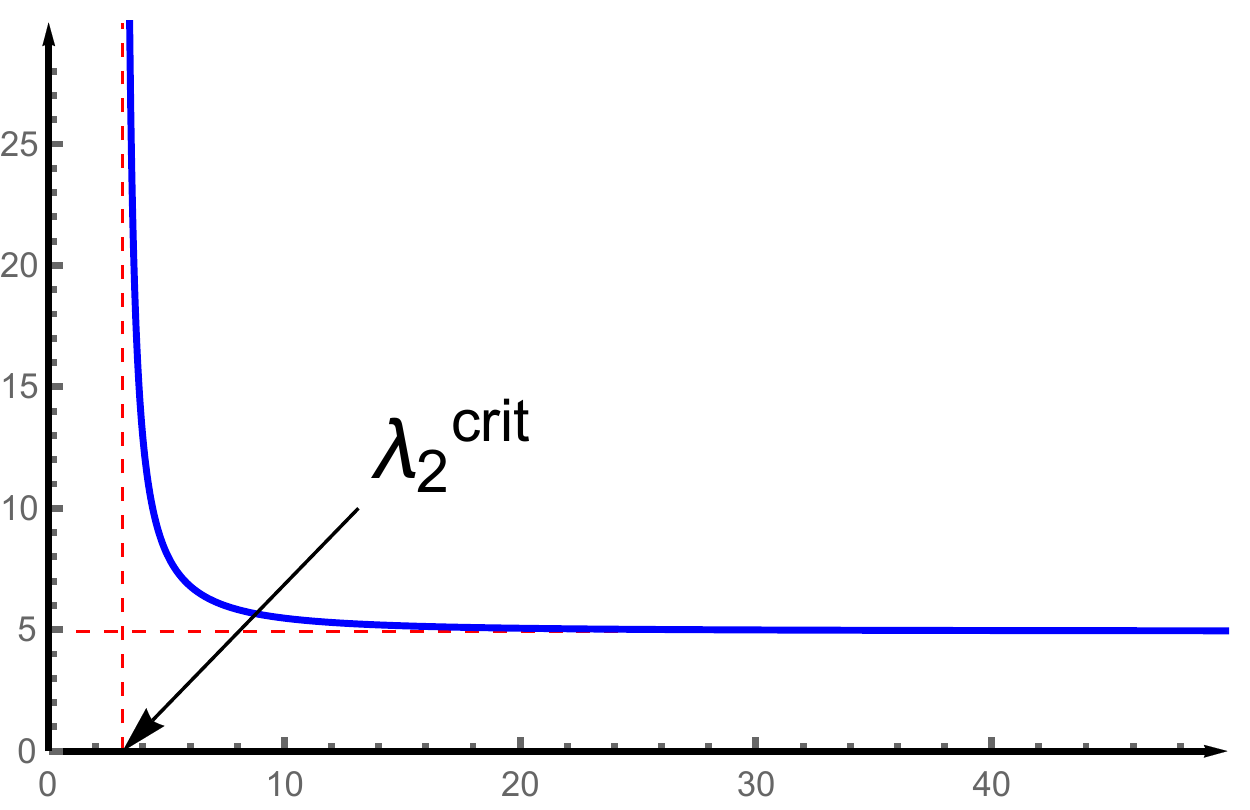}\quad
\includegraphics[width=0.3\columnwidth]{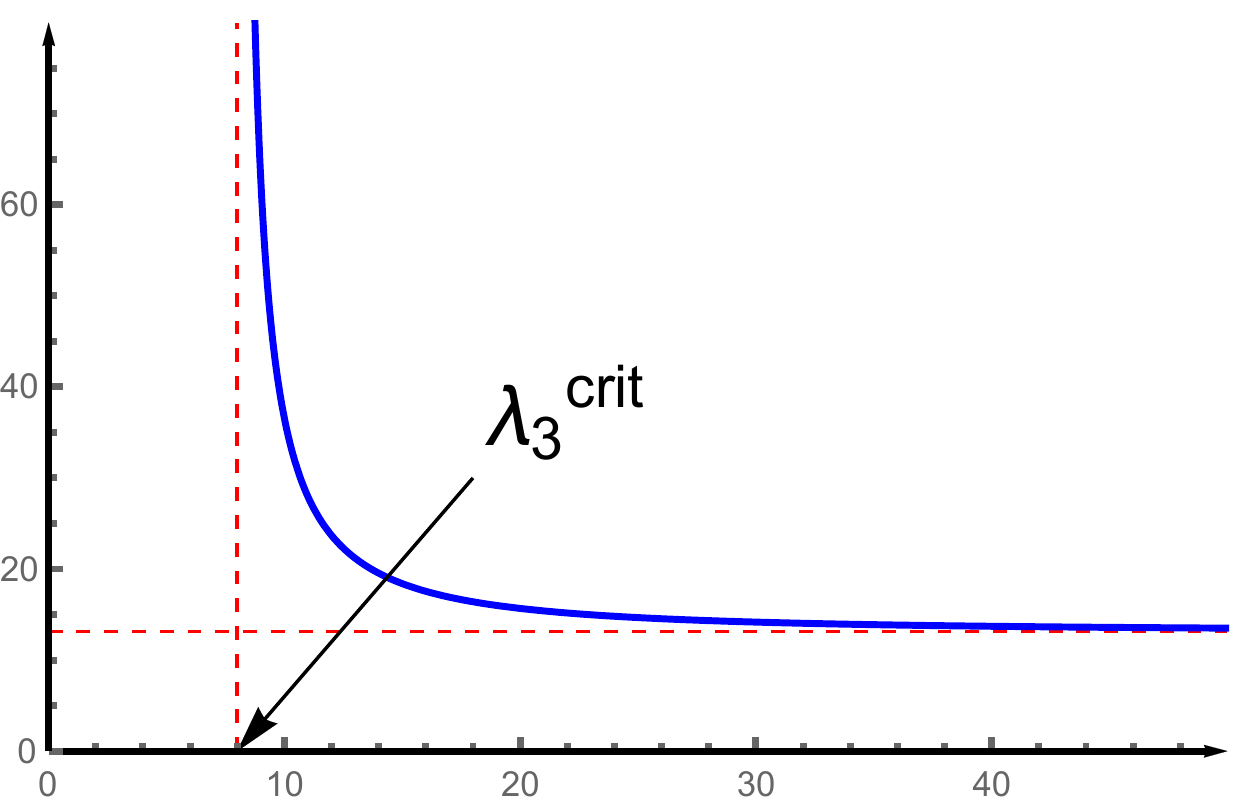}\quad
\includegraphics[width=0.3\columnwidth]{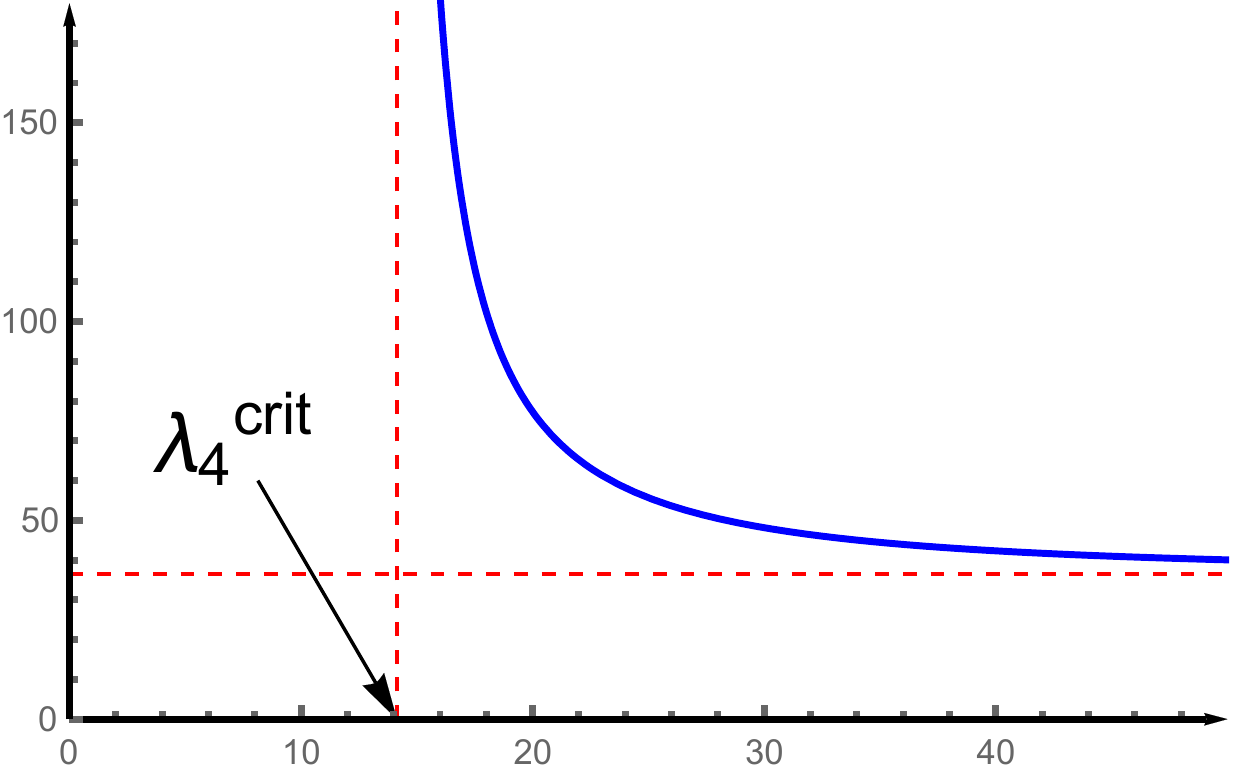}
\caption
{
The expected vertex number of the zero cell of the stationary hyperbolic Poisson hyperplane tessellation $Z_{d,\lambda,(d+1)/2}=Z^{0}_{d,\lambda}$ in dimension $d=2$ (left panel), $d=3$ (middle panel) and $d=4$ (right panel) as a function of $\lambda >\lambda^{\text{crit}}_d$, where  $\lambda^{\text{crit}}_d$ is shown as a dashed vertical line. The dashed horizontal line represents the expected vertex number of the zero cell of the stationary and isotropic \textit{Euclidean} Poisson hyperplane tessellation; see Section~\ref{sec:Convergence}.  %After rescaling the horizontal axis with $\alpha = \frac 12 \lambda\sqrt \pi \, \Gamma (\frac d2)/ \Gamma(\frac {d+1}{2})$, this shows the expected facet number of the beta$^*$  polytope $P_{d,\alpha,{d+1\over 2}}$.
}
\label{fig:VertexNumberHyperplaneMosaic}
\end{figure}

\begin{corollary}\label{cor:ZeroCellSmallDimensions}
\begin{itemize}
\item[(i)] 	For $d=2$ and $\alpha>\pi$ one has that
\begin{align*}
	\EE f_0(P_{2,\alpha,\frac 32})=\EE f_1(P_{2,\alpha,\frac 32})={\alpha^2\pi^2\over 2(\alpha^2-\pi^2)}.
\end{align*}
\item[(ii)] 	For $d=3$ and $\alpha>2\pi$ one has that
\begin{align*}
	\EE f_k(P_{3,\alpha,2})=
	\begin{cases}
		{2\alpha^2\pi^2\over 3(\alpha^2-4\pi^2)}+2	&:k=0,\\
		{2\alpha^2\pi^2\over \alpha^2-4\pi^2}	&:k=1,\\
		{4\alpha^2\pi^2\over 3(\alpha^2-4\pi^2)}	&:k=2.
	\end{cases}
\end{align*}

\item[(iii)] For $d=4$ and $\alpha>3\pi$ one has that
\begin{align*}
	\EE f_k(P_{4,\alpha,\frac 52})=
	\begin{cases}
		\frac{40 \alpha^4 \pi^2 - 36 \alpha^2 \pi^4 - 3 \alpha^4 \pi^4}{8(\alpha^4-10\alpha^2\pi^2+9\pi^4)}	&:k=0,\\
		\frac{10 \alpha^4 \pi^2 - 9 \alpha^2 \pi^4}{2(\alpha^4-10\alpha^2\pi^2+9\pi^4)}	&:k=1,\\
		\frac{3\alpha^4\pi^4}{4(\alpha^4-10\alpha^2\pi^2+9\pi^4)}	&:k=2,\\
		\frac{3\alpha^4\pi^4}{8(\alpha^4-10\alpha^2\pi^2+9\pi^4)}	&:k=3.
	\end{cases}
\end{align*}
\end{itemize}
\end{corollary}
%%%ZK: Die aktuellen Formeln sind numerisch konsistent mit Theorem~\ref{thm:f-vector,beta=(d+1)/2}. Die Grenzwerte für $\alpha\to\infty$ stimmen mit dem Euklidschen Fall überein.

Part (i) recovers a formula of Santal{\'o} and Ya\~{n}ez~\cite[Equation~(6.9) on p.~163]{SantaloYanez}. Note that their $\lambda$ corresponds to our
$\alpha/\pi = \lambda/\pi$
(by~\eqref{eq:duality_zero_cell_beta_star} these quantities are equal for $d=2$); see the paragraph after Equation~(2.12) in~\cite{SantaloYanez}.

The papers~\cite{SantaloYanez} and~\cite{santalo_average} contain  many other explicit formulas related to Poisson line tessellations in the hyperbolic plane, but cases (ii) and (iii) seem to be new.

Let us finally discuss the asymptotic behaviour of $\EE f_k(Z_{d,\lambda}^0)$ in the large intensity limit, i.e.\ as  $\lambda\to\infty$. By~\eqref{eq:duality_zero_cell_beta_star}, we may pass to the convex dual and  look at $\EE f_{d-k-1}(P_{d,\alpha,(d+1)/2})$ as $\alpha\to\infty$.
The terms appearing in Theorem~\ref{thm:f-vector,beta=(d+1)/2} satisfy
$$
T_{m}(\alpha)
:=
\left(\frac{\alpha}{2\pi}\right)^{m} {\Gamma({\alpha\over 2\pi}-{m-1\over 2})\over\Gamma({\alpha\over 2\pi}+{m+1\over 2})}
=
1 + \frac{\pi^2 m (m^2-1)}{6 \alpha^2} + O(\alpha^{-4})
\;\;
\text{ as } \alpha\to\infty,
$$
see~\cite{tricomi_erdelyi} for the asymptotic expansion of the quotient of Gamma functions.
%Eigentlich ist es eine Mathematica-Rechnung.
As a result,
\begin{align*}
	\EE f_{\ell-1}\big(P_{d,\alpha,\frac{d+1}{2}}\big)
	&=\frac{\pi^{\ell}}{\ell!}\sum_{\substack{m\in \{\ell,\ldots, d\}\\ m\equiv d \Mod 2}} (A[m,\ell]-A[m-2,\ell])+ {Q_{d,\ell}\over\alpha^2} + O(\alpha^{-4})\\
	&=\frac{\pi^{\ell}}{\ell!}A[d,\ell]+ {Q_{d,\ell}\over\alpha^2} + O(\alpha^{-4})
\;\;
\text{ as }
\alpha\to\infty,
\end{align*}
where $Q_{d,\ell}>0$ is a constant only depending on $d$ and $\ell$.
Not surprisingly, the first term is the expected number of $(\ell-1)$-faces of the convex dual of the \emph{flat} (i.e.\ Euclidean) Poisson zero cell, which is visible as a dashed horizontal line in Figure~\ref{fig:VertexNumberHyperplaneMosaic}; see~\cite[Theorem~2.1]{KabluchkoZeroPolytope}. The large intensity limit will be treated in more generality in Section~\ref{sec:Convergence}.
%the result is consistent with that of Theorem \ref{thm:convergence_beta_star}.

\subsection{Poisson-Voronoi tessellations in the hyperbolic space}\label{sec:Voronoi}
Next we turn to the study of the typical cell in the $d$-dimensional hyperbolic Poisson-Voronoi tessellation. We shall define a more general model, indexed by two parameters $\lambda$ and $\beta$, which reduces to the special case we are interested in if $\beta = d$.

\subsubsection{Reduction to \texorpdfstring{beta$^*$}{beta*} polytopes}
We start with a model-independent definition which is for convenience stated in the hyperboloid model. Recall that $e= (1,0,\ldots,0)$ is the apex of $\HH^d$. Fix some parameters $\lambda>0$ and $\beta>\max(d/2,1)$ and consider a Poisson process $\eta_{d,\lambda,\beta}$ on $\HH^d$ whose intensity with respect to the hyperbolic volume measure $\nu_{\text{hyp}}$ on $\HH^d$ is given by
$$
x\mapsto  \lambda \cdot \left(\sinh \frac {d_{\text{hyp}}(e,x)} 2 \right)^{2\beta - 2 d},
\qquad
x\in \HH^d,
$$
where we recall that $d_{\text{hyp}}(\cdot, \cdot)$ denotes the hyperbolic distance on $\HH^d$.  The \textit{Voronoi cell} of $e$ in the point process $\eta_{d,\lambda,\beta}\cup\{e\}$ is the random set
$$
V_{d,\lambda,\beta} := \{x\in \HH^d: d_{\text{hyp}}(x,e) \leq d_{\text{hyp}}(y,e) \text{ for all } y\in \eta_{d,\lambda,\beta}\}.
$$
Actually, we are mostly interested in the special case $\beta=d$ in which the intensity measure is $\lambda$ times the Riemannian volume measure $\nu_{\text{hyp}}$, the Poisson process $\eta_{d,\lambda,d}$ is stationary under isometries of $\HH^d$, and $V_{d,\lambda,d}$ has the same distribution as the typical cell of the stationary Poisson-Voronoi tessellation generated by $\eta_{d,\lambda,d}$, as explained in Section~\ref{sec:Motivation}.

\begin{figure}[t]
%\begin{center}
\centering
	\includegraphics[width=0.35\columnwidth]{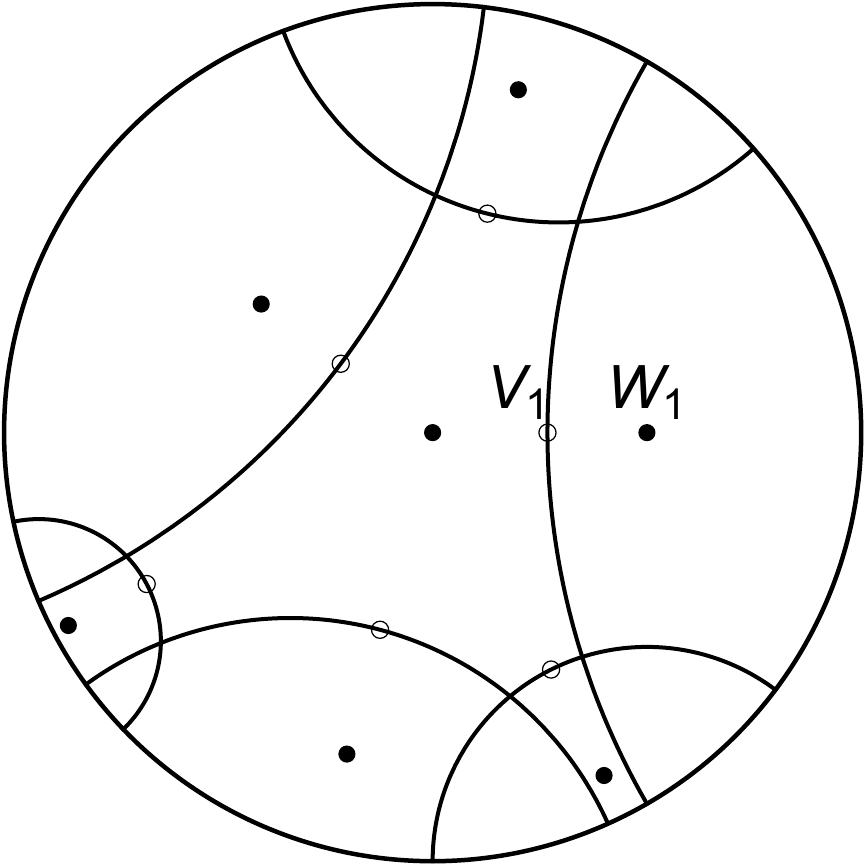}
\hspace*{1cm}
	\includegraphics[width=0.35\columnwidth]{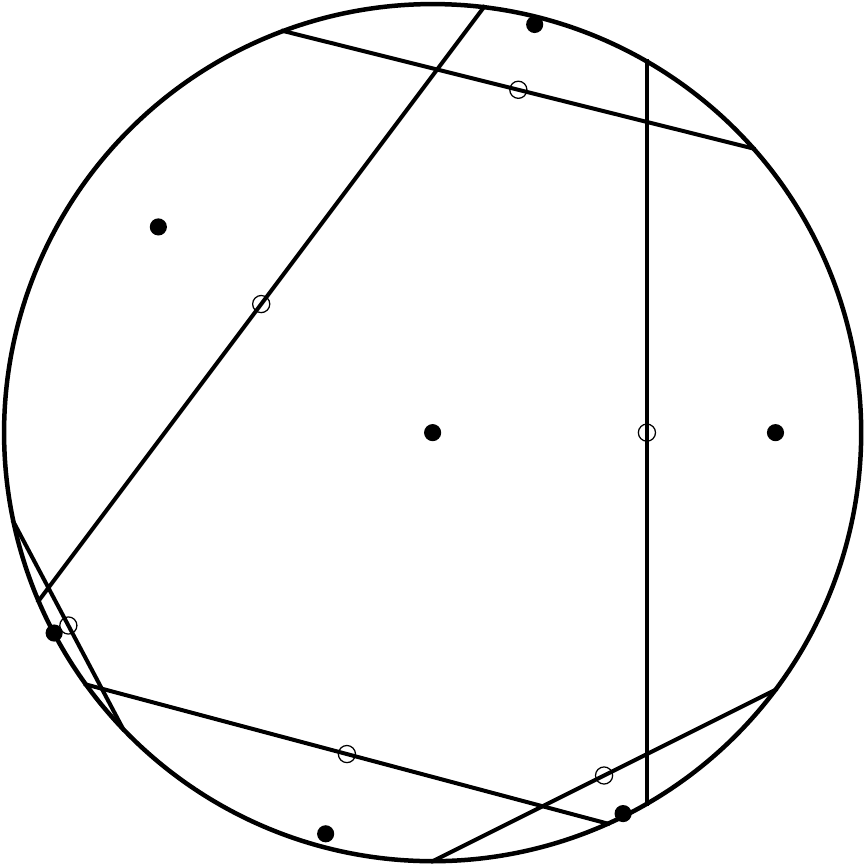}
%\end{center}
	\caption{Proof of Theorem~\ref{thm:TypicalVoronoiCellGeneralized}. The left panel shows the construction of the Voronoi cell in the Poincar\'e model. The right panel is the image of the left one under the map $\varphi_{\text{Poi$\to$Kl}}$.}
\label{fig:voronoi_proof}
\end{figure}

\begin{theorem}[Reduction of the typical Voronoi cell to beta$^*$ sets, I]\label{thm:TypicalVoronoiCellGeneralized}
Let $d\in \NN$, $\lambda>0$ and $\beta>\max(d/2,1)$.
Then the closure of the gnomonic projection $\Pi_{\text{gn}}(V_{d,\lambda,\beta})$ has the same distribution as the convex dual $P_{d,\alpha,\beta}^\circ$ of the beta$^*$ set $P_{d,\alpha,\beta}$ with $\alpha =2^d \lambda /\tilde c_{d,\beta}$.
%and the same distribution as the set $Z_{d,2^d \lambda \omega_d,\beta}$ defined in Section~\ref{sec:Hyperplane}.
\end{theorem}
\begin{proof}
The proof is divided into $3$ steps. The main idea is shown on Figure~\ref{fig:voronoi_proof}.

\medskip
\noindent
\textit{Step 1.}
First we pass  to the Poincar\'e model on $\BB^d$ by the stereographic projection $\Pi_{\text{st}}$. Note that it maps the apex $e$ to $0$. The image of $\eta_{d,\lambda,\beta}$ under $\Pi_{\text{st}}$ is again a Poisson process denoted by $\Pi_{\text{st}}(\eta_{d,\beta, d})$.  Recalling the formula for the volume element~\eqref{eq:metric_volume_poincare} as well as  the formulas $d_{\text{Poi}} (0, w) = 2 \artanh \|w\|$ and $\sinh (\artanh \|w\|) = \|w\| /\sqrt{1-\|w\|^2}$, we conclude that  the intensity of  $\Pi_{\text{st}}(\eta_{d,\lambda,\beta})$ with respect to the Lebesgue measure on $\BB^d$ is given by
\begin{equation}\label{eq:intensity_W}
w\mapsto 2^d \lambda  \frac{\|w\|^{2\beta - 2 d}}{(1-\|w\|^2)^{\beta}},
\qquad
w\in \BB^d.
\end{equation}
Let $W_1,W_2,\ldots \in \BB^d$ be the points of $\Pi_{\text{st}}(\eta_{d,\lambda,\beta})$.
The typical Voronoi cell in the Poincar\'e model is given by
$$
\Pi_{\text{st}}(V_{d,\lambda,\beta}) = \{w\in \BB^d: d_{\text{Poi}}(w,0) \leq d_{\text{Poi}}(W_i,0) \text{ for all } i\in \NN\}.
$$
Let us give a more explicit description of this cell as an intersection of hyperbolic half-spaces.
For every point $W_i$ let $V_i$ be midpoint, in the sense of the Poincar\'e metric, of the segment joining $0$ to $W_i$, that   is  $2 d_{\text{Poi}}(0,V_i) = d_{\text{Poi}}(0,W_i)$. Let $H_i^-$ be a hyperbolic half-space, still in the sense of the Poincar\'e model, which contains $0$ and whose bounding hyperbolic hyperplane $H_i$ passes through $V_i$ and is orthogonal to the segment joining $0$ to $W_i$; see the left panel of Figure~\ref{fig:voronoi_proof}.  Then,
$$
\Pi_{\text{st}}(V_{d,\lambda,\beta}) = \bigcap_{i=1}^\infty H_i^-.
$$
%Note that, in fact, $H_i$ is a sphere orthogonal to the unit sphere.

\medskip
\noindent
\textit{Step 2.}
And now let us map everything to the Klein model by the map $\varphi_{\text{Poi$\to$Kl}}:\BB^d\to\BB^d$. Using the relation
$$
d_{\text{Poi}} (0, V_i) =  \frac 12 d_{\text{Poi}} (0, W_i) =  d_{\text{Kl}} (0, W_i)
$$
mentioned in~\eqref{eq:dist_0_x_poincare}, we conclude that $\varphi_{\text{Poi$\to$Kl}}(V_i) = W_i$. It follows that $\varphi_{\text{Poi$\to$Kl}}(H_i^-)$ is a half-space $G_i^-$, in the sense of the Klein model, which contains $0$ and whose boundary (which is a Euclidean affine hyperplane) passes through the point $W_i$ and is orthogonal to the segment $[0, W_i]$; see the right panel of Figure~\ref{fig:voronoi_proof}.  Summarizing, we have
$$
\Pi_{\text{gn}}(V_{d,\lambda,\beta}) := \varphi_{\text{Poi$\to$Kl}}(\Pi_{\text{st}}(V_{d,\lambda,\beta}))
=
\{v\in \BB^d: \langle v, W_i \rangle \leq \|W_i\|^2 \text{ for all } i\in \NN\}.
$$

\medskip
\noindent
\textit{Step 3.}
Note that the Poisson process formed by the points $W_1,W_2,\ldots$ is isotropic and their radial parts $\|W_1\|, \|W_2\|,\ldots$ form a Poisson process on $(0,1)$ with intensity
$$
r\mapsto 2^d \lambda \omega_d \frac{r^{2\beta -  d - 1}}{(1-r^2)^{\beta}},
\qquad
0<r<1,
$$
which follows from~\eqref{eq:intensity_W} together with the polar integration formula.
Recalling the notation introduced in Section~\ref{subsubsec:hyperbolic_hyperplane_tess_def} we conclude that the closure of $\Pi_{\text{gn}}(V_{d,\lambda,\beta})$ has the same distribution as the zero cell $Z_{d,2^d \lambda \omega_d,\beta}$ defined in~\eqref{eq:zero_cell_def}. By Theorem~\ref{thm:ZeroCell}, the latter has the same distribution as the convex dual of the beta$^*$ set $P_{d,\alpha,\beta}$ with $\alpha = 2^d \lambda \omega_d / (\tilde c_{d,\beta}\omega_d) =2^d \lambda / \tilde c_{d,\beta}$, which completes the proof.
\end{proof}

\begin{remark}
From the above proof and Theorem~\ref{thm:PolytopeOrNot_Tessellation} we see that $V_{d,\lambda,\beta}$ is a.s.\ bounded in the hyperbolic metric provided that $\beta>(d+1)/2$ or $\beta= (d+1)/2$ and $2^d \lambda \omega_d > \lambda^{\text{crit}}_d$. In these cases, the word ``closure'' can be removed from the statement of Theorem~\ref{thm:TypicalVoronoiCellGeneralized}. %In particular, the Voronoi cell is bounded in the case $\beta = d$.
\end{remark}

\subsubsection{Expected \texorpdfstring{$f$}{f}-vector of the typical hyperbolic Poisson-Voronoi cell}
Next we turn to the special case of the stationary Poisson-Voronoi tessellation discussed in Section~\ref{sec:Motivation}.
Specializing Theorem~\ref{thm:TypicalVoronoiCellGeneralized} to $\beta=d$ (and noting that the case $\beta=d=1$ not covered by it can be proven in the same way)  we obtain the following result.

%Let us recall from Section \ref{sec:Motivation} that $V^{\text{typ}}_{d,\lambda}$ stands for the typical cell of a stationary Poisson-Voronoi tessellation with intensity $\lambda>0$ in the hyperbolid model of the hyperbolic $d$-space $\HH^d$. We also recall that $\Pi_{\text{gn}}:\HH^d\to\BB^d$ denotes the gnomonic projection, which relates the hyperbolid model with the Klein model of the $d$-dimensional hyperbolic space. For a convex set $K\subset\RR^d$, containing the origin in its interior, we let $K^\circ$ be the usual convex dual of $K$.

\begin{theorem}[Reduction of the typical Voronoi cell to beta$^*$ sets, II]\label{thm:TypicalVoronoiCell}
The gnomonic projection of the typical cell $V^{\text{typ}}_{d,\lambda}$ of the stationary hyperbolic Poisson-Voronoi tessellation with intensity $\lambda>0$ on $\HH^d$, $d\in\NN$, has the same distribution as the convex dual $P_{d,\alpha,d}^\circ$ of the beta$^*$ polytope $P_{d,\alpha,d}$ with
%Let $V^{\text{typ}}_{d,\lambda}$ be the typical cell of a hyperbolic Poisson-Voronoi tessellation of intensity $\lambda$ and $P_{d,\tilde c_{d,d}^{-1}2^d\lambda,d}$ be the beta$^*$ polytope with parameter $d$ and intensity
$$
\alpha = 2^d\lambda/\tilde c_{d,d} = \frac{2^d \pi^{d/2}  \lambda \, \Gamma(d/2)}{\Gamma(d)}.
$$
%In particular, we have
\end{theorem}

Theorem~\ref{thm:TypicalVoronoiCell} together with Theorem~\ref{thm:f-vector_beta_star} imply an explicit formula for the expected number of $k$-faces of the typical hyperbolic Voronoi cell $V^{\text{typ}}_{d,\lambda}$.
\begin{theorem}[Expected $f$-vector of the typical Voronoi cell]\label{thm:TypicalVoronoiCellExpectedFaceNums}
For all $d\in \NN$, $\lambda>0$ and $k\in\{0,\ldots,d-1\}$ we have
\begin{align}\label{eq:k-faces_typ_Voronoi}
	\EE f_k(V^{\text{typ}}_{d,\lambda})
	=\EE f_{d-k-1}\big(P_{d,2^d\lambda/\tilde c_{d,d},d}\big)
	=2\sum_{s=0}^{\lfloor k/2 \rfloor} \mathbb I^*_{2^d\lambda/\tilde c_{d,d},d-2s}(d)\cdot\tilde{\mathbb J}_{d-2s,d-k}\Big(d-s-\frac 12\Big).
\end{align}
\end{theorem}

For small values of $d$, we will now present explicit formulas for the expected number of $k$-faces of $P_{d,\alpha,d}$, and thus, also for the expected face numbers of $V^{\text{typ}}_{d,\lambda}$; see the first equality in~\eqref{eq:k-faces_typ_Voronoi}. We remark that for $d=2$ our result is in line with that of Isokawa~\cite[Theorem 1]{IsokawaPlane} and for $d=3$ with the one obtained in~\cite[Theorem 1.1]{Isokawa}.  Some numerical values are shown in Figure~\ref{fig:VertexNumber}.

\begin{figure}[t]
\centering
\includegraphics[width=0.3\columnwidth]{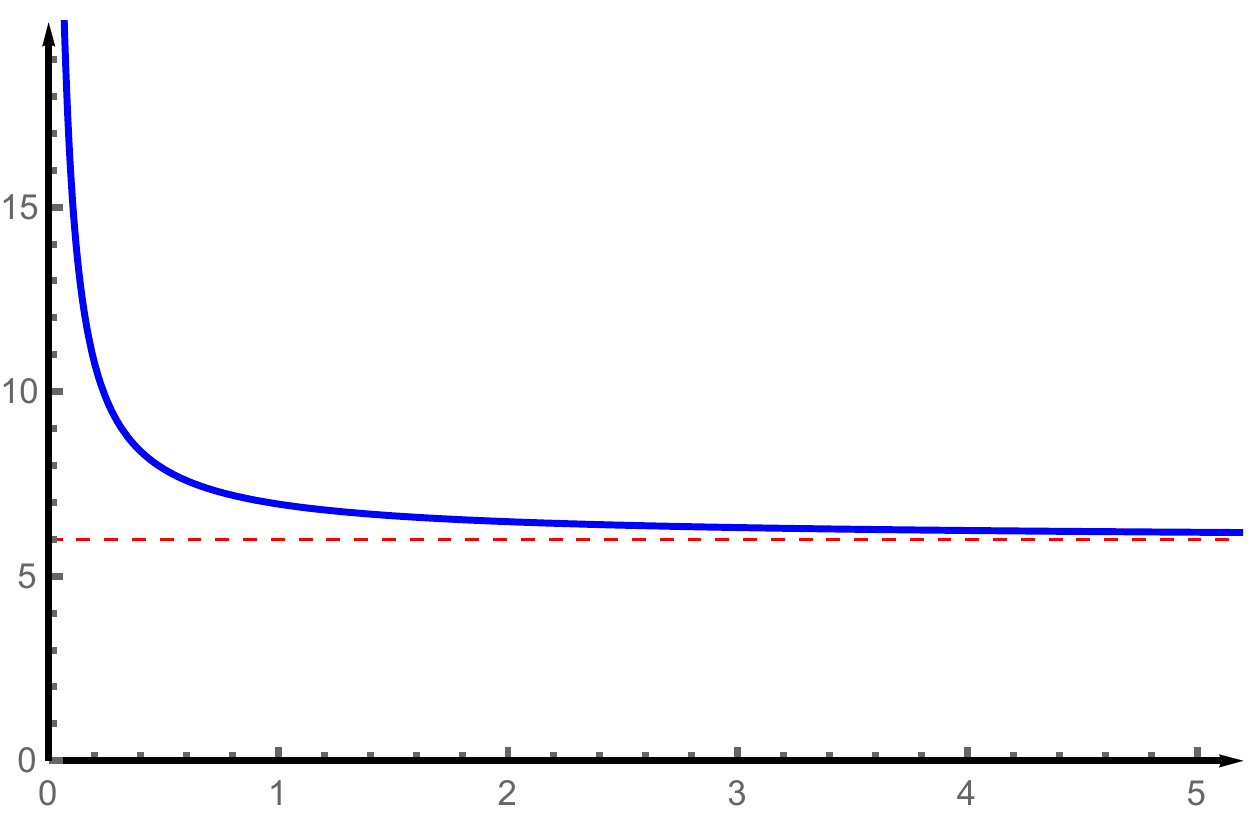}\quad
\includegraphics[width=0.3\columnwidth]{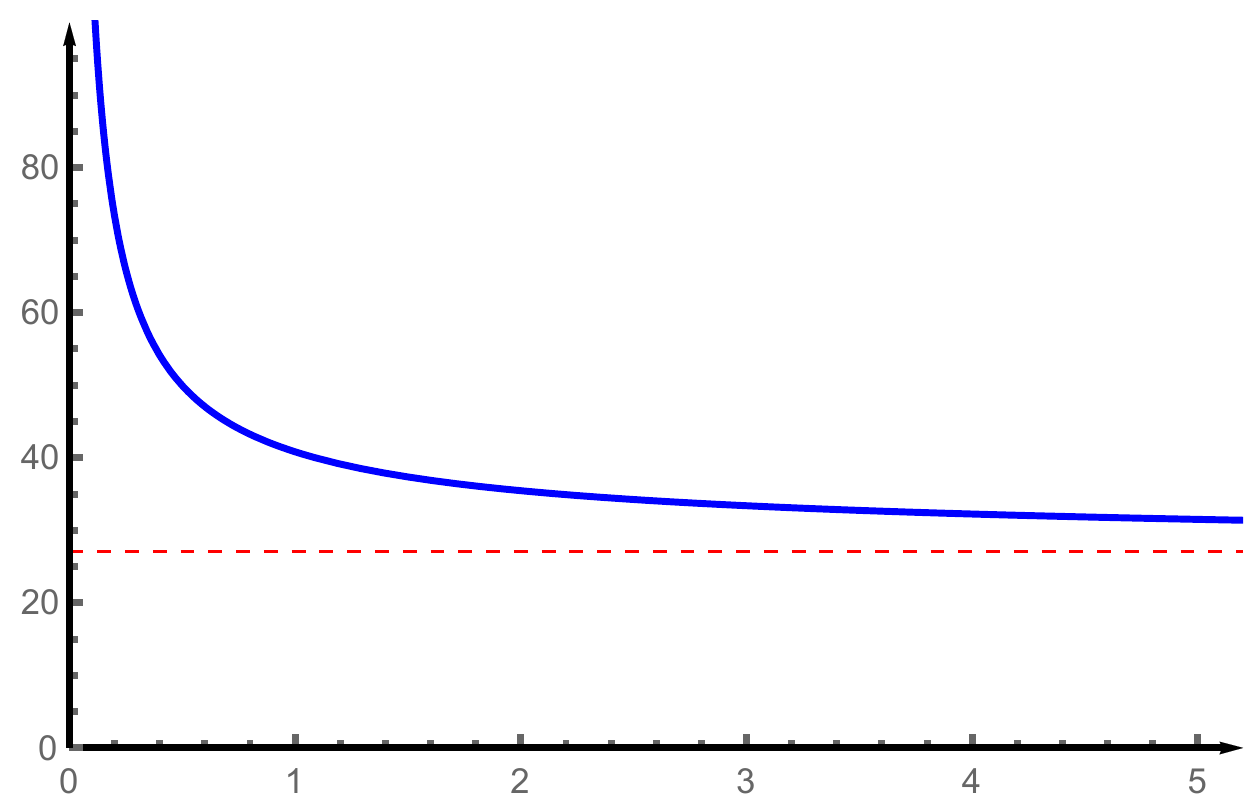}\quad
\includegraphics[width=0.3\columnwidth]{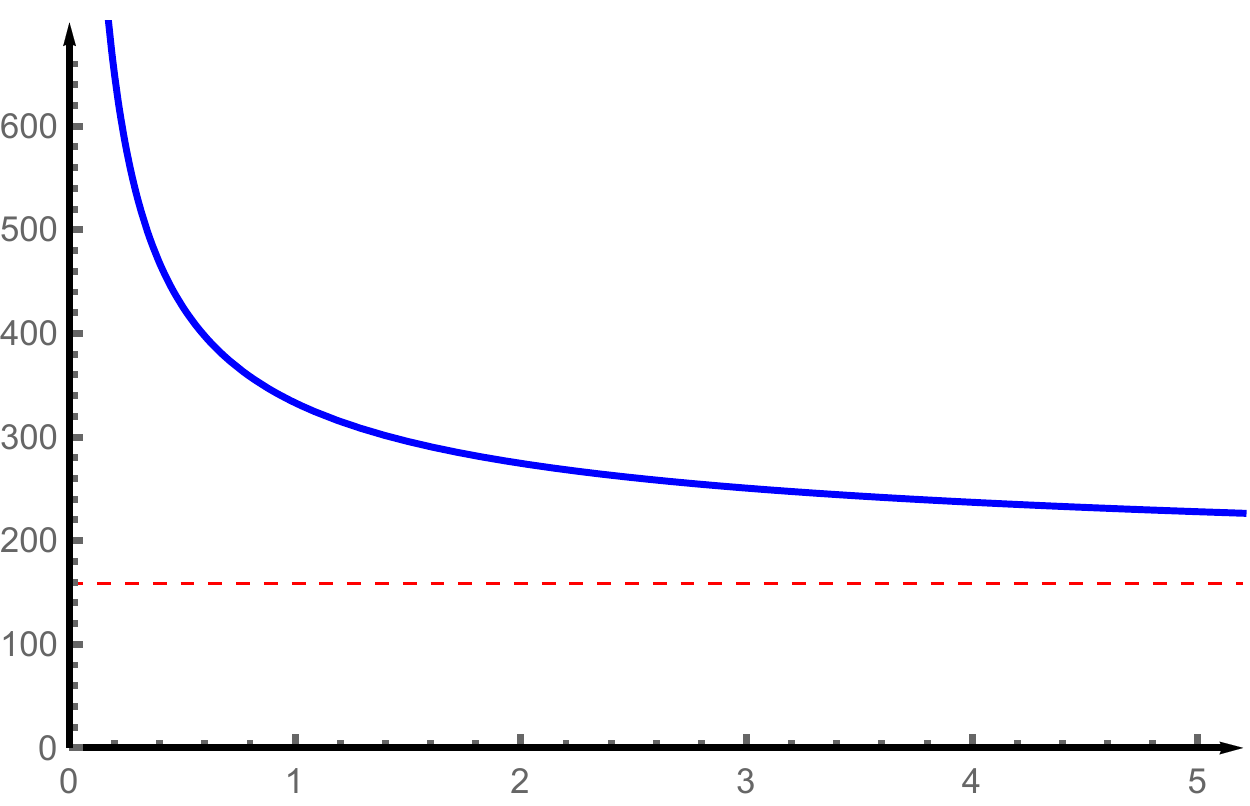}
\caption
{
The expected vertex number of the typical hyperbolic Voronoi cell $V^{\text{typ}}_{d,\lambda}$ in dimension $d=2$ (left panel), $d=3$ (middle panel) and $d=4$ (right panel) for intensities $0<\lambda\leq 5$. The dashed horizontal line represents  the expected vertex number of the typical \textit{Euclidean} Voronoi cell; see  Section~\ref{sec:Convergence}.  %After rescaling the horizontal axis with $\alpha = 2^d \pi^{d/2}  \lambda  \Gamma(d/2)/\Gamma(d)$, this shows the expected facet number of the beta star polytope $P_{d,\alpha,d}$  }.
}
\label{fig:VertexNumber}
\end{figure}

\begin{corollary}\label{cor:k-faces_beta=d=234}
	Let $\alpha>0$.
\begin{itemize}
\item[(i)] For $\beta=d=2$ one has that
\begin{align*}
	\EE f_0(P_{2,\alpha,2}) = \EE f_1(P^{}_{2,\alpha,2}) = 6\Big(1+\frac{2}{\alpha}\Big).
\end{align*}

\item[(ii)] For $\beta=d=3$ one has that
\begin{align*}
	\EE f_k(P_{3,\alpha,3})	=
	\begin{cases}
		\mathbb I^*_{\alpha,3}(3)+2	&:k=0,\\
		3\,\mathbb I^*_{\alpha,3}(3)	&:k=1,\\
		2\,\mathbb I^*_{\alpha,3}(3) 	&:k=2,
	\end{cases}
\end{align*}
where
\begin{align*}
	\mathbb I^*_{\alpha,3}(3)
	&	=\frac{64\alpha^3}{105\pi}\int\limits_0^\infty \sinh^8 (\varphi)  e^{-\frac{\alpha}{2\pi}(\sinh(2\varphi)-2\varphi)}\,\dint\varphi=\frac{64\alpha^3}{105\pi}\int\limits_0^\infty \sinh^8 (\varphi) e^{-\frac{\alpha}{\pi}(\sinh(\varphi) \cosh (\varphi)-\varphi)}\,\dint\varphi.
\end{align*}

\item[(iii)] For $\beta=d=4$ one has that
\begin{align*}
	\EE f_k(P_{4,\alpha,4})
	=
	\begin{cases}
		\frac{54}{143}\,\mathbb I^*_{\alpha,4}(4)+2\mathbb I^*_{\alpha,2}(4)	&:k=0,\\
		\frac{340}{143}\,\mathbb I^*_{\alpha,4}(4)+2\mathbb I^*_{\alpha,2}(4)	&:k=1,\\
		4\,\mathbb I^*_{\alpha,4}(4) 	&:k=2,\\
		2\,\mathbb I^*_{\alpha,4}(4) 	&:k=3,
	\end{cases}
\end{align*}
where
\begin{align*}
	\mathbb I^*_{\alpha,4}(4)
	&	=\frac{2145\alpha^4}{32768}\int\limits_0^\infty \sinh^{15} (\varphi)  e^{-\alpha(2+\cosh \varphi)\sinh^4(\frac\varphi 2)}\,\dint\varphi,\\
	\mathbb I^*_{\alpha,2}(4)
	&	=\frac{35\alpha^2}{64}\int\limits_0^\infty \sinh^7 (\varphi)  e^{-\alpha(2+\cosh \varphi)\sinh^4(\frac\varphi 2)}\,\dint\varphi.
\end{align*}
\end{itemize}
\end{corollary}
%%%ZK: Die Formeln scheinen numerisch konsistent mit der Formel für den erwarteten $f$-Vektor des beta$^*$-Polytops  zu sein.

\subsection{Asymptotics of \texorpdfstring{beta$^*$}{beta*} polytopes for large intensities and monotonicity}\label{sec:Convergence}
Consider a Poisson-Voronoi tessellation or a Poisson hyperplane tessellation of the hyperbolic space whose intensity parameter $\alpha$ increases to $\infty$. The cells of such tessellations become smaller and since the role of the curvature becomes negligible on small scales, it is natural to conjecture that the hyperbolic cells become close to their Euclidean counterparts as $\alpha\to\infty$. Similar conclusion should apply to the corresponding cells in the spherical geometry.
In this section we shall confirm this conjecture by studying the  asymptotic behaviour of  beta$^*$ sets $P_{d,\alpha,\beta}$ with fixed $\beta>d/2$, as $\alpha\to\infty$. Observe that the rescaled beta$^*$ set $\alpha^{-1/(2\beta-d)}P_{d,\alpha,\beta}$ is generated by a Poisson process  with Lebesgue intensity
\begin{align}\label{eq:convergence_density_beta$^*$}
\alpha^{\frac d{2\beta-d}}\cdot f_{d,\alpha,\beta}\big(x\cdot \alpha^{\frac{1}{2\beta-d}}\big)=\tilde c_{d,\beta}\big(\|x\|^2-\alpha^{-\frac{2}{2\beta-d}}\big)^{-\beta}
\ind_{\{\|x\|>\alpha^{-1/(2\beta-d)}\}}
\overset{}{\underset{\alpha\to\infty}\longrightarrow}
\tilde c_{d,\beta}\cdot \|x\|^{-2\beta},
\end{align}
where the limit holds pointwise for all $x\in\RR^d\backslash \{0\}$ and we used that $2\beta-d>0$ by assumption. This motivates us to consider a Poisson process $\chi_{d,\mu,\beta}$ on $\RR^d\backslash \{0\}$ whose intensity (with respect to the Lebesgue measure) is given by
$$
x\mapsto \frac{\mu \, \tilde c_{d,\beta}}{\|x\|^{2\beta}},
\qquad
x\in\RR^d\backslash\{0\}.
$$
Here, $\mu>0$ and $\beta> d/2$ are parameters. The atoms of $\chi_{d,\mu,\beta}$ cluster at $0$ but not at $\infty$. The convex hull of the atoms of $\chi_{d,\mu,\beta}$ is a random polytope, denoted here by $\conv \chi_{d,\mu,\beta}$,  which already appeared in~\cite{KabluchkoAngles,KabluchkoMarynychTemesvariThaele,KabluchkoThaeleZaporozhets} (with a different parametrization) and is referred to as a Poisson polytope in the sequel. Its convex dual object appeared earlier in~\cite{HoermannHugReitznerThaele,HugSchneiderLargeCells}, see also the references cited therein. The next theorem is not surprising in view of~\eqref{eq:convergence_density_beta$^*$}.

%To this end, for $\lambda>0$ let $\chi_{d,\lambda}$ be a Poisson process on $\RR^d\backslash \{0\}$ with Lebesgue density $\|x\|^{-(d+\lambda)}$
%This suggests that, as $\alpha\to\infty$, the beta$^*$ set $P_{d,\alpha,\beta}$ converges, after suitable rescaling to the convex hull of the Poisson process $\chi_{d,2\beta-d}$ (also suitably rescaled).

%We can further address the suitable scaling parameters. Consider the Poisson polytope $\conv\chi_{d,2\beta-d}$, rescaled by a parameter $\mu>0$. Then $\mu\cdot \conv\chi_{d,2\beta-d}$ is generated by a Poisson process on $\RR^d\backslash\{0\}$ whose density is
%\begin{align*}
%\Big\|\frac x\mu\Big\|^{-2\beta}\cdot \frac{1}{\mu^d}=\|x\|^{-2\beta}\mu^{2\beta-d},\qquad x\in\RR^d\backslash\{0\}.
%\end{align*}
%Choosing $\mu=(\tilde c_{d,\beta})^{\frac{1}{2\beta-d}}$ yields that $\mu\cdot\conv \chi_{d,2\beta-d}$ has density $\tilde c_{d,\beta}\cdot \|x\|^{-2\beta}$.

%In view of ~\eqref{eq:convergence_density_beta$^*$}, this suggests that the rescaled beta$^*$ polytope $\alpha^{-\frac{1}{2\beta-d}}P_{d,\alpha,\beta}$ converges weakly as $\alpha\to\infty$ to the rescaled Poisson polytope $(\tilde c_{d,\beta})^{\frac{1}{2\beta-d}}\conv\chi_{d,2\beta-d}$ on the space of compact convex subsets of $\RR^d$ equipped with the Hausdorff metric. Our next theorem shows that this is indeed the case.

\begin{theorem}[Convergence of beta$^*$ sets]\label{thm:PolytopesConverge}
Fix $\beta > d/2$. Then, as $\alpha\to\infty$, the random set $\alpha^{-1/(2\beta-d)}P_{d,\alpha,\beta}$ converges weakly on the space of compact convex subsets of $\RR^d$ equipped with the Hausdorff metric to the  Poisson polytope $\conv \chi_{d,1,\beta}$.
\end{theorem}

Continuity arguments similar to those in~\cite{KabluchkoMarynychTemesvariThaele} allow to conclude from Theorem~\ref{thm:PolytopesConverge} the convergence of the $f$-vector of the beta$^*$ polytope $P_{d,\alpha,\beta}$ to that of the Poisson polytope $\conv\chi_{d,1,\beta}$, both in distribution and in expectation. For example, it is possible to show that
$$
(f_k(P_{d,\alpha,\beta}))_{k=0}^{d-1}
\overset{d}{\underset{\alpha\to\infty}\longrightarrow}
(f_k(\conv\chi_{d,1,\beta}))_{k=0}^{d-1}
\quad
\text{ in distribution, }
\quad
\beta \geq \frac{d+1}2.
$$
We choose a different method based on the explicit formula of Theorem~\ref{thm:f-vector_beta_star}  and present a result in which also the speed of convergence of the expected $f$-vector to its limit is addressed.
%However, the explicit description of $\EE f_k(P_{d,\alpha,\beta})$ as presented in Theorem \ref{thm:f-vector_beta_star} yields the following stronger result

\begin{theorem}[Expected $f$-vector for large intensities]\label{thm:convergence_beta_star}
	Fix $k\in\{0,\ldots,d-1\}$ and suppose that $\beta\ge (d+1)/2$. Then, as $\alpha\to\infty$,
	\begin{align*}
		\EE f_k(P_{d,\alpha,\beta})
		%&	=2\sum_{s=0}^\infty\tilde{\mathbb I}_{\infty,d-2s}(2\beta-d)\cdot\tilde{\mathbb{J}}_{d-2s,k+1}\Big(\beta-s-\frac 12\Big) + O(\alpha^{...})\\
		&	= \EE f_k(\conv \chi_{d,1,\beta})+ C_{d,k,\beta}\alpha^{-{2\over 2\beta-d}}+O(\alpha^{-{4\over 2\beta-d}}),
	\end{align*}
where $C_{d,k,\beta}$ is a constant which may be given explicitly and only depends on the parameters $d$, $k$ and $\beta$.
%where $\chi_{d,2\beta-d}$ is a Poisson process on $\RR^d\backslash\{0\}$ with density $\|x\|^{-2\beta}$ and $c_{d,\beta}>0$ is a constant only depending on $d$ and $\beta$.
\end{theorem}
%\begin{remark}
%%%% ZK: Ich glaube nicht, dass die Formel für die Konstante $c_{d,\beta}$ stimmt. Es fehlt der J-Term. Deshalb habe ich die Bemerkung abgekürzt.
%It is in principle possible to compute the precise value of $c_{d,\beta}$. In fact, from Theorem \ref{thm:f-vector_beta_star} and Proposition~\ref{prop:asymptotics_I_star} below it follows that $c_{d,\beta}=2\sum_{s=0}^\infty K(2\beta-d,d-2s)$ with the constants $K(2\beta-d,d-2s)$ given by $\eqref{eq:const(l,n)}$. However, since the exact expression is rather involved, we refrain from presenting the formula for $c_{d,\beta}$. \ZK{Ich glaube, wir sollten diese Konstante umbenennen.}
%\end{remark}
\begin{proof}%[Proof of Theorem~\ref{thm:convergence_beta_star}.]
Fix some $k\in\{0,\ldots,d-1\}$. We recall from Theorem~\ref{thm:f-vector_beta_star}  that
$$
\EE f_k(P_ {d,\alpha,\beta})=2\sum_{s=0}^{\lfloor (d-k-1)/2 \rfloor}   \mathbb I^*_{\alpha,d-2s}(2\beta-d)\cdot\tilde {\mathbb J}_{d-2s,k+1}\Big(\beta-s-\frac 12\Big),
$$
where in the case $\beta= (d+1)/2$ we assume that $\alpha$ is sufficiently large.
On the other hand,  $\EE f_k(\conv \chi_{d,1,\beta})$  has been computed in~\cite[Theorem~1.21]{KabluchkoThaeleZaporozhets} and is given by
\begin{align*}
\EE f_k(\conv\chi_{d,1,\beta})
&=
2\sum_{s=0}^{\lfloor (d-k-1)/2 \rfloor} \tilde{\mathbb I}_{\infty,d-2s}(2\beta-d)\cdot\tilde{\mathbb J}_{d-2s,k+1}\Big(\beta-s-\frac 12\Big),
\end{align*}
where
\begin{align*}
\tilde{\mathbb I}_{\infty,n}(\lambda)=\frac{\lambda^{n-1}}{n}\frac{\tilde c_{1,\frac{\lambda n+1}{2}}}{(\tilde c_{1,\frac{\lambda+1}{2}})^n}.
\end{align*}
The notation $\tilde{\mathbb I}_{\infty,n}(\lambda)$, as already introduced in~\cite{KabluchkoAngles}, is motivated by the fact that
$\tilde{\mathbb I}_{\infty,n}(\lambda)$ is the limit of the quantities $\tilde{\mathbb I}_{\alpha,n}(\lambda)$ as $\alpha\to\infty$, see the proof of Theorem~1.21 in~\cite{KabluchkoThaeleZaporozhets}, as well as of the quantities $\mathbb I^*_{\alpha,n}(\lambda)$, see the following Proposition~\ref{prop:asymptotics_I_star}. This proposition immediately yields the desired asymptotic expansion in Theorem~\ref{thm:convergence_beta_star}. Its   proof is postponed to Section~\ref{sec:large_intensities_proofs}.  We refrain from presenting an explicit formula for the constant $C_{d,k,\beta}$ since it is rather involved.
\end{proof}

\begin{proposition} \label{prop:asymptotics_I_star}
Suppose that $\lambda\ge 1$ and $n\in\NN$. Then, as $\alpha\to\infty$, we have the asymptotic expansion
\begin{align*}
\mathbb I^*_{\alpha,n}(\lambda)=\frac{\lambda^{n-1}}{n}\frac{\tilde c_{1,\frac{\lambda n+1}{2}}}{(\tilde c_{1,\frac{\lambda+1}{2}})^n}+\frac{K_1(\lambda,n)}{\alpha^{\frac 2\lambda}}+O\big(\alpha^{-\frac 4\lambda}\big) = \tilde{\mathbb I}_{\infty,n}(\lambda)+\frac{K_1(\lambda,n)}{\alpha^{\frac 2\lambda}}+O\big(\alpha^{-\frac 4\lambda}\big),
\end{align*}
where the formula for the constant $K_1(\lambda,n)$ can be found in~\eqref{eq:const(l,n)} below.
\end{proposition}

Let us consider the special cases $\beta= (d+1)/2$ and $\beta=d$. From~\cite[Theorem~1.23]{KabluchkoThaeleZaporozhets} it follows that one can identify the Euclidean Poisson zero cell and the Euclidean typical Poisson-Voronoi cell with the convex duals of $\conv \chi_{d, \mu, (d+1)/2}$ and $\conv \chi_{d,\mu, d}$, respectively, for suitably chosen intensities $\mu$ that are responsible for rescaling only and do not influence the $f$-vector. It follows from Theorem~\ref{thm:convergence_beta_star}, combined with~\eqref{eq:duality_zero_cell_beta_star} and~\eqref{eq:k-faces_typ_Voronoi}, that the expected numbers of $k$-faces of $Z_{d,\lambda}^0$ and $V^{\text{typ}}_{d,\lambda}$ converge, as $\lambda\to\infty$, to the expected numbers of $k$-faces of the corresponding Euclidean objects, which are shown in Figures~\ref{fig:VertexNumberHyperplaneMosaic} and~\ref{fig:VertexNumber} as dashed horizontal lines.  Moreover, the speed of convergence for the typical Poisson-Voronoi cell is $\text{const}(d,k)\cdot  \lambda^{-2/d}  + O(\lambda^{-4/d})$. The distributional convergence of the $f$-vectors can also be obtained by following the approach of~\cite{KabluchkoMarynychTemesvariThaele}.

\medspace

Numerical simulations (see, in particular Figures~\ref{fig:VertexNumber} and~\ref{fig:VertexNumberHyperplaneMosaic}) suggest that $\EE f_k (P_{d,\alpha,\beta})$ is a strictly decreasing function of $\alpha$. Together with the results of the present section, this would imply that the hyperbolic typical Voronoi cell and the hyperbolic Poisson zero cell are expected to have \textit{more} $k$-dimensional faces than their Euclidean analogues, for all $k\in \{0,\ldots, d-1\}$.  Our next theorem shows that this is indeed the case. On the other hand, let us remark that in the spherical setting, the expected $f$-vectors are strictly increasing functions of the number of points (or, in the Poisson process setting, intensity); see~\cite{Bonnet_etal} and~\cite[Theorem~1.5]{KabluchkoThaeleZaporozhets}. In particular, spherical cells are expected to have \textit{less} $k$-faces than their Euclidean analogues.

\begin{theorem}[Monotonicity of the expected $f$-vector]\label{thm:Monotonicity}
Fix $k\in\{0,\ldots,d-1\}$ and $\beta\geq (d+1)/2$. Then the function $\alpha \mapsto \EE f_k(P_{d,\alpha,\beta})$
%\begin{align*}
%\alpha &\mapsto \EE f_k(P_{d,\alpha,\beta})\\
%((d-1)\pi,\infty) & \to (0,\infty)\qquad\qquad\qquad\text{if }\beta={d+1\over 2}\\
%(0,\infty) & \to (0,\infty)\qquad\qquad\qquad\text{if }\beta>{d+1\over 2}
%\end{align*}
is strictly monotone decreasing in the range $\alpha>0$ (if $\beta>(d+1)/2$) or $\alpha>(d-1)\pi$ (if $\beta=(d+1)/2$).
\end{theorem}

\subsection{Stochastic geometry in de Sitter space}

In this section we shall present hyperbolic counterparts of the results on  random polytopes in a half-sphere obtained in~\cite{BaranyHugReitznerSchneider,Bonnet_etal,KabluchkoZeroPolytope,KabluchkoMarynychTemesvariThaele}. Our approach will provide an interpretation of beta$^*$ polytopes as quite natural random objects in the de Sitter half-space (whose points are related by duality to hyperplanes in $\HH^d$, as we shall see).   Let us begin by recalling the results on the $d$-dimensional upper half-sphere in $\RR^{d+1}$, which is denoted by
$$
\SS^d_+:=\{x = (x_0,\ldots,x_{d})\in \RR^{d+1}: x_0\geq 0, x_0^2 + x_1^2 + \ldots + x_d^2 = 1\}.
$$
Let $U_1,\ldots,U_n$ be random vectors drawn uniformly and independently from $\SS^d_+$. The polyhedral convex cone generated by these vectors (also called their positive hull) is denoted by
$$
C_{d,n} = \pos(U_1,\ldots,U_n) = \left\{\lambda_1 U_1 + \ldots + \lambda_n U_n: \lambda_1,\ldots,\lambda_n\geq 0\right\}.
$$
The study of the random cone $C_{d,n}$ and the random spherical polytope $C_{d,n}\cap \SS^{d}_+$ has been initiated in the work of B\'ar\'any, Hug, Reitzner and Schneider~\cite{BaranyHugReitznerSchneider} and continued in~\cite{Bonnet_etal,KabluchkoMarynychTemesvariThaele} and~\cite{KabluchkoZeroPolytope}. It is convenient to replace $C_{d,n}$ by its horizontal cross-section $P := C_{d,n} \cap \{x_0=1\}$ as follows.  By definition, the gnomonic projection $\Pi_{\text{gn}}: \{x_0>0\} \to \RR^d$ maps a point $x\in \RR^{d+1}$ with $x_0>0$ to the intersection of the ray spanned by $x$ with the hyperplane $\{x_0=1\}$, that is
\begin{equation}\label{eq:de_sitter_gnomonic}
\Pi_{\text{gn}} (x_0,x_1,\ldots,x_d) = \left(\frac{x_1}{x_0},\ldots, \frac {x_d}{x_0}\right),
\end{equation}
see Figure \ref{fig:gnomproj}, left panel. Here and in the following  we  identify the hyperplane  $\{x_0=1\}$ with $\RR^d$ via the map $(1, v) \mapsto v$, $v\in \RR^d$.
It is well known that the gnomonic projection maps the uniform distribution on $\SS^d_+$ to the $d$-dimensional Cauchy distribution (which is a special case of the beta' distribution) with the density
\begin{equation}\label{eq:def_cauchy}
\tilde f (v)=\tilde{c}_{d,\frac{d+1}{2}} ( 1+\|v\|^2)^{-\frac{d+1}2},
\qquad
v\in\RR^d.
\end{equation}
Hence, the horizontal cross-section $P = C_{d,n} \cap\{x_0=1\}$ is a beta' polytope  in $\{x_0=1\} \equiv \RR^d$, that is the convex hull of $n$ independent Cauchy-distributed random  points $V_1:= \Pi_{\text{gn}}(U_1),\ldots, V_n:= \Pi_{\text{gn}}(U_n)$.  This observation from~\cite{KabluchkoMarynychTemesvariThaele} makes it possible to derive explicit formulae for $\EE f_k(C_{d,n})$, the expected  number of $k$-dimensional faces, in terms of the quantities $A[n,k]$ which already appeared in Section~\ref{sec:Hyperplane} and a certain array $B\{n,k\}$; see~\cite[Theorem~2.2]{KabluchkoZeroPolytope} and also~\cite[Theorem~1.1]{Kabluchko_SimplifiedFormulas}. Further quantities of interest include the solid angle $\measuredangle (C_{d,n})$ and the Grassmann angles $\gamma_{\ell}(C_{d,n})$ with $\ell\in \{1,\ldots, d\}$, which may be defined by
\begin{align}
\measuredangle(C_{d,n})
&=
\frac 12 \int\limits_{P} \tilde f (v) \dint v
=
\frac 12 \int\limits_{P} \tilde{c}_{d,\frac{d+1}{2}} ( 1+\|v\|^2)^{-\frac{d+1}2} \,\dint v,
\label{eq:angle_def_cauchy}\\
\gamma_{\ell}(C_{d,n})
&=
\int\limits_{(\RR^d)^{d+1-\ell}}  \ind_{\{P \cap  \aff (v_1,\ldots, v_{d+1-\ell}) \neq \varnothing\}}  \tilde f (v_1) \ldots \tilde f (v_{d+1-\ell}) \,\dint v_1 \ldots \dint v_{d+1-\ell}.\label{eq:angle_grassmann_def_cauchy}
\end{align}
The expectations of both quantities satisfy Efron-type identities, see~\cite[Equation~(4.13)]{KabluchkoZeroPolytope} and~\cite[Theorem~2.7]{KabluchkoMarynychTemesvariThaele}, which allows to express them through the expected $f$-vector.

\begin{figure}[t]
	\centering
	\begin{minipage}{0.45\textwidth}
			\begin{tikzpicture}
			\clip (-1,0) rectangle (5.5,4);
			\node at (2,3) {\includegraphics[width=\textwidth]{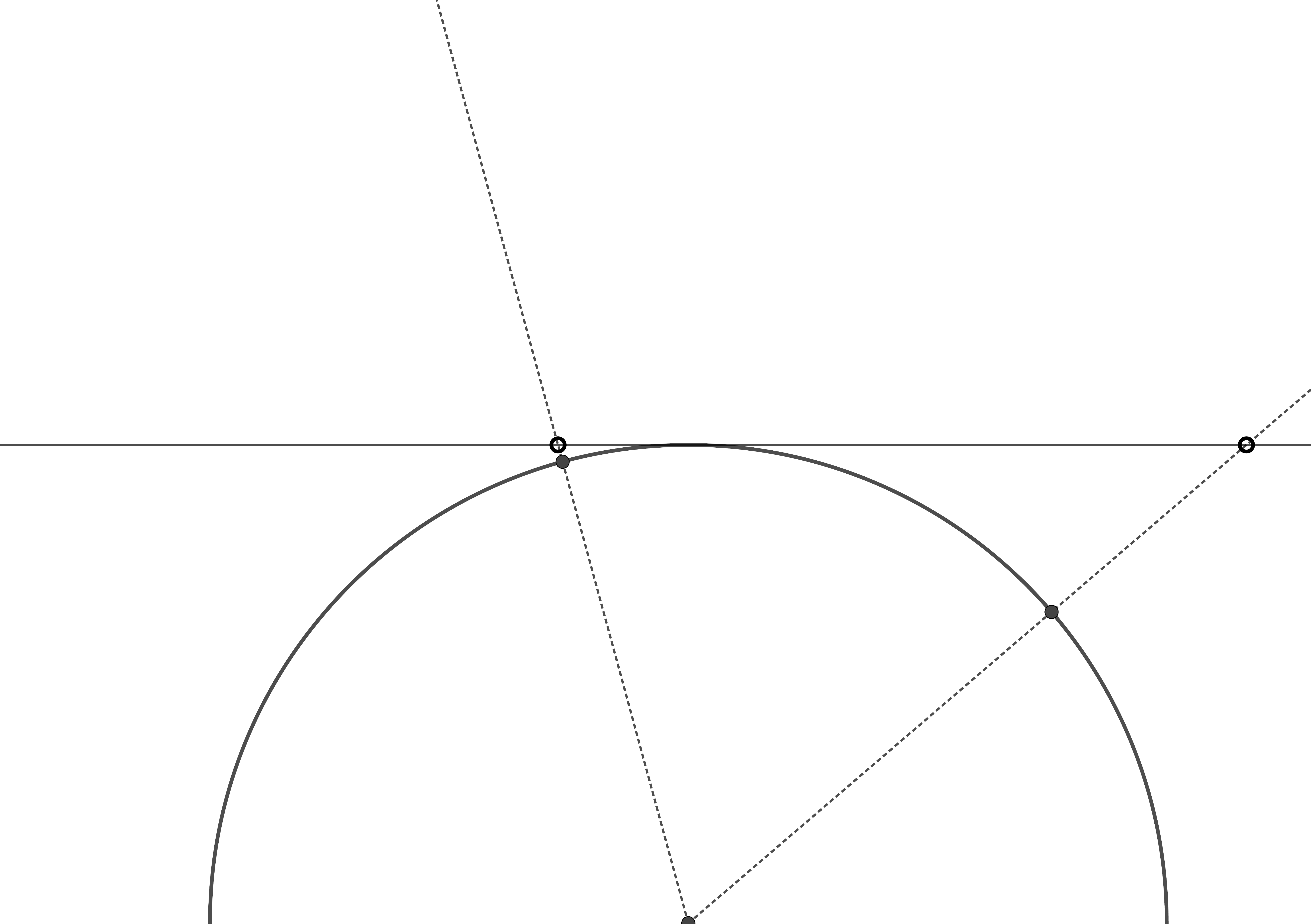}};
					\node at (-0.1,0.55) {\footnotesize $\mathbb{S}_+^d$};
					\node at (0,3.35) {\footnotesize $\mathbb{R}^d$};
					\node at (2,3.3) {\footnotesize $\Pi_{\rm gn}(x)$};
					\node at (4.7,3.3) {\footnotesize $\Pi_{\rm gn}(y)$};
					\node at (1.7,2.9) {\footnotesize $x$};
					\node at (4.2,1.9) {\footnotesize $y$};

		\end{tikzpicture}
	\end{minipage}
	\qquad
	\begin{minipage}{0.45\textwidth}
	\begin{tikzpicture}
		\clip (-1,0) rectangle (5.5,4);
		\node at (2,3) {\includegraphics[width=\textwidth]{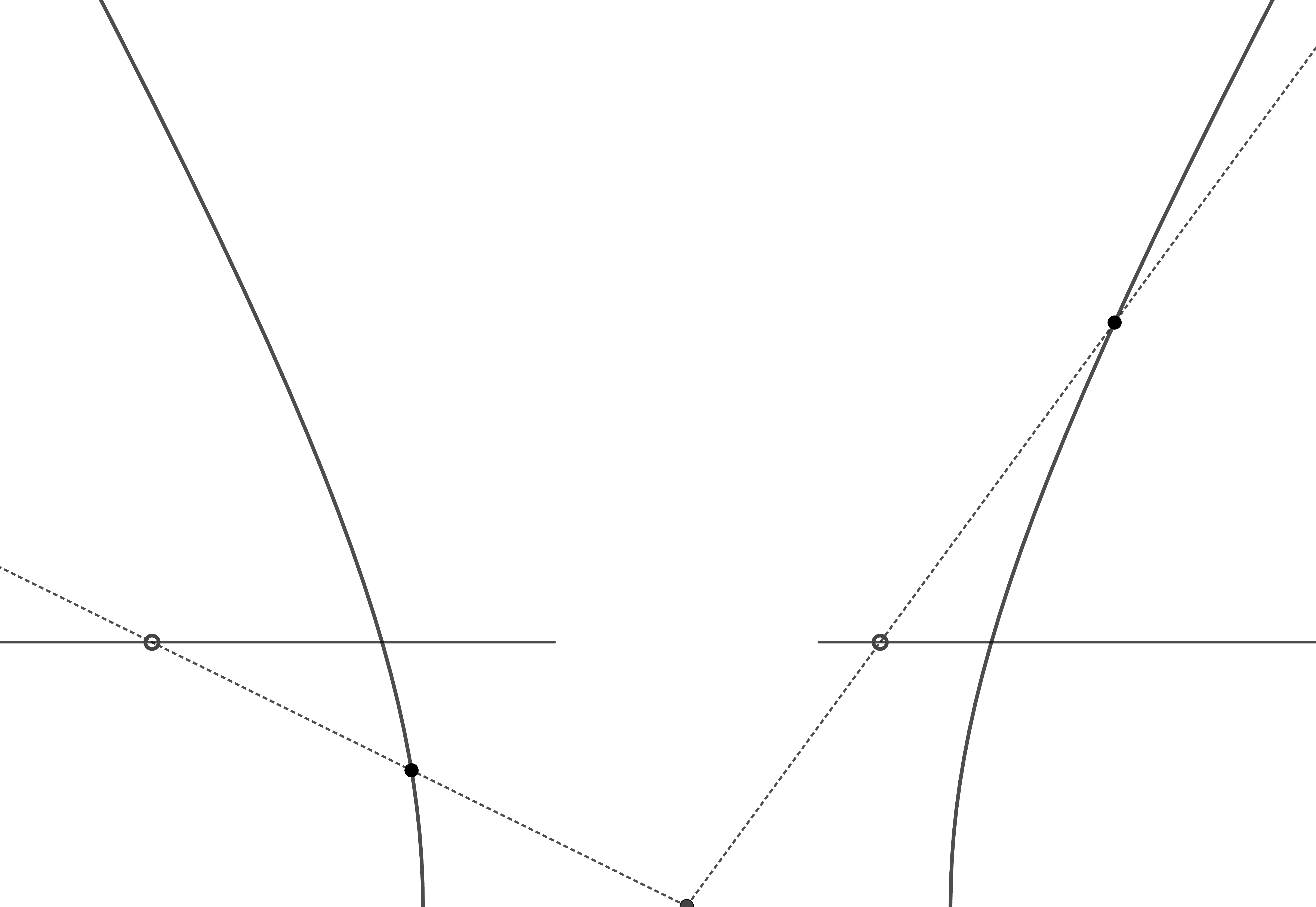}};	
				\node at (4,0.6) {\footnotesize $\mathbb{L}_+^d$};
				\node at (1.2,2.2) {\footnotesize $\mathbb{R}^d\setminus\widebar\BB^d$};
				\node at (-0.5,2.2) {\footnotesize $\Pi_{\rm gn}(x)$};
				\node at (2.8,2.2) {\footnotesize $\Pi_{\rm gn}(y)$};
				\node at (0.5,1.1) {\footnotesize $x$};
				\node at (4.6,3.5) {\footnotesize $y$};
		
	\end{tikzpicture}
\end{minipage}
	\caption{Illustration of the gnomonic projections $\Pi_{\rm gn}:\SS_+^d\to\RR^d$ and $\Pi_{\rm gn}:\LL_+^d\to\RR^d\setminus\bar\BB^d$ on the half-sphere $\SS^{d-1}_+$ (left panel) and the de Sitter half-space $\LL_+^{d}$ (right panel) for $d=1$.}
	\label{fig:gnomproj}
\end{figure}

%it through $\EE f_1(C_{d,n+1})$.  More generally, formulas for the expected Grassmann angles of $C_{d,n}$ can be obtained via.

%Formula for the Poissonized cone. Problem: n can be smaller than d and I don't know whether the formula for the expected f-vector  still holds.
%$$
%\EE f_{\ell - 1} (C_{d,\text{Poi}(\lambda)})
%=
%\frac{\pi^{\ell}}{\ell!} \sum_{\substack{m \in \{\ell,\ldots, d\}\\ m \equiv d \Mod 2}} \frac{A[m,\ell] - A[m-2,\ell]}{(m-1)!} %\left(\frac{\lambda}{\pi}\right)^m \int_0^\pi (\sin y)^{m-1} e^{-\lambda/\pi} \dint y.
%$$

Let us now turn to the hyperbolic analogues of these results. We shall work in the space $\RR^{d+1}$. Recall that the  Minkowski product of two vectors $x= (x_0,x_1,\ldots, x_d)$ and $x' = (x_0',x_1',\ldots,x_d')$ in $\RR^{d+1}$ is defined by
$$
B(x,x') = x_0 x_0' - x_1 x_1' - \ldots - x_d x_d'.
$$
The \textit{de Sitter space} $\LL^d$ is the hyperboloid defined by the equation $B(x,x) = - 1$, that is
$$
\LL^d = \{x = (x_0,x_1,\ldots,x_d)\in \RR^{d+1}: x_0^2 - x_1^2 - \ldots - x_d^2 = -1 \}.
$$
For comparison, recall that the hyperboloid model $\HH^d$ of the hyperbolic geometry is defined by the equation $B(x,x)= 1$ together with $x_0>0$.  The Minkowski product
induces on $\LL^d$ a pseudo-Riemannian metric which is invariant under the action of the Lorentz group $O(1,d)$ on $\LL^d$. Even though this metric is not positive-definite, it defines a volume measure on $\LL^d$ as follows. If, in some local coordinates $z_1,\ldots, z_d$, the pseudo-Riemannian metric is given by a matrix $g_{ij}(z)$, then the corresponding volume element is given by
$$
\dint \nu_{\text{dS}}(\dint z_1,\ldots, \dint z_d) = |\det (g_{ij}(z))|^{1/2} \, \dint z_1\ldots \dint z_d.
$$

In our setting, the de Sitter space is the analogue of the unit sphere $\SS^d$, while the analogue of the upper half-sphere $\SS^{d}_+$ is the \textit{upper de Sitter half-space} defined by
$$
\LL^d_+ = \{x \in \LL^d: x_0>0\}.
$$
We can identify $\LL^d_+$ with its ``Klein model'' $\RR^d\backslash \widebar\BB^d$ via the gnomonic projection $\Pi_{\text{gn}}: \LL^d_+ \to \RR^d \backslash \widebar\BB^d$ defined in the same way as in~\eqref{eq:de_sitter_gnomonic}, see Figure \ref{fig:gnomproj}, right panel.
%This identification is very similar to the identification between the hyperboloid and the Klein models $\HH^d$ and $\BB^d$.
The next simple proposition shows that the pseudo-Riemannian metric and the volume element induced on $\RR^d \backslash \widebar\BB^d$ by this identification are very similar to what is known for the Klein model of hyperbolic geometry; see~\eqref{eq:metric_volume_klein}.
\begin{proposition}\label{prop:metric_de_sitter}
The images of the pseudo-Riemannian metric and the corresponding volume element of the upper de Sitter half-space under the gnomonic projection are given by
\begin{equation}\label{eq:metric_volume_de_sitter}
\dint s^2 =  \frac{\|\dint u\|^2}{1- \| u \|^2} +   \frac{\langle u, \dint u \rangle^2}{(\| u \|^2 - 1)^2},
\quad
\nu_{\text{dS}}(\dint u_1,\ldots, \dint u_d)  = \frac{\dint u_1\ldots \dint u_d}{(\|u\|^2 - 1)^{\frac{d+1}2}},
\quad
u\in \RR^d\backslash \widebar\BB^d.
\end{equation}
\end{proposition}
\begin{proof}
The inverse of the gnomonic projection is given by
$$
\Pi^{-1}_{\text{gn}} (u) = \left(\frac{1}{\sqrt{\|u\|^2-1}}, \frac{u}{\sqrt{\|u\|^2-1}} \right),
\qquad
u\in \RR^d\backslash\widebar\BB^d.
$$
Given this, one checks that the differential of the map $\Pi^{-1}_{\text{gn}}$ is given by
$$
\dint \Pi^{-1}_{\text{gn}}(u)
=
\Pi^{-1}_{\text{gn}}(u + \dint u) - \Pi^{-1}_{\text{gn}}(u)
=
\left(
 -\frac{\langle u, \dint u \rangle}{(\|u\|^2-1)^{3/2}},
\frac{ (\|u\|^2-1)\dint u - \langle u, \dint u \rangle u}{(\|u\|^2-1)^{3/2}}
\right).
%+
%o(\dint u).
$$
Computing the squared Minkowski pseude-norm $B(\dint \Pi^{-1}_{\text{gn}}(u),\dint \Pi^{-1}_{\text{gn}}(u))$ of the infinitesimal vector on the right-hand side, we arrive at the formula for $\dint s^2$ stated in~\eqref{eq:metric_volume_de_sitter}. Let us now turn to the volume element. The matrix of the pseudo-Riemannian metric $\dint s^2 = \sum_{i,j=1}^d g_{ij}\dint u_i \dint u_j$ in the coordinates $u_1,\ldots, u_d$ is therefore given by
$$
g_{ij} = g_{ij}(u) = \frac{u_i u_j}{(\|u\|^2-1)^2} - \frac{\delta_{ij}}{\|u\|^2-1},
\qquad
i,j=1,\ldots, d,
$$
where $\delta_{ij}$ denotes the Kronecker delta. To compute the determinant of this matrix, we denote by $I_{d}$ the $(d\times d)$-identity matrix, and let $J:\RR^d \to \RR^d$ be an orthogonal projection on the line spanned by $u$. Then, using the shorthand  $p:= 1/(\|u\|^2-1)$, we can write
$$
\det (g_{ij})_{i,j=1}^d
=
p^d \det (p u_iu_j - \delta_{ij})
=
p^d \det (p \|u\|^2 J - I_{d\times d})
=
(-p)^d (1 - p \|u\|^2)
=
\frac{(-1)^{d+1}}{(1 -\|u\|^2)^{d+1}}.
$$
It follows that the volume element is given by the formula stated in~\eqref{eq:metric_volume_de_sitter}.
\end{proof}

Consider now a Poisson process on the upper de Sitter half-space $\LL^d_+$ whose intensity measure is given by the standard volume element $\dint V$ times $\alpha  \tilde c_{d,(d+1)/2}$, where $\alpha > 0$ is some constant. Denote the atoms of this point process (listed in some order) by $U_1^*,U_2^*,\ldots$ and consider the convex cone $C_{d,\alpha}^*= \pos (U_1^*,U_2^*,\ldots)\subset \RR^{d+1}$ spanned by these vectors.
\begin{proposition}
The horizontal cross-section $P^* := C^*_{d,\alpha} \cap \{x_0=1\}$ of $C^*_{d,\alpha}$ has the same distribution as the beta$^*$ polytope $P_{d, \alpha, (d+1)/2}$.
\end{proposition}
\begin{proof}
Clearly, $C^*_{d,\alpha} \cap \{x_0=1\}$ is the convex hull of the points $V_1^*:=\Pi_{\text{gn}}(U_1^*), V_2^*:=\Pi_{\text{gn}}(U_2^*),\ldots$ According to Proposition~\ref{prop:metric_de_sitter}, the points $V_1^*,V_2^*,\ldots$ form a Poisson process on $\RR^d \backslash \widebar\BB^d$ with intensity $\alpha \tilde c_{d,(d+1)/2} (\|u\|^2 - 1)^{-(d+1)/2}$, and the claim follows.
\end{proof}

Let $\zeta := \sum_{j=1}^\infty \delta_{V_j^*}$ be the point process on $\RR^d\equiv \{x_0=1\}$ formed by the points $V_1^*,V_2^*,\ldots$. Note that $\zeta$ has the same law as $\zeta_{d,\alpha, (d+1)/2}$ and its intensity is $f_{d,\alpha, (d+1)/2}$; see Section~\ref{subsec:def_and_existence_of_beta_star}.
Motivated by~\eqref{eq:angle_def_cauchy}, the analogue of the solid angle for the cone $C^*_{d,\alpha}$ is defined by
%For a convex set $P\subset \RR^d$ containing $\widebar\BB^d$ in its interior, we define
$$
\measuredangle^*(C^*_{d,\alpha})
:=
\frac 12 \int\limits_{\RR^d \backslash P^*}   f_{d,1,\frac{d+1}2} (v) \dint v
=
\frac 12 \int\limits_{\RR^d \backslash P^*}   \tilde c_{d,\frac{d+1}{2}}(\|v\|^2-1)^{-\frac{d+1}{2}} \dint v.
$$
Note that in contrast to~\eqref{eq:angle_def_cauchy} the integral is taken over $\RR^d \backslash P^*$. For $\alpha > (d-1)\pi$ the integral over $P^*$ is a.s.\ infinite because $P^* \supset r \BB^d$ with some random $r>1$ by Theorem~\ref{thm:PolytopeOrNot}.
The next proposition is an Efron-type identity which, combined with Theorem~\ref{thm:f-vector,beta=(d+1)/2}, yields an explicit formula for $\EE \measuredangle^*(P)$.
\begin{proposition}
We have $\EE \measuredangle^*(C^*_{d,\alpha}) = \frac {1}{2\alpha} \EE f_1(C^*_{d,\alpha})$.
\end{proposition}
\begin{proof}
By Mecke's formula, see Proposition~\ref{prop:Slivnyak-Mecke},  we have
\begin{align*}
\EE f_0(P^*)
&=
\EE \sum_{v\in \zeta} \ind_{\{v\in \cF_0(P^*)\}}
=
\EE \sum_{v\in \zeta} \ind_{\{v\notin \conv (\zeta \backslash\{v\})\}}
=
\int\limits_{\RR^d \backslash \widebar\BB^d} \PP[v\notin \conv \zeta] f_{d,\alpha, \frac{d+1}{2}}(v)\,\dint v\\
&=
\EE \int\limits_{\RR^d \backslash P^*} f_{d,\alpha, \frac{d+1}{2}}(v)\, \dint v
=
2 \alpha \cdot \EE \measuredangle^*(C^*_{d,\alpha}),
\end{align*}
and the claim follows since $f_0(P^*) = f_1(C^*_{d,\alpha})$.
\end{proof}

To generalize the previous result, we define the analogues of the Grassmann angles (see~\eqref{eq:angle_grassmann_def_cauchy}) in the de Sitter space as follows.  For $\ell\in \{1,\ldots,d\}$, the functional $\gamma_{\ell}^*$ of the cone $C_{d,\alpha}^*$ is defined by
%in the upper de Sitter half-space
%(which is identified with $\RR^d\backslash\widebar\BB^d$)
%as follows.
%The analogue Grassmann angle
%For a convex set $P^*\subset \RR^d$ containing $\widebar\BB^d$, the Grassmann angle $\gamma_{\ell}^*(P)$ is defined by
$$
\gamma_\ell^* (C_{d,\alpha}^*) = \int\limits_{(\RR^d \backslash\widebar\BB^d)^{d+1-\ell}}  \ind_{\{P^* \cap  \aff (v_1,\ldots, v_{d+1-\ell}) = \varnothing\}}  f_{d,1,\frac{d+1}2} (v_1) \ldots f_{d,1,\frac{d+1}2} (v_{d+1-\ell}) \,\dint v_1 \ldots \dint v_{d+1-\ell}.
$$
Let us stress that we have ``$= \varnothing$'' in the indicator function, which is  in contrast to~\eqref{eq:angle_grassmann_def_cauchy}.
\begin{proposition}
For all $\ell\in \{1,\ldots,d\}$ we have
$\EE \gamma_\ell^* (C_{d,\alpha}^*) = \alpha^{\ell - d - 1} \EE f_{d-\ell+1}(C_{d,\alpha}^*)$.
\end{proposition}
\begin{proof}
Let $k:= d+1-\ell \in \{1,\ldots,d\}$.
Then, by the Mecke formula stated in Proposition~\ref{prop:Slivnyak-Mecke},
\begin{align*}
\EE f_{k-1}(P^*)
&=
\EE \sum_{(v_1,\ldots, v_k)\in \zeta_{\neq}^k} \ind_{\{\conv(v_1,\ldots,v_k) \in \cF_{k-1}(P^*)\}} \\
&=
\int\limits_{(\RR^d \backslash \widebar\BB^d)^k}
 \PP[\conv(v_1,\ldots,v_k) \in \cF_{k-1}(\conv (\zeta \cup \{v_1,\ldots,v_k\}))]  \prod_{j=1}^k f_{d,\alpha, (d+1)/2}(v_j)\, \dint v_j\\
&=
\int\limits_{(\RR^d \backslash \widebar\BB^d)^k}
 \PP[\aff(v_1,\ldots,v_k) \cap P^*  = \varnothing]
 \prod_{j=1}^k f_{d,\alpha, (d+1)/2}(v_j) \,\dint v_j.
 % f(v_1) \ldots f(v_k) \dint v_1 \ldots \dint v_k,
\end{align*}
and the claim follows since $f_{d,\alpha, (d+1)/2}(v_j) =\alpha f_{d,1, (d+1)/2}(v_j)$.
\end{proof}

\begin{remark}\label{rem:invariance_hyperplane_measure}
There is a duality between the de Sitter geometry in $\RR^d \backslash \widebar\BB^d$ and the hyperbolic geometry in the Klein model $\BB^d$. Namely, to a vector $x \in \LL^d_+$ we assign the hyperplane $\{y\in \RR^{d+1}: B(x,y) = 0\}\cap \HH^d$  in the hyperboloid model. Under gnomonic projection, this correspondence takes the following form: to a point $v\in \RR^d \backslash \widebar\BB^d$ in the ``Klein model'' of the  de Sitter half-space we assign the hyperplane $\{z\in \BB^d: \langle z,v\rangle = 1\}\cap \BB^d$ in the Klein model of the hyperbolic space. Under this correspondence (which is the usual polarity w.r.t.\ the unit sphere), the volume measure $\nu_{\text{dS}}$ on $\RR^d \backslash \widebar\BB^d$, see~\eqref{eq:metric_volume_de_sitter},  is mapped to some measure on the space of hyperplanes in the Klein model. The latter measure is invariant under isometries since so is the volume element $\dint V$. In fact, the proof of Theorem~\ref{thm:ZeroCell}, with $\beta = (d+1)/2$, shows that this measure coincides with the measure  $\mu_{d, \omega_d, (d+1)/2}$ defined in Section~\ref{subsubsec:hyperbolic_hyperplane_tess_def}, thus proving its invariance under  isometries.
\end{remark}

\subsection{The case \texorpdfstring{$\beta=(d+2)/2$}{beta=(d+2)/2}}\label{sec:d+2/2}

Finally, we discuss the case where $\beta=(d+2)/2$, which we included since the corresponding terms in the $f$-vector of $P_{d,\alpha,(d+2)/2}$ become explicit.
\begin{theorem} \label{thm:f-vector,beta=(d+2)/2}
Let $d\in\NN$,  $\alpha>0$ and $\ell\in\{1,\ldots,d\}$. Then
\begin{equation*}
\EE f_{\ell-1}(P_{d,\alpha,\frac{d+2}{2}})
=\sqrt{\alpha e^{\alpha} \over\pi}\,
\sum_{\substack{m \in \{\ell,\ldots, d\}\\ m \equiv d \Mod 2}}
K_{m-\frac 12}\left(\frac \alpha 2\right)
\left(\binom{m}{\ell}\binom{m+\ell}{\ell}-\binom{m-2}{\ell}\binom{m+\ell-2}{\ell} \right),
\end{equation*}
where $K_\nu(z)$ denotes the modified Bessel function of the second kind which may be defined~\cite[pp.~181, 172]{Watson} by
\begin{align}\label{eq:def_Bessel_2nd_kind}
K_\nu(z)
=
\int\limits_0^\infty\cosh (\nu t) \cdot e^{-z\cosh t}\,\dint t
=
\frac{\sqrt{\pi} (z/2)^\nu}{\Gamma(\nu+\frac 12)}\int\limits_1^\infty e^{-zt} (t^2-1)^{\nu - \frac 12} \dint t
,
\quad
z,\nu>0.
\end{align}
\end{theorem}

Let us now provide a geometric interpretation of the special case $\beta= (d+2)/2$; see~\cite[Lemma~2.3(c)]{BesauEtAl} for an analogous construction in spherical geometry. Consider an infinite measure $\nu_{\text{dS}}'$ on de Sitter half-space $\LL_+^d$ which is obtained from the standard Lebesgue measure on $\RR^d \backslash \widebar\BB^d$ by lifting it via the map $v\mapsto (\sqrt{\|v\|^2-1}, v)$. Let  $U_1', U_2',\ldots$ be the atoms of the Poisson process on $\LL^d_+$ with intensity measure $\nu_{\text{dS}}'$ and consider the cone $C_d':= \pos(U_1',U_2',\ldots)\subset \RR^{d+1}$.
\begin{proposition}
The horizontal cross-section $C'_{d} \cap \{x_0=1\}$ has the same distribution as the beta$^*$ polytope $P_{d, \alpha, (d+2)/2}$ with $\alpha= 1/\tilde c_{d,(d+2)/2}$.
\end{proposition}
\begin{proof}
The points $V_i':= \Pi_{\text{gn}}(U_i')$ form a Poisson process on $\RR^d \backslash \widebar\BB^d  \subset \{x_0=1\}$ whose intensity measure $m$ is the image of the Lebesgue measure under the involution $\Psi:\RR^d \backslash \widebar\BB^d \to \RR^d \backslash \widebar\BB^d$ given by $\Psi(v) = v/\sqrt{\|v\|^2-1}$. Clearly, $m$ is rotationally invariant and
$$
m(\RR^d \backslash R \BB^d) = \Vol \left( \frac{\BB^d}{\sqrt{1-R^{-2}}} \right) = \frac{\kappa_d}{(1-R^{-2})^{d/2}},
\qquad
R>1.
$$
Taking the derivative in $R$, and dividing by $\omega_d R^{d-1}$, we conclude that the Lebesgue density of $m$ is given by $v\mapsto (\|v\|^2-1)^{-(d+2)/2}$, $v\in \RR^d \backslash \widebar\BB^d$, which proves the claim.
\end{proof}

\section{Boundedness of cells and existence of \texorpdfstring{beta$^*$}{beta*} polytopes: Proofs} \label{sec:proof_PolytopeOrNot}

In this section we prove the results on the boundedness of hyperbolic zero cells stated in Theorems~\ref{thm:PolytopeOrNot_Tessellation} and~\ref{thm:PolytopeOrNot_critical_tessellation}. By duality, these imply conditions under which beta$^*$ sets are polytopes as stated in Theorems~\ref{thm:PolytopeOrNot} and~\ref{thm:PolytopeOrNot_critical}.
The proofs rely on classical results about random coverings of the unit sphere which we now recall.

\subsection{Coverings by arcs and caps}
Given a deterministic sequence $\ell_1, \ell_2, \ldots \in (0,1)$, consider random arcs $S_1,S_2,\ldots$ of lengths $2\pi \ell_1, 2\pi \ell_2,\ldots$ placed uniformly and  independently on a unit circle $\SS^1$. The question whether the whole circle is covered infinitely often by these arcs with probability $1$, i.e.\ whether $\PP[\limsup S_n = \SS^1] = 1$, goes back to Dvoretzky~\cite{dvoretzky} and has been studied in the works of Billard, Kahane, Erd\H{o}s,  Orey and Mandelbrot. This line of research culminated in the work of Shepp~\cite{shepp} who proved that the circle is covered infinitely often with probability $1$ if and only if
\begin{equation}\label{eq:shepp_crit}
\sum_{n=1}^\infty \frac 1 {n^2} e^{\ell_1+\ldots + \ell_n} = +\infty,
\end{equation}
provided the sequence $\ell_n$ is non-increasing.
Coverings by random sets in more general metric spaces have been studied in several works including~\cite{HoffmannJoergensen} and~\cite{FengEtAl}; see the latter paper for more pointers to the literature.
Let us recall a special case of a result of Hoffmann-J{\o}rgensen~\cite[Section~5]{HoffmannJoergensen}, for which we introduce some further notation. An open spherical cap in the unit sphere $\SS^{d-1}$ is a set of the form
$$
S (u,h) = \{x\in \SS^{d-1}: \langle u, x\rangle > h\},
$$
where $u\in \SS^{d-1}$  is called the centre of the cap and $h\in (0,1)$. Consider now a deterministic sequence of numbers $h_1,h_2,\ldots\in (0,1)$ and a sequence of independent points $u_1,u_2,\ldots$ drawn uniformly at random from the sphere $\SS^{d-1}$. We are interested in whether the caps $S(u_1,h_1), S(u_2,h_2),\ldots$ cover the sphere with probability $1$.
Recall that $\sigma_{d-1}$ denotes the spherical Lebesgue measure on $\SS^{d-1}$ normalized such that  $\sigma_{d-1}(\SS^{d-1}) = 1$.

\begin{theorem}[Hoffmann-J{\o}rgensen~\cite{HoffmannJoergensen}]\label{theo:hoffmann_joergensen}
Let the sequence $h_1,h_2,\ldots$ be such that
\begin{equation}\label{eq:hoffmann_joergensen_condition}
\lim_{n\to\infty} n \sigma_{d-1}(S (u_n,h_n)) = a \in [0,\infty].
\end{equation}
\begin{itemize}
\item [(i)] If $a>1$, then $\PP[\limsup S(u_n,h_n) = \SS^{d-1}]=1$. That is,  with probability $1$, the sphere is covered infinitely often by the caps.
\item [(ii)] If $a<1$, then $\PP[\limsup S(u_n,h_n) = \SS^{d-1}]=0$ and $\PP[\cup_{n=1}^\infty \overline{S(u_n,h_n)} \neq \SS^{d-1}]>0$.  In particular, the probability that there are points not covered by the closed caps is positive.
\end{itemize}
\end{theorem}
Note that the case $a=1$ from this theorem will be discussed in Remark~\ref{remark:doubly_critical_Hoffmann} below.
\begin{proof}
Claim~(i) follows from (5.3) in Section 5 of~\cite{HoffmannJoergensen} (note that our $\sigma_{d-1}(S (u_n,h_n))$ corresponds to $a_n^{p}$ there). Claim~(ii) follows from (5.5) in Section 5 of~\cite{HoffmannJoergensen} (which claims that $\limsup S(u_n,h_n) \neq  \SS^{d-1}$ a.s.) together with (3.10) on page~175 of~\cite{HoffmannJoergensen} (which claims that uncovered points exist with positive probability).
\end{proof}

\subsection{Proofs of Theorems~\ref{thm:PolytopeOrNot_Tessellation} and~\ref{thm:PolytopeOrNot}}
Our aim is to prove Theorem~\ref{thm:PolytopeOrNot_Tessellation}. Fix some $d\in\NN$, $\beta>\max(d/2,1)$ and $\lambda>0$. Recall that we are interested in the zero cell
\begin{equation}\label{eq:zero_cell_def_rep}
Z_{d,\lambda,\beta} = \{x\in \widebar\BB^d: \langle x, u_n\rangle \leq r_n \text{ for all } n\in \NN\} \subset \widebar\BB^d,
\end{equation}
where $u_1,u_2,\ldots$ are as above and $r_1<r_2<\ldots$ form an independent  Poisson process on $(0,1)$ with intensity
$$
f(r)
=
\lambda \frac{r^{2\beta - d - 1}}{(1-r^2)^{\beta}},
\qquad
0<r<1.
$$
In what follows, for two sequences $(a_n)_{n\in\NN}$ and $(b_n)_{n\in\NN}$ we write $a_n\sim b_n$, provided that $a_n/b_n\to 1$, as $n\to\infty$.

\begin{lemma}\label{lem:PPP_on_01_asympt_H_n}
With probability $1$, we have
$$
1-r_n \sim  \left(\frac{\lambda}{2^\beta (\beta-1)}\right)^{1/(\beta-1)} \cdot \frac 1 {n^{1/(\beta-1)}},\qquad
\text{ as }
n\to\infty.
$$
\end{lemma}
\begin{proof}
We would like to have a representation $r_n = \varphi(P_n)$, where $P_1<P_2<\ldots$ are the arrivals of a homogeneous, unit intensity Poisson process on $(0,\infty)$, and $\varphi:[0,\infty) \to [0,1)$ is a suitable monotone increasing function with $\varphi(0) = 0$ and $\lim_{y\to +\infty} \varphi(y) = 1$. By the well-known transformation property of Poisson processes, such representation holds if
\begin{equation}\label{eq:transform_poi_proc}
\int\limits_0^{\varphi(y)} \frac{\lambda r^{2\beta - d - 1}}{(1-r^2)^{\beta}} \dint r = y
\qquad
\text{ for all }
y\in [0,\infty).
\end{equation}
Using the L'Hospital rule, one easily checks that
$$
\int\limits_0^{z} \frac{\lambda r^{2\beta - d - 1}}{(1-r^2)^{\beta}} \dint r \sim \frac{\lambda (1-z)^{1-\beta}}{2^\beta (\beta-1)},
\qquad
\text{ as }
z\uparrow 1.
$$
Since $\varphi(y) \uparrow  1$ as $y\to  +\infty$, it follows that
$$
\frac{\lambda (1-\varphi(y))^{1-\beta}}{2^\beta (\beta-1)}   \sim y,\qquad
\text{ as }
y\to +\infty.
$$
To complete the proof, recall that $1 - r_n = 1- \varphi(P_n)$ and $P_n \sim n$ a.s., as $n\to\infty$, by the law of large numbers.
\end{proof}

\begin{proof}[Proof of Theorem~\ref{thm:PolytopeOrNot_Tessellation}]
It is well known (and follows from general properties of beta distributions, see, e.g., Lemma~3.1 with $\beta=-1$ in~\cite{KabluchkoThaeleZaporozhets}) that the normalized spherical volume of a spherical cap $S(u,h) = \{x\in \SS^{d-1}: \langle u, x\rangle > h\}$ satisfies
$$
\sigma_{d-1} (S(u,h))
=
\frac{\Gamma(\frac d2)}{\sqrt \pi \, \Gamma (\frac{d-1}{2})}  \int\limits_{h}^1 (1-s^2)^{\frac{d-3}{2}} \dint s
\sim
\frac{\Gamma(\frac d2)}{\sqrt \pi \, \Gamma (\frac{d-1}{2})} \cdot \frac{2^{\frac{d-1}{2}} (1-h)^{\frac{d-1}{2}}}{d-1},\qquad
\text{ as } h\uparrow 1,
$$
where the asymptotic equivalence can be verified using the rule of L'Hospital.
Taking into account Lemma~\ref{lem:PPP_on_01_asympt_H_n} it follows that with probability $1$,
\begin{equation}\label{eq:vol_caps_asymptotics}
\sigma_{d-1} (S(u_n,r_n))
\sim
\frac{\Gamma(\frac d2)}{\sqrt \pi \, \Gamma (\frac{d-1}{2})}
\frac{2^{\frac{d-1}{2}}}{d-1}
\left(\frac{\lambda}{2^\beta (\beta-1)}\right)^{\frac{d-1}{2(\beta-1)}} \cdot n^{ - \frac{d-1}{2\beta-2}},
\qquad
\text{ as }
n\to\infty.
\end{equation}

Let now $\beta > (d+1)/2$. Then, $(d-1)/(2\beta-2) < 1$ and we can apply Theorem~\ref{theo:hoffmann_joergensen}~(i) with $a=\infty$  to almost every realization of $r_1<r_2<\ldots$ to conclude that $\SS^{d-1}$ is a.s.\ covered by the open caps $S(u_n, r_n)$, $n\in \NN$. By compactness, we can extract finitely many open caps that cover the sphere. It follows that the zero cell $Z_{d,\lambda,\beta}$ is a polytope contained in $\BB^d$ (and not intersecting $\SS^{d-1}$).

Let now $d/2< \beta < (d+1)/2$. Then, $(d-1)/(2\beta-2) > 1$ and we can apply  Theorem~\ref{theo:hoffmann_joergensen}~(ii) with $a=0$ to conclude that the closures of the caps $S(u_n, r_n)$, $n\in \NN$,  do not cover the sphere on certain event, $E$ say, of positive probability. On this event, the zero cell $Z_{d,\lambda,\beta}$ intersects $\SS^{d-1}$. On the same event $E$, $Z_{d,\lambda,\beta}$ is not a (Euclidean) polytope because if it were a polytope, it could be defined by finitely many inequalities of the form $\langle x, u_n\rangle \leq r_n$. However, the union of finitely many closed spherical caps is closed, which means that its complement (being non-empty) contains a spherical cap. This contradicts the assumption that $Z_{d,\lambda,\beta}\subset \widebar\BB^d$ is a polytope.

Consider now the critical case $\beta = (d+1)/2$. After some elementary transformations, \eqref{eq:vol_caps_asymptotics} takes the following form: with probability $1$,
$$
\sigma_{d-1} (S(u_n,r_n))
\sim
\frac{\lambda \Gamma(\frac d2)}{\sqrt \pi \, \Gamma (\frac{d-1}{2}) (d-1)^2}
\cdot \frac 1n
=
\frac{\lambda}{\lambda^{\text{crit}}_d} \cdot \frac 1n,
\qquad
\text{ as }
n\to\infty.
$$
If $\lambda > \lambda^{\text{crit}}_d$, respectively, $0 < \lambda < \lambda^{\text{crit}}_d$, then we can argue as above, applying Part~(i), respectively, Part~(ii), of Theorem~\ref{theo:hoffmann_joergensen}.
\end{proof}
\begin{remark}\label{remark:doubly_critical_Hoffmann}
The doubly critical case when $\beta = (d+1)/2$ and $\lambda=\lambda^{\text{crit}}_d$ corresponds to the missing case $a=1$ in Theorem~\ref{theo:hoffmann_joergensen}. In fact, for $a=1$, Hoffmann-J{\o}rgensen~\cite[Section~5]{HoffmannJoergensen} has a more refined result in which Condition~\eqref{eq:hoffmann_joergensen_condition}  is replaced by the assumption $n \sigma_{d-1}(S (u_n,h_n)) - 1  \sim b/\log n$ for some $b\in \RR$. Unfortunately, there is a gap in the range of $b$ not covered by the results of~\cite{HoffmannJoergensen} and (after slightly more involved computations than above) one can check that the doubly critical case falls precisely into this gap.
\end{remark}

\begin{proof}[Proof of Theorem~\ref{thm:PolytopeOrNot}]
The claim follows from Theorem~\ref{thm:PolytopeOrNot_Tessellation} via the duality relation $P_{d,\alpha,\beta}^\circ \stackrel{d}{=} Z_{d,\lambda, \beta}$ stated in Theorem~\ref{thm:ZeroCell}. Note that the critical value of $\alpha$ is given by
$$
\alpha^{\text{crit}}_d
=
\frac{\lambda^{\text{crit}}_d} {\tilde c_{d,\frac{d+1}{2}}\omega_{d}} = (d-1)\pi,
$$
where we used~\eqref{eq:DensityZeta} and the formula $\omega_d = d \kappa_d = d  \pi^{d/2} / \Gamma(\frac d2 + 1)$.
\end{proof}

\subsection{Proofs of Theorems~\ref{thm:PolytopeOrNot_critical_tessellation} and~\ref{thm:PolytopeOrNot_critical}}

Although Theorem~\ref{thm:PolytopeOrNot_critical_tessellation} is known from~\cite{PorretBlanc}, \cite[Section~6]{BenjaminiJonassonEtAL} and~\cite{TykessonCalka}, we shall provide a short proof. Then, Theorem~\ref{thm:PolytopeOrNot_critical} follows by duality stated in Theorem~\ref{thm:ZeroCell}.
To prove that $Z_{2,\pi, 3/2}$ is a (Euclidean) polygon that does not touch the unit sircle $\SS^1$ we shall verify Shepp's condition~\eqref{eq:shepp_crit}.  To this end, we use the distributional representation $r_n = \varphi(P_n)$, where $P_1<P_2<\ldots$ are the arrivals of a homogeneous, unit intensity Poisson process on $(0,\infty)$, and $\varphi:[0,\infty) \to [0,1)$ is a monotone increasing function with
$$
y = \pi \int\limits_0^{\varphi(y)} (1-r^2)^{-3/2} \dint r = \frac {\pi \varphi(y)}{\sqrt{1-\varphi^2(y)}},
\qquad
y\geq 0,
$$
which is a special case of~\eqref{eq:transform_poi_proc}.
Solving this equation yields $\varphi(y) = y / \sqrt{y^2 + \pi^2}$ and we arrive at
$$
r_n =  \frac{P_n}{\sqrt {P_n^2 + \pi^2}} = 1 - \frac{\pi^2}{2 P_n^2} + O(P_n^{-4}),
\qquad
\text{ as }
n\to\infty.
$$
The length of the arc $\{x\in \SS^1:\langle x, u_n\rangle > r_n\} $ is $2\pi \ell_n$, where
$$
\ell_n = \frac 1 {2\pi} \cdot 2 \arccos r_n = \frac 1 {P_n} + O(P_n^{-2}),
\qquad
\text{ as }
n\to\infty.
$$
By the law of the iterated logarithm, we have $P_k = k + o(\sqrt{k\log k})$ a.s., as $k\to\infty$.  It follows that
$$
\frac 1 {P_1} + \ldots +  \frac 1 {P_n}
=
\sum_{k=1}^n \frac 1 {k + o(\sqrt{k\log k})}
=
\sum_{k=1}^n \frac 1k \left (1  - o\left( \sqrt{\frac {\log k} {k}}\right)\right)
=
\log n + O(1),
$$
as $n\to\infty$.
Hence,
$$
\sum_{n=1}^\infty \frac 1 {n^2} e^{\ell_1+\ldots + \ell_n}
=
\sum_{n=1}^\infty \frac 1 {n^2} e^{\frac 1 {P_1} + \ldots +  \frac 1 {P_n} + O(1)}
=
\sum_{n=1}^\infty \frac 1 {n^2} e^{\log n + O(1)}
=
\sum_{n=1}^\infty \frac {e^{O(1)}} {n} = \infty.
$$
Hence, Shepp's criterium~\eqref{eq:shepp_crit} implies that the circle $\SS^1$ is covered by the arcs with probability $1$. By compactness, we can extract finitely many open arcs covering the circle and it follows that $Z_{2,\pi, 3/2}$  is a polygon with probability one.
\hfill $\Box$

\section{Properties of  \texorpdfstring{beta$^*$}{beta*} intensities}
In this section we shall state and prove two basic properties of the beta$^*$ intensities: invariance under projections and the canonical decomposition. We recall that for $d\in \NN$, $\beta>d/2$ and $\alpha>0$ we defined the beta$^*$ intensity $f_{d,\alpha,\beta}$ by
\begin{equation}\label{eq:f_d_alpha_beta_def_repeat}
f_{d,\alpha,\beta}(x) := \frac{\alpha\tilde c_{d,\beta}}{(\|x\|^2-1)^{\beta}}
,\qquad x\in\RR^d\backslash \widebar\BB^d,
\qquad \tilde c_{d,\beta} = {\Gamma(\beta)\over\pi^{d\over 2}\Gamma(\beta-{d\over 2})}.
\end{equation}
Let $\zeta_{d,\alpha,\beta}$ be a Poisson process on $\RR^d\backslash \widebar\BB^d$ with intensity  $f_{d,\alpha,\beta}$ with respect to the Lebesgue measure.

\subsection{Projections}
For $k\in\{1,\ldots,d\}$ we let $\pi_k:\RR^d\to\RR^k$ denote the orthogonal projection onto the first $k$ coordinates, that is $\pi_k(z_1,\ldots,z_d) = (z_1,\ldots,z_k)$.  The next lemma shows that Poisson processes with density of the form $f_{d,\alpha,\beta}$ are stable under projections, meaning that the orthogonal projections onto lower-dimensional subspaces of such Poisson processes are again of the same type with a suitably modified parameter $\beta$. We remark that this projection behaviour is similar to the one for beta and beta' densities as considered in~\cite[Lemma~3.1]{KabluchkoThaeleZaporozhets}.

\begin{figure}[t]
	\centering
	\includegraphics[width=0.8\columnwidth]{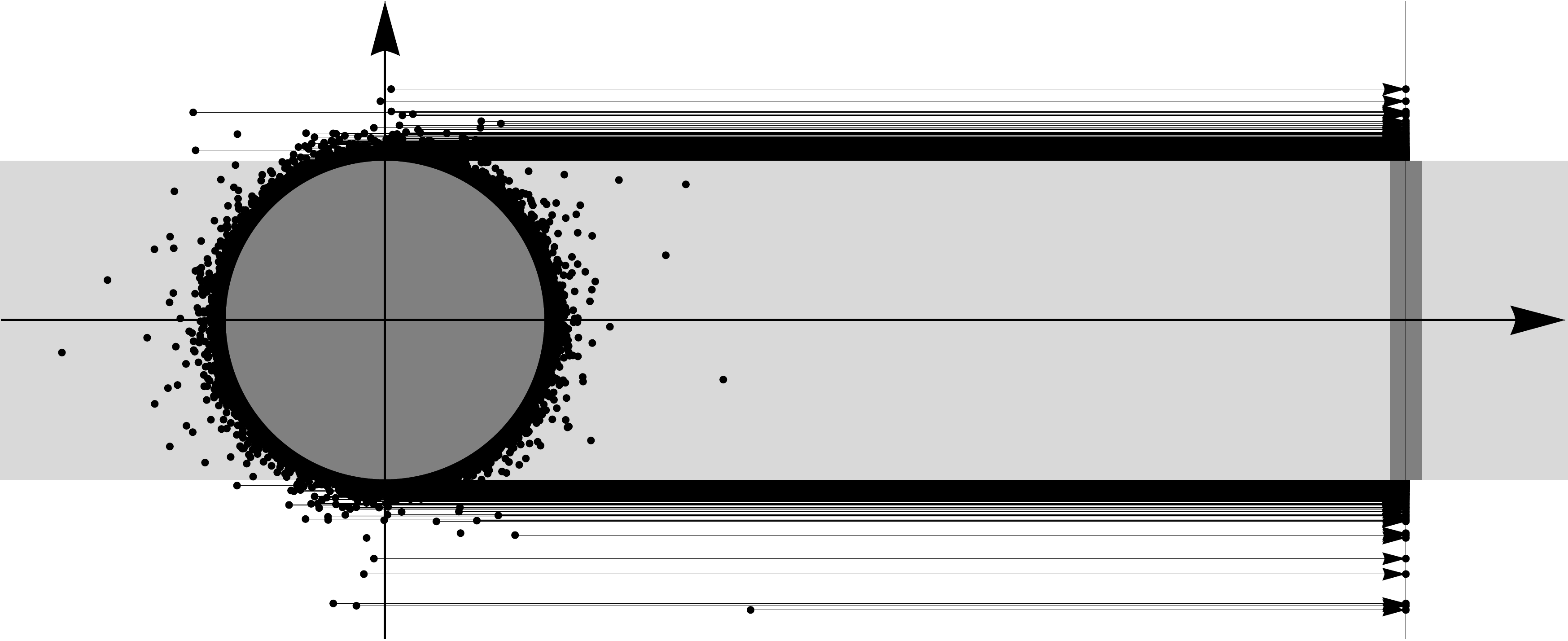}
	\caption{Projection property of beta$^*$ intensities; see Lemma~\ref{lem:Projection}.}
\label{fig:projection}
\end{figure}

\begin{lemma}\label{lem:Projection}
Let $\zeta_{d,\alpha,\beta}$ be a Poisson process in $\RR^d\backslash \widebar\BB^d$ with intensity $f_{d,\alpha,\beta}$. Then, its projection
$$
\pi_k(\zeta_{d,\alpha,\beta}) := \sum_{\substack{z \text{ is an atom of } \zeta_{d,\alpha,\beta}\\ \pi_k (z) \notin \widebar \BB^k}} \delta_{\pi_k(z)}
$$
has the same distribution as the Poisson process $\zeta_{k,\alpha,\beta-{d-k\over 2}}$ in $\RR^k\backslash \widebar \BB^k$.
\end{lemma}
\begin{remark}
Note that the points of $\zeta_{d,\alpha,\beta}$ are located outside the closed unit ball $\widebar\BB^d$ and accumulate close to its boundary; see Figure~\ref{fig:projection}. In our context it is useful to think of $\BB^d$ as kind of `black hole' into which $\zeta_{d,\alpha,\beta}$ densely puts an infinite number of points at infinite intensity.
%It is convenient to think of $\widebar\BB^d$ as of a ``black hole'' which is filled with the points of $\zeta_{d,\alpha,\beta}$ with infinite intensity.
Under projection $\pi_k$, the points of $\zeta_{d,\alpha,\beta}$ belonging to the cylinder $\widebar \BB^k \times \RR^{d-k}$ (which contains $\widebar\BB^d$)  are mapped to the new `black hole' $\widebar \BB^k$, and fill it with infinite intensity. Note that infinitely many points which were initially not in the black hole may fall into it after projection.
In a rigorous treatment, such points have to be excluded from the definition of $\pi_k(\zeta_{d,\alpha,\beta})$; see Figure~\ref{fig:projection}.
\end{remark}
\begin{proof}[Proof of Lemma~\ref{lem:Projection}]
By the transformation property of Poisson processes, $\pi_k(\zeta_{d,\alpha,\beta})$ is a Poisson process in $\RR^k \backslash \widebar \BB^k$ and its intensity measure is the image of the intensity measure of $\zeta_{d,\alpha,\beta}$ under $\pi_k$ (provided the image measure is locally finite, which we shall see a posteriori). To determine the image measure, we represent the points from $\RR^d$ as $(x,y)$ with $x\in \RR^k$ and $y\in \RR^{d-k}$, so that $\pi_k((x,y)) = x$. Let us fix some $x\in\RR^k$ such that $h:=\|x\|>1$. Then the Lebesgue density of the intensity measure of $\pi_k(\zeta_{d,\alpha,\beta})$ at $x$ is given by
\begin{align*}
\psi(x) = \int\limits_{\RR^{d-k}}f_{d,\alpha,\beta}((x,y))\,\dint y
%&
=\alpha\tilde c_{d,\beta}\int\limits_{\RR^{d-k}}{\dint y\over( h^2+\|y\|^2-1)^\beta}
%&
%={\alpha\tilde c_{d,\beta}\over(\|x\|^2-1)^\beta}\int_{\RR^{d-k}}{\dint y\over\Big(1+{\|y\|^2\over\|x\|^2-1}\Big)^\beta}
= {\alpha\tilde c_{d,\beta}\over(h^2-1)^\beta}\int\limits_{\RR^{d-k}}{\dint y\over\Big(1+{\|y\|^2\over h^2-1}\Big)^\beta},
\end{align*}
where we used that $\|(x,y)\|^2 = h^2 + \|y\|^2$.
Applying the substitution $y/\sqrt{h^2-1}=z$ we see that
\begin{align*}
\psi(x)
%\int_{\RR^{d-k}}f_{d,\alpha,\beta}((x,y))\,\dint y
=
{\alpha\tilde c_{d,\beta}\over(h^2-1)^\beta}\int\limits_{\RR^{d-k}}{(h^2-1)^{d-k\over 2}\,\dint z\over(1+\|z\|^2)^\beta}
=
{\alpha\tilde c_{d,\beta}\over(h^2-1)^{\beta-{d-k\over 2}}}\int\limits_{\RR^{d-k}}{\dint z\over(1+\|z\|^2)^\beta}.
\end{align*}
It remains to compute the integral on the right-hand side. Using the polar integration in $\RR^{d-k}$ and the substitution $s=r^2$ we obtain
\begin{align*}
\int\limits_{\RR^{d-k}}{\dint z\over(1+\|z\|^2)^\beta} &= \omega_{d-k}\int\limits_0^\infty{r^{d-k-1}\over(1+r^2)^\beta}\,\dint r = {\omega_{d-k}\over 2}\int\limits_0^\infty{s^{{d-k\over 2}-1}\over(1+s)^{\beta}}\,\dint s\\
&={\pi^{d-k\over 2}\over\Gamma({d-k\over 2})}{\Gamma({d-k\over 2})\Gamma(\beta-{d-k\over 2})\over\Gamma(\beta)} ={1\over \tilde c_{d-k,\beta}}.
\end{align*}
Putting pieces together yields
\begin{align*}
\psi(x)
= {\tilde c_{d,\beta}\over \tilde c_{d-k,\beta}}{\alpha\over(h^2-1)^{\beta-{d-k\over 2}}}
= {\Gamma(\beta)\over\pi^{d-k\over 2}\Gamma(\beta-{d-k\over 2})}{\alpha\over(h^2-1)^{\beta-{d-k\over 2}}}
= {\alpha\tilde c_{k,\beta-{d-k\over 2}}\over(h^2-1)^{\beta-{d-k\over 2}}}.
\end{align*}
This proves that the two Poisson processes $\pi_k(\zeta_{d,\alpha,\beta})$ and $\zeta_{k,\alpha,\beta-{d-k\over 2}}$ have the same intensity measure and are hence identically distributed. The argument is thus complete.
\end{proof}

\subsection{Canonical decomposition}\label{sec:canonical_decomp}
%where $d\in\NN$, $\alpha>0$ and $\beta> d/2$.
We are going to state a property of beta$^*$ measures which can be regarded as an analogue of the canonical decomposition of beta and beta' distributions of Ruben and Miles~\cite{Ruben_Miles,MilesIsotropicSimplices}; see also Theorems~3.3 and 3.6 from~\cite{KabluchkoThaeleZaporozhets}.
Let $\Phi_{d,\alpha,\beta}$ be the (non-random) measure on $\RR^d\backslash\widebar\BB^d$ with Lebesgue density $f_{d,\alpha,\beta}$; see~\eqref{eq:f_d_alpha_beta_def_repeat}. Although the measure $\Phi_{d,\alpha,\beta}$ is infinite for $\beta > \max (d/2,1)$, it is convenient to think probabilistically and imagine that we sample a tuple $x=(x_1,\dots,x_k)$ whose components $x_1,\ldots,x_k$ are $k\in\{1,\ldots,d\}$ ``random vectors'' in $\RR^d\backslash\widebar\BB^d$  that are stochastically independent and ``distributed'' according this infinite measure.  In Part (b) of the following Theorem~\ref{thm:canonical_decomp} we shall provide a description of the ``distribution'' of their affine hull $A_x:=\aff(x_1,\dots,x_k)$.  By rotational invariance, it suffices to compute the ``distribution'' of the distance from $A_x$ to the origin (which is an infinite measure). Moreover, in Part~(a) we shall describe the distribution of the points $x_1,\ldots,x_k$ inside their own affine hull $A_x$. This requires clarification since the affine subspace $A_x$ is itself random.  To address this issue, we shall fix certain non-random identification $I_{A}: A \to \RR^{k-1}$ for every affine subspace $A\subset \RR^d$ with $\dim A = k-1$ (see below) and look at the points $I_{A_x}(x_1), \ldots, I_{A_x}(x_k)$ in $\RR^{k-1}$, after certain normalization.

%For every affine subspace $A\in A(d,k-1)$ we fix an isometry $I_A : A\to\RR^{k-1}$ such that $I_A(p(A))=0$, where $p(A)$ denotes the orthogonal projection of the origin onto $A$.
%If not explicitly stated otherwise,
For $\ell\in\{0,1,\ldots,d\}$ we write $G(d,\ell)$ for the Grassmannian of $\ell$-dimensional linear subspaces of $\RR^d$ and $A(d,\ell)$ for the Grassmannian of $\ell$-dimensional affine subspaces of $\RR^d$.
We denote by $\pi_E:\RR^d\to E$ the orthogonal projection onto an affine subspace $E\subset \RR^d$, and by $p(E)=\pi_E(0)=\text{argmin}_{x\in E}\|x\|$ the projection of the origin onto $E$.
For every affine subspace $E\in A(d,k-1)$ let us fix an isometry $I_E:E\to\RR^{k-1}$ such that $I_E(p(E))=0$. Additionally, we  require  that $(x,E)\mapsto I_E(\pi_E(x))$ defines a measurable map from $\RR^d\times A(d,k-1)$ to $\RR^{k-1}$, where the spaces $\RR^d$, $\RR^{k-1}$ and $A(d,k-1)$ are equipped with their standard Borel $\sigma$-algebras, see~\cite[Chapter~13.2]{SW} for the $\sigma$-algebra of $A(d,k-1)$. Note that the choice of $I_E$ is not unique.
%(Essentially due to rotation invariance of $f_{d,\alpha,\beta}$,)

\begin{theorem}[Canonical decomposition]\label{thm:canonical_decomp}
Let $\Phi_{d,\alpha,\beta}$ be the measure on $\RR^d\backslash \widebar\BB^d$ with Lebesgue density $f_{d,\alpha,\beta}$ and parameters $\alpha>0$ and $\beta> d/2$. Fix $k\in\{1,\dots,d\}$ and let $x=(x_1,\ldots,x_k)$ be a $k$-tuple of  points from $\RR^d\backslash\widebar\BB^d$. Let $A_x:=\aff(x_1,\dots,x_k)$ be the affine hull of the points $x_1,\dots,x_k$, and $p(A_x)$ be the orthogonal projection of the origin on $A_x$. Recall that $\dist(A_x)=\|p(A_x)\|$ denotes the distance from the origin to $A_x$. Define the set
$$
\mathcal{S}_{k,d}:=\{x=(x_1,\dots,x_k)\in (\RR^d)^k: A_x\cap\widebar\BB^d=\varnothing, \dim A_x = k-1\}
$$
and the transformation
\begin{align*}
\begin{array}{cccc}
T: & \mathcal S_{k,d} & \to & (\RR^{k-1})^k\times (\RR^{d-k+1}\backslash \widebar \BB^{d-k+1}) \\
 & (x_1,\dots,x_k) & \mapsto & (T_1(x),T_2(x))
\end{array}
\end{align*}
given by
\begin{align*}
T_1(x):=\left(\frac{I_{A_x}(x_1)}{\sqrt{\dist^2(A_x)-1}},\ldots,\frac{I_{A_x}(x_k)}{\sqrt{\dist^2(A_x)-1}}\right),
\quad
T_2(x):=I_{A_x^\perp}(p(A_x)).
\end{align*}
The restriction of a measure $\mu$ to a set $B$ is denoted by $\mu|_B$.  Then the following hold.
\begin{itemize}
\item[(a)] We have the decomposition
$$
\bigg(\bigotimes_{i=1}^k\Phi_{d,\alpha,\beta}\bigg)\bigg|_{\mathcal S_{k,d}} \circ T^{-1}=\phi_{k-1}\otimes \Bigg(\bigg(\bigotimes_{i=1}^k\Phi_{d,\alpha,\beta}\bigg)\bigg|_{\mathcal S_{k,d}} \circ T_2^{-1}\Bigg)
$$
as a product measure, where $\phi_{k-1}$ is the probability measure on $(\RR^{k-1})^k$ whose Lebesgue density is a constant multiple of
\begin{align*}
V_{k-1}^{d-k+1}(\conv(z_1,\dots,z_k))\,\prod_{i=1}^k \tilde f_{k-1,\beta}(z_i),\qquad z_1,\ldots,z_k\in\RR^{k-1}.
\end{align*}
%\TG{Das wurde vorher mit $\Delta^{d-k+1}(z_1,\dots,z_k)$ notiert. Ich habe das geändert damit es konsistent zur Notation im Beweis für die $T$-Funktionale ist! Das betrifft dann ebenfalls die ursprüngliche Notation der Normierungskonstanten $D(d,k,\beta)$ unten, jetzt $S_{k,\beta}(d-k)$ wie im Theorem zum $T$-funktional}
Here, $V_{k-1}(\conv(z_1,\dots,z_k))$ is the $(k-1)$-volume of the simplex $\conv(z_1,\ldots,z_k)$ and $\tilde f_{k-1,\beta}(z)$ is the beta' density on $\RR^{k-1}$ given by
$$
\tilde f_{k-1,\beta}(z) = \tilde c_{k-1,\beta} (1+\|z\|^2)^{-\beta},
\qquad
z\in \RR^{k-1}.
$$
%spanned by $k\in\{1,\ldots,d+1\}$ points in $\RR^d$;
\item[(b)] The Lebesgue density of the (usually, infinite) measure $\big(\bigotimes_{i=1}^k\Phi_{d,\alpha,\beta}\big)\big|_{\mathcal S_{k,d}} \circ T_2^{-1}$ on the space $\RR^{d-k+1}\backslash\widebar\BB^{d-k+1}$ is given by $f_{d-k+1,\alpha^k,k\beta-\frac{(d+1)(k-1)}{2}}$.
\end{itemize}
\end{theorem}

In the proof of Theorem~\ref{thm:canonical_decomp} we will rely on an integral-geometric transformation formula for which we refer to~\cite[Theorem 7.2.7]{SW}. For $\ell \in \{0,\ldots, d\}$ we use $\nu_\ell$ to denote the invariant Haar probability measure on the linear Grassmannian $G(d,\ell)$. Moreover, let
\begin{equation}\label{eq:defmuk}
\mu_\ell(\,\cdot\,) := \int\limits_{G(d,\ell)}\int\limits_{L^\perp}\ind_{\{L+x\in\,\cdot\,\}}\,\dint x\, \nu_\ell(\dint L)
\end{equation}
be the invariant measure on the affine Grassmannian $A(d,\ell)$; see~\cite[Chapter 13.2]{SW}.
We use the convention that whenever we integrate over an affine subspace  of $\RR^d$ (such as $L^\bot$ in the above formula), the corresponding differential $\dint x$ refers to the integration with respect to the Lebesgue measure in that subspace, which will always be clear from the context.

\begin{proposition}[Affine Blaschke-Petkantschin formula]\label{prop:BlaschkePetkantschin}
Let $k\in\{1,\dots,d+1\}$ and $h:(\RR^d)^k\to\RR$ be a non-negative measurable function. Then
\begin{align*}
&\int\limits_{(\RR^d)^k}h(x_1,\dots,x_k)\,\dint(x_1,\dots,x_k)\\
&\qquad\qquad=B(d,k)\int\limits_{A(d,k-1)}\int\limits_{E^k}h(x_1,\dots,x_k)V_{k-1}^{d-k+1}(\conv(x_1,\dots,x_k))\,\dint (x_1,\dots,x_k)\,\mu_{k-1}(\dint E),
\end{align*}
where, recalling  that $\omega_\ell:=2\pi^{\ell/2}/\Gamma(\ell/2)$ is the surface area of the unit sphere in $\RR^\ell$,
\begin{align*}
B(d,k):=((k-1)!)^{d-k+1}\frac{\omega_{d-k+2}\cdot\ldots\cdot \omega_d}{\omega_1\cdot\ldots\cdot\omega_{k-1}},\qquad B(d,1):=1.
\end{align*}
\end{proposition}

\begin{proof}[Proof of Theorem~\ref{thm:canonical_decomp}]
Let $g:(\RR^{k-1})^k\to[0,\infty]$ and $h:\RR^{d-k+1}\backslash\widebar \BB^{d-k+1}\to [0,\infty]$ be non-negative measurable functions. Our interest lies in the following quantity:
\begin{align*}
B_{g,h}:=
&	\int\limits_{(\RR^{k-1})^k\times (\RR^{d-k+1}\backslash\widebar \BB^{d-k+1})}g(z_1,\dots,z_k)h(y)		\left(\bigg(\bigotimes_{i=1}^k\Phi_{d,\alpha,\beta}\bigg)\bigg|_{\mathcal S_{k,d}} \circ T^{-1}\right)(\dint(z_1,\dots,z_k,y))\\
=&	\int\limits_{\mathcal S_{k,d}}g\Bigg(\frac{I_{A_x}(x_1)}{\sqrt{\dist^2(A_x)-1}},\ldots,\frac{I_{A_x}(x_k)}{\sqrt{\dist^2(A_x)-1}}\Bigg)h\big(I_{A_x^\perp}(p(A_x))\big)\bigg(\prod_{i=1}^k\Phi_{d,\alpha,\beta}(\dint x_i)\bigg).
\end{align*}
Using  the affine Blaschke-Petkantschin formula from Proposition~\ref{prop:BlaschkePetkantschin} and recalling~\eqref{eq:defmuk} %that $\mu_{k-1}$ denotes the invariant measure on $A(d,k-1)$,
we obtain
\begin{align*}
B_{g,h}
&	=B(d,k)\int\limits_{A(d,k-1)}\int\limits_{A^k}\ind_{\{A\cap \widebar\BB^d=\varnothing\}}\cdot g\Bigg(\frac{I_{A}(x_1)}{\sqrt{\dist^2(A)-1}},\ldots,\frac{I_{A}(x_k)}{\sqrt{\dist^2(A)-1}}\Bigg)\\
&	\qquad\times h\big(I_{A^\perp}(p(A))\big)V_{k-1}^{d-k+1}(\conv(x_1,\dots,x_k))\bigg(\prod_{i=1}^kf_{d,\alpha,\beta}(x_i)\,\dint x_i\bigg)\mu_{k-1}(\dint A)\\
&	=\alpha^k\cdot(\tilde c_{d,\beta})^k\cdot B(d,k)\int\limits_{A(d,k-1)}\int\limits_{(\RR^{k-1})^k}\ind_{\{\dist(A)>1\}}\cdot g\Bigg(\frac{y_1}{\sqrt{\dist^2(A)-1}},\ldots,\frac{y_k}{\sqrt{\dist^2(A)-1}}\Bigg)\\
&	\qquad\times V_{k-1}^{d-k+1}(\conv(y_1,\dots,y_k))\bigg(\prod_{i=1}^k\ind_{\{\dist^2(A)+\|y_i\|^2>1\}}\frac{\dint y_i}{(\dist^2(A)+\|y_i\|^2-1)^{\beta}}\bigg)\\
&	\qquad\times h\big(I_{A^\perp}(p(A))\big)\mu_{k-1}(\dint A).
\end{align*}
In the second step, we used the substitution $y_i=I_A(x_i)\in\RR^{k-1},\,1\le i\le k$, together with the fact that $\|x_i\|^2=\dist^2(A)+\|y_i\|^2$, and the definition of $f_{d,\alpha,\beta}$. Recall that $I_A:A\to\RR^{k-1}$ is an isometry satisfying $I_A(p(A))=0$. Applying the substitution
$$
z_i:=y_i/\sqrt{\dist^2(A)-1}\in\RR^{k-1}, \qquad 1\le i\le k,
$$
we observe the following:
\begin{align*}
(\dist^2(A)+\|y_i\|^2-1)^{-\beta}
&	=(\dist^2(A)-1)^{-\beta}\bigg(\frac{\|y_i\|^2}{\dist^2(A)-1}+1\bigg)^{-\beta}\\
&	=(\dist^2(A)-1)^{-\beta}(\|z_i\|^2+1)^{-\beta},\\
\dint y_i
&	=(\dist^2(A)-1)^{\frac{k-1}{2}}\,\dint z_i,\\
V_{k-1}(\conv(y_1,\dots,y_k))
&	=(\dist^2(A)-1)^{\frac{k-1}{2}}V_{k-1}(\conv(z_1,\dots,z_k)).
\end{align*}
Hence, $B_{g,h}$ transforms into
\begin{align}
&\alpha^k(\tilde c_{d,\beta})^k B(d,k)\int\limits_{A(d,k-1)}\int\limits_{(\RR^{k-1})^k}\ind_{\{\dist(A)>1\}}\cdot g(z_1,\dots,z_k)h\big(I_{A^\perp}(p(A))\big) V_{k-1}^{d-k+1}(\conv(z_1,\dots,z_k))\notag\\
&	\qquad\times(\dist^2(A)-1)^{\frac{(d-k+1)(k-1)}{2}+\frac{k(k-1)}{2}-k\beta}\bigg(\prod_{i=1}^k(1+\|z_i\|^2)^{-\beta}\,\dint z_i\bigg)\mu_{k-1}(\dint A)\notag\\
&	=\frac{\alpha^k(\tilde c_{d,\beta})^k B(d,k)}{(\tilde c_{k-1,\beta})^k}\left(\,\int\limits_{A(d,k-1)}h\big(I_{A^\perp}(p(A))\big)\ind_{\{\dist(A)>1\}}(\dist^2(A)-1)^{-\gamma}\,\mu_{k-1}(\dint A) \right)\label{eq:int1}\\
&	\qquad\times\left(\,\int\limits_{(\RR^{k-1})^k}g(z_1,\dots,z_k)V_{k-1}^{d-k+1}(\conv(z_1,\dots,z_k))\bigg(\prod_{i=1}^k\tilde f_{k-1,\beta}(z_i)\,\dint z_i\bigg)\right)\label{eq:int2},
\end{align}
where we also used the definition~\eqref{eq:BetaprimeDensity} of the beta' density $\tilde f_{k-1,\beta}$ and set
$$
\gamma:=-\frac{(d-k+1)(k-1)}{2}-\frac{k(k-1)}{2}+k\beta=k\beta-\frac{(d+1)(k-1)}{2}.
$$
Note that $\gamma>(d-k+1)/2$ since $\beta>d/2$ by assumption. We observe that the above formula already exhibits the desired product structure for part (a) and that the integral in line~\eqref{eq:int2} is already in the desired form for $\phi_{k-1}$ in part (a). It remains  to rewrite the integral in line~\eqref{eq:int1} using the definition~\eqref{eq:defmuk} of the invariant measure $\mu_{k-1}$ on $A(d,k-1)$. This yields
\begin{align*}
&\int\limits\limits_{A(d,k-1)}h\big(I_{A^\perp}(p(A))\big)\ind_{\{\dist(A)>1\}}(\dist^2(A)-1)^{-\gamma}\,\mu_{k-1}(\dint A)\\
&	\quad=\int\limits_{G(d,k-1)}\int\limits_{L^\perp}h\big(I_{L^\perp}(x))\big)\ind_{\{\|x\|>1\}}(\|x\|^2-1)^{-\gamma}\,\dint x\,\nu_{k-1}(\dint L)\\
&	\quad=\int\limits_{\RR^{d-k+1}}h(y)\ind_{\{\|y\|>1\}}(\|y\|^2-1)^{-\gamma}\,\dint y\\
&	\quad=\frac{1}{\tilde c_{d-k+1,\gamma}}\int\limits_{\RR^{d-k+1}\backslash \widebar \BB^{d-k+1}}h(y)f_{d-k+1,1,\gamma}(y)\,\dint y,
\end{align*}
where in the first step, we used that for every linear subspace $L\in G(d,k-1)$ and point $x\in L^\perp$ we have $(L+x)^\perp=L^\perp$, $p(L+x)=x$ and $\dist(L+x)=\|x\|$. In the second step, we applied the substitution $y:=I_{L^\perp}(x)\in\RR^{d-k+1}$ together with the fact that $\|y\|=\|x\|$, while the last equation follows from the definition of the beta$^*$ intensity. Finally, this yields
\begin{multline*}
B_{g,h}
	=\left(\frac{(\tilde c_{d,\beta})^k\cdot B(d,k)\cdot S_{k,\beta}(d-k)}{\tilde c_{d-k+1,\gamma}}\int\limits_{\RR^{d-k+1}\backslash \widebar \BB^{d-k+1}}h(y)f_{d-k+1,\alpha^k,\gamma}(y)\,\dint y\right)\\
\times\left(\frac{1}{S_{k,\beta}(d-k)(\tilde c_{k-1,\beta})^k}\int\limits_{(\RR^{k-1})^k}g(z_1,\dots,z_k)V_{k-1}^{d-k+1}(\conv(z_1,\dots,z_k))\bigg(\prod_{i=1}^k\tilde f_{k-1,\beta}(z_i)\,\dint z_i\bigg)\right),
\end{multline*}
since its turns out that the normalization constant for the second integral satisfies
\begin{align}
&\int\limits_{(\RR^{k-1})^k}V_{k-1}^{d-k+1}(\conv(z_1,\dots,z_k))\bigg(\prod_{i=1}^k \tilde f_{k-1,\beta}(z_i)\,\dint z_i\bigg)\label{eq:def_norm_const}\\
&\quad	=\frac{1}{((k-1)!)^{d-k+1}}\frac{\Gamma(\beta-\frac d2)^k\Gamma(k\beta-\frac{(d+1)(k-1)}{2})}{\Gamma(\beta-\frac{k-1}{2})^k\Gamma(k\beta-\frac{dk}{2})}\prod_{j=1}^{k-1}\left[\frac{\Gamma(\frac j2+\frac{d-k+1}{2})}{\Gamma(\frac j2)} \right]=S_{k,\beta}(d-k)(\tilde c_{k-1,\beta})^k,\notag
\end{align}
where $S_{k,\beta}(d-k)$ is defined in~\eqref{eq:def_Sdbeta}. The integral in the first line is the $(d-k+1)$-st moment of the volume of a beta' simplex in $\RR^{k-1}$, whose value is known, see e.g.~\cite[Proposition~2.8]{KabluchkoTemesvariThaele} or~\cite[Theorem 2.3 (c)]{GroteKabluchkoThaele} and follows from a formula due to Miles~\cite{MilesIsotropicSimplices}.

Taking the explicit expressions of the constants $\tilde c_{d,\beta}$, $B(d,k)$ and $S_{k,\beta}(d-k)$ into consideration, elementary manipulations yield
\begin{align*}
\frac{(\tilde c_{d,\beta})^k\cdot B(d,k)\cdot S_{k,\beta}(d-k)}{\tilde c_{d-k+1,\gamma}}=1.
\end{align*}
Recalling the definition of the probability measure $\phi_{k-1}$  from the statement of the theorem finally gives the identity
\begin{align*}
B_{g,h}
	&=\left(\hspace*{2pt}\int\limits_{\RR^{d-k+1}\backslash \widebar \BB^{d-k+1}}h(y)f_{d-k+1,\alpha^k,\gamma}(y)\,\dint y \right)\left(\,\int\limits_{(\RR^{k-1})^k}g(z_1,\dots,z_k)\,\phi_{k-1}(\dint(z_1,\dots,z_k))\right).
\end{align*}
Since both measures on the right-hand side are $\sigma$-finite (the second one is even a probability measure), the form of both integrals and the product structure proves claim (a). Inserting $g=1$ also yields (b), and thus, completes the proof.
\end{proof}

\section{Expected \texorpdfstring{$T$}{T}-functional and intrinsic volumes: Proofs}\label{sec:proof_Tfunctional}

Our first aim is to prove Theorem~\ref{thm:Tfunctional} on the expected $T$-functional of beta$^*$ sets. For this, we start by recalling  a version of the multivariate Mecke formula for Poisson processes that can be found in~\cite[Theorem 4.4]{LPbook} or~\cite[Corollary~3.2.3]{SW}. Let $\XX$ be a Polish space supplied with a non-atomic locally finite measure $\lambda$. By ${\bf N}(\XX)$ we denote the space of counting measures on $\XX$, which can be endowed with a canonical $\sigma$-field, see~\cite{LPbook,SW} for details.

\begin{proposition}[Mecke's formula]\label{prop:Slivnyak-Mecke} Let $\eta$ be a Poisson process on $\XX$ with intensity measure $\lambda$. Fix $m\in\NN$ and let $h:\XX^m\times{\bf N}(\XX)\to\RR$ be a non-negative measurable function. Then
\begin{align*}%\label{eq:Slivnyak-Mecke}
\EE \sum_{(x_1,\ldots,x_m)\in\eta_{\neq}^m}h(x_1,\ldots,x_m,\eta)
= \int\limits_{\XX^m}\EE h(x_1,\ldots,x_m,\eta\cup\{x_1,\ldots,x_m\})\,\lambda^m(\dint(x_1,\ldots,x_m)),
\end{align*}
where $\eta_{\neq}^m$ is the collection of all $m$-tuples of pairwise distinct points of $\eta$.
\end{proposition}

We can now turn to the proof of Theorem \ref{thm:Tfunctional}.

\begin{proof}[Proof of Theorem \ref{thm:Tfunctional}]
Let us begin with the special case $d=1$. By definition of the functional $T_{a,b}$ we can write
\begin{align*}
\EE T_{a,b}(P_{1,\alpha,\beta}) &= \EE[|\min\zeta_{1,\alpha,\beta}|^a+ (\max\zeta_{1,\alpha,\beta})^a]
=
2\EE(\max\zeta_{1,\alpha,\beta})^a.
%&= 2a\int\limits_{1}^\infty h^{a-1}\PP[|\min\zeta_{1,\alpha,\beta}|\geq h]\,\dint h\\
%&= 2a\int\limits_1^\infty h^{a-1}(1-\PP[\zeta_{1,\alpha,\beta}\cap(-\infty,-h)=\varnothing])\,\dint h.
\end{align*}
Since $\zeta_{1,\alpha,\beta}$ is a Poisson process with intensity $f_{1,\alpha,\beta}$ given by \eqref{eq:DensityZeta}, for every $h>1$ we have
$$
\PP[\max\zeta_{1,\alpha,\beta}\leq h]
=
\PP[\zeta_{1,\alpha,\beta}\cap(h,\infty)=\varnothing]
=
%\exp\Big\{-\alpha\tilde{c}_{1,\beta}\int\limits_{-\infty}^{-h}{\dint r\over(r^2-1)^\beta}\Big\}
%=
\exp\Big\{-\alpha\tilde{c}_{1,\beta}\int\limits_{h}^{\infty}{\dint r\over(r^2-1)^\beta}\Big\}.
$$
Differentiating in $h$ yields the density of $\max\zeta_{1,\alpha,\beta}$, which implies  the formula for $\EE T_{a,b}(P_{1,\alpha,\beta})$ in Theorem \ref{thm:Tfunctional}. To check whether $\EE T_{a,b}(P_{1,\alpha,\beta})<\infty$, we need to analyse the integrability at $h=1$ and at $h=\infty$. For $\beta>1$ the function is clearly integrable at $h=1$ and is asymptotically equivalent to a multiple of $h^{a}(h^2-1)^{-\beta}$ at $h=\infty$, which is integrable provided that $a<2\beta-1$.
%For $\beta=1$, the integrand simplifies to $h^{a-1}(1-({h-1\over h+1})^{\alpha/(2\pi)})$. This function behaves like a multiple of $(h-1)^{\alpha/(2\pi)}$ at $h=1$ and like a multiple of $h^{a-2}$ at $h=\infty$ and hence integrability follows in this case as long as $a<1$.

For $d\geq 2$ we start by observing that any facet, that is, any $(d-1)$-face, of a beta$^*$ set $P_{d,\alpha,\beta}$ is a $(d-1)$-dimensional simplex, which means that almost surely every facet $F\in\cF_{d-1}(P_{d,\alpha,\beta})$ arises as convex hull $F=\conv(x_1,\ldots,x_d)$ of $d$ distinct points $(x_1,\ldots,x_d)\in(\zeta_{d,\alpha,\beta})_{\neq}^d$ of the underlying Poisson process $\zeta_{d,\alpha,\beta}$. This allows us to rewrite $\EE T_{a,b}(P_{d,\alpha,\beta})$ as
\begin{align*}
\EE T_{a,b}(P_{d,\alpha,\beta}) &= {1\over d!}\EE\sum_{(x_1,\ldots,x_d)\in(\zeta_{d,\alpha,\beta})_{\neq}^d}\ind\{\conv(x_1,\ldots,x_d)\in\cF_{d-1}(\conv(\zeta_{d,\alpha,\beta}))\}\\
&\qquad\times\dist(\aff(x_1,\ldots,x_d))^a\,V_{d-1}(\conv(x_1,\ldots,x_d))^b\\
&={(\tilde{c}_{d,\beta}\alpha)^d\over d!}\int\limits_{(\RR^d\backslash \widebar\BB^d)^d}\PP[\conv(x_1,\ldots,x_d)\in\cF_{d-1}(\conv(\zeta_{d,\alpha,\beta}\cup\{x_1,\ldots,x_d\}))]\\
&\qquad\times\dist(\aff(x_1,\ldots,x_d))^a\,V_{d-1}(\conv(x_1,\ldots,x_d))^b\,\prod_{i=1}^d{\dint x_i\over(\|x_i\|^2-1)^\beta},
\end{align*}
where we applied the multivariate Mecke formula for Poisson processes from Proposition~\ref{prop:Slivnyak-Mecke}. In a next step, we apply the affine Blaschke-Petkantschin formula, which we rephrased in Proposition~\ref{prop:BlaschkePetkantschin}, to deduce that $\EE T_{a,b}(P_{d,\alpha,\beta})$ is the same as
\begin{align*}
& {(\tilde{c}_{d,\beta}\alpha)^d\over d}{\omega_d\over 2}\int\limits_{A(d,d-1)}\int\limits_{H^d}\PP[\conv(x_1,\ldots,x_d)\in\cF_{d-1}(\conv(\zeta_{d,\alpha,\beta}\cup\{x_1,\ldots,x_d\}))]\\
&\qquad\times \dist(H)^a\,V_{d-1}(\conv(x_1,\ldots,x_d))^{b+1}\ind\{H\cap \BB^d=\varnothing\}\prod_{i=1}^d{\dint x_i\over(\|x_i\|^2-1)^\beta}\,\mu_{d-1}(\dint H)\\
&= {(\tilde{c}_{d,\beta}\alpha)^d\over d}{\omega_d\over 2}\int\limits_{G(d,d-1)}\int\limits_{E^\perp}\int\limits_{(E+h)^d}\PP[\conv(x_1,\ldots,x_d)\in\cF_{d-1}(\conv(\zeta_{d,\alpha,\beta}\cup\{x_1,\ldots,x_d\}))]\\
&\qquad\times \|h\|^a\,V_{d-1}(\conv(x_1,\ldots,x_d))^{b+1}\,\ind\{\|h\|>1\}\prod_{i=1}^d{\dint x_i\over(\|x_i\|^2-1)^\beta}\,\dint h\,\nu_{d-1}(\dint E),
\end{align*}
where in the last step we used the decomposition of the invariant measure $\mu_{d-1}$ from~\eqref{eq:defmuk}. We also used that $\conv \zeta_{d,\alpha,\beta}$ a.s.\ contains $\widebar\BB^d$ by Proposition~\ref{prop:poi_process_properties}. For a linear subspace $L\subset \RR^d$ we let $\pi_L:\RR^d\to L$ denote the orthogonal projection onto $L$.
Observe that $\pi_{E^\perp}(x_1)=\ldots=\pi_{E^\perp}(x_d)$ and $\conv(x_1,\ldots,x_d)$ is a facet of $\conv(\zeta_{d,\alpha,\beta}\cup\{x_1,\ldots,x_d\})$ if and only if $\pi_{E^\perp}(x_1)\notin\pi_{E^\perp}(\conv(\zeta_{d,\alpha,\beta}))$. Next, we define the non-absorption probability
$$
p_{d,\alpha,\beta}(r) := \PP[re_1\notin \pi_{\lin(e_1)}(\conv(\zeta_{d,\alpha,\beta}))],\qquad r>1,
$$
where $e_1,\ldots, e_d$ is the standard orthonormal basis of $\RR^d$.
Each point $x_i\in E+h$ can be represented as $y_i + h$ for some uniquely determined $y_i\in E$. Applying for $i\in\{1,\ldots,d\}$ this change of variables and using the rotational symmetry of the integral (which allows us to identify the hyperplane $E$ with the coordinate subspace $\RR^{d-1}$ and $E^\perp$ with $\RR$) we arrive at
\begin{align}
\nonumber\EE T_{a,b}(P_{d,\alpha,\beta}) &= {(\tilde{c}_{d,\beta}\alpha)^d\over d}{\omega_d}\int\limits_1^\infty\int\limits_{(\RR^{d-1})^d}p_{d,\alpha,\beta}(h)\,h^a\,V_{d-1}(\conv(y_1,\ldots,y_d))^{b+1}\\
\nonumber&\hspace{4cm}\times\prod_{i=1}^d{\dint y_i\over(\|y_i\|^2+h^2-1)^\beta}\,\dint h\\
\nonumber&={(\tilde{c}_{d,\beta}\alpha)^d\over d}{\omega_d}\int\limits_1^\infty\int\limits_{(\RR^{d-1})^d}p_{d,\alpha,\beta}(h)\,h^a\,V_{d-1}(\conv(z_1,\ldots,z_d))^{b+1}(h^2-1)^{(d-1)(b+1)\over 2}\\
\nonumber&\hspace{4cm}\times\prod_{i=1}^d{(h^2-1)^{d-1\over 2}\,\dint z_i\over(h^2-1)^\beta(1+\|z_i\|^2)^\beta}\,\dint h\\
&={(\tilde{c}_{d,\beta}\alpha)^d\over d}{\omega_d}S_{d,\beta}(b)\int\limits_1^\infty p_{d,\alpha,\beta}(h)\,h^a\,(h^2-1)^{{(d-1)(b+1)\over 2}-d\big(\beta-{d-1\over 2}\big)}\,\dint h,
\label{eq:08-07a}
\end{align}
where we used the substitution $z_i=y_i/\sqrt{h^2-1}$ for $i\in\{1,\ldots,d\}$ and the abbreviation
$$
S_{d,\beta}(b) := \int\limits_{(\RR^{d-1})^d}V_{d-1}(\conv(z_1,\ldots,z_d))^{b+1}\,\prod_{i=1}^d{\dint z_i\over(1+\|z_i\|^2)^\beta}.
$$
As already observed in~\eqref{eq:def_norm_const}, $S_{d,\beta}(b)$ is -- up to the normalization constant $\tilde{c}_{d-1,\beta}^{d}$ -- the $(b+1)$-st moment of the volume of a beta' simplex in $\RR^{d-1}$ and given by
\begin{align*}
S_{d,\beta}(b) = {\tilde{c}_{d-1,\beta}^{-d}\over ((d-1)!)^{b+1}}{\Gamma(d(\beta-{d-1\over 2})-{(d-1)(b+1)\over 2})\over\Gamma(d(\beta-{d+b\over 2}))}\bigg({\Gamma(\beta-{d+b\over 2})\over\Gamma(\beta-{d-1\over 2})}\bigg)^{d}\prod_{i=1}^{d-1}{\Gamma({i+b+1\over 2})\over\Gamma({i\over 2})}.
\end{align*}
It remains to determine the probability $p_{d,\alpha,\beta}$ in the integral term in~\eqref{eq:08-07a}.
Since by Lemma~\ref{lem:Projection} the projected point process $\pi_{\lin(e_1)}(\conv(\zeta_{d,\alpha,\beta}))$ has the same distribution as the Poisson process $\zeta_{1,\alpha,\beta-{d-1\over 2}}$ on $\RR\backslash[-1,1]$ (whose points cluster at $\pm 1$), we can express $p_{d,\alpha,\beta}(h)$ as
\begin{align*}
p_{d,\alpha,\beta}(h) = \PP[\zeta_{1,\alpha,\beta-{d-1\over 2}}\cap[h,\infty)=\varnothing]
=
\exp\Bigg\{-\alpha \tilde{c}_{1,{\beta-{d-1\over 2}}}\int\limits_h^\infty{\dint r\over(r^2-1)^{\beta-{d-1\over 2}}}\Bigg\}.
\end{align*}
This proves the formula in Theorem \ref{thm:Tfunctional}.

In order to complete the proof, we need to verify that $\EE T_{a,b}(P_{d,\alpha,\beta})$ is finite under the assumptions of the theorem. Since $b<2\beta-d$, the prefactor $S_{d,\beta}(b)$ is finite and we only need to ensure that the integral in the expression for $\EE T_{a,b}(P_{d,\alpha,\beta})$ converges. For this it is sufficient to ensure that the integrand is integrable at $h=1$ and at $h=\infty$. Let us start with the case $\beta>(d+1)/2$. Integrability at $h=1$ follows, since at $h=1$ the integrand behaves like a multiple of $\exp\big\{-\alpha c(d,\beta)(h-1)^{-\beta+{d+1\over 2}}\big\}(h^2-1)^{-\gamma}$, for some constant $c(d,\beta)>0$ and where
$$
\gamma:=d\left(\beta-{d-1\over 2}\right)-{(d-1)(b+1)\over 2}
$$
is the negative exponent of the term $h^2-1$. On the other hand, at $h=\infty$ the integrand behaves like a multiple of $h^ah^{-2\gamma}=h^{a-2\gamma}$ and the exponent $a-2\gamma$ is less than $-1$ by our assumption on $a$, which in turn yields the desired integrability at $h=\infty$. Summarizing, this shows that $\EE T_{a,b}(P_{d,\alpha,\beta})<\infty$ for any $\alpha>0$ under the mentioned constraints on $a$ and $b$ for $\beta>(d+1)/2$ (and in fact only in these cases).

On the other hand, if $\beta=(d+1)/2$, the argument leading to integrability at $h=\infty$ remains the same and we can concentrate on integrability at $h=1$. To this end, we note that, since $\int\limits_h^\infty{\dint r\over r^2-1}={1\over 2}\log{h+1\over h-1}$ and $\tilde{c}_{1,1}=1/\pi$, it follows that $p_{d,\alpha,{d+1\over 2}}(h)=({h-1\over h+1})^{\alpha\over 2\pi}$ for $h>1$. This shows that at $h=1$ the integrand behaves like a multiple of $(h-1)^{{\alpha\over 2\pi}-\gamma}$, where again $\gamma$ is the negative exponent of the term $h^2-1$ as above. So, integrability at $h=1$ holds if (and only if)  ${\alpha\over 2\pi}-\gamma>-1$ or, equivalently, $\alpha>2\pi(\gamma-1)$. Using that in our case $\gamma=d-{(d-1)(b+1)\over 2}$, it follows that
\begin{align*}
2\pi(\gamma-1) = \pi(2d-(d-1)(b+1)-2) = \pi(d-db+b-1) = \pi(d-1)(1-b),
\end{align*}
which implies integrability at $h=1$ if (and only if) $\alpha>\pi(d-1)(1-b)$.

Finally, if the constraints on $\alpha$, $\beta$, $a$ and $b$ mentioned in the statement of the theorem are not satisfied, integrability either at $h=1$ or $h=\infty$ does not hold, implying that in these cases $\EE T_{a,b}(P_{d,\alpha,\beta})=\infty$. The proof is thus complete. %\hfill $\Box$
\end{proof}

\begin{proof}[Proof of Proposition~\ref{prop:IntrinsicVolume}]
The case $k=d$ has already been discussed before the statement of Proposition~\ref{prop:IntrinsicVolume}. To obtain the formula for $\EE V_k(P_{d,\alpha,\beta})$ for $k\in\{1,\ldots,d-1\}$ we use Kubota's formula~\cite[Equations (6.11) and (5.5)]{SW} (also known as mean projection formula) from integral geometry and Fubini's theorem. The combination of these two results shows that
$$
\EE V_k(P_{d,\alpha,\beta}) = {d\choose k}{\kappa_d\over\kappa_k\kappa_{d-k}}\int\limits_{G(d,k)}\EE V_k(\pi_LP_{d,\alpha,\beta})\,\nu_k(\dint L),
$$
where we recall that $\pi_L:\RR^d\to L$ stands for the orthogonal projection on to $L$.
However, by Lemma~\ref{lem:Projection} the projected polytope $\pi_LP_{d,\alpha,\beta}$ is a beta$^*$ polytope in $L$ with parameters $\alpha$ and $\beta-{d-k\over 2}$. Identifying $L$ with $\RR^k$ and using rotational symmetry, this yields
$$
\EE V_k(P_{d,\alpha,\beta}) = {d\choose k}{\kappa_d\over\kappa_k\kappa_{d-k}}\EE V_k(P_{k,\alpha,\beta-{d-k\over 2}}) = {1\over d}{d\choose k}{\kappa_d\over\kappa_k\kappa_{d-k}}\,\EE T_{1,1}(P_{k,\alpha,\beta-{d-k\over 2}})
$$
and finishes the proof.
\end{proof}

\section{Expected \texorpdfstring{$f$}{f}-vector and external angle sums: Proofs}
\subsection{Expected external angle sums: Proof of Theorem~\ref{thm:exp_ext_anlge_sums}}\label{sec:proof_ext_angles}

Suppose that either $\beta>(d+1)/2$ and $\alpha>0$ or $\beta=(d+1)/2$ and $\alpha > (d-1)\pi$. Our goal is to prove that
\begin{align*}
\EE\Bigg[\sum_{G\in\cF_k(P_{d,\alpha,\beta})}\gamma(G,P_{d,\alpha,\beta})\Bigg]
=\mathbb I^*_{\alpha,k+1}(2\beta-d)
\end{align*}
for $k\in\{0,\dots,d-1\}$, where $\mathbb I^*_{\alpha,m}(\lambda)$ is defined as in~\eqref{eq:def_I_star} and~\eqref{eq:def_I_star_sum}. In order to do this, we need to find a suitable description of the tangent cones $T_G(P_{d,\alpha,\beta})$ and their duals, the normal cones $N_G(P_{d,\alpha,\beta})$.
%It will be more convenient to do this in a deterministic case.

\paragraph{Description of the external angles:}
Let $x_1,\dots,x_{k+1}\in\RR^d$, for $k\in\{0,\dots,d-1\}$, be affinely independent points and denote their affine hull by $A_x:=\aff(x_1,\dots,x_{k+1})$. Then, we have $\dim A_x=k$ while the linear subspace $A_x^\perp$ has dimension $d-k$.
Furthermore, let $\pi_{A_x^\perp}:\RR^d\to A_x^\perp$ be the orthogonal projection onto $A_x^\perp$ and recall that $I_{A_x^\perp}:A_x^\perp\to\RR^{d-k}$ is an isometry such that $I_{A_x^\perp}(0)=0$.
Now, we consider a realization $x_{k+2},x_{k+3},\ldots\in\RR^{d}\backslash\widebar\BB^d$ of the Poisson process $\zeta_{d,\alpha,\beta}$ and define the points
\begin{align*}
y:=\pi_{A_x^\perp}(x_1)=\ldots=\pi_{A_x^\perp}(x_{k+1}),
\qquad
y_i:=\pi_{A_x^\perp}(x_{k+i+1})\in A_x^\perp,\quad i\in\NN.
\end{align*}
We are interested in the polytope $P:=\conv(x_i:i\in\NN)$. At first, assume $G:=\conv(x_1,\dots,x_{k+1})$ is a $k$-face of $P$ and note that this already implies that $A_x\cap \widebar\BB^d=\varnothing$ a.s.\ because  $P$ contains $r\widebar\BB^d$ for some $r>1$ by Proposition~\ref{thm:PolytopeOrNot}.
Defining $\bar x=(x_1+\ldots+x_{k+1})/(k+1)$, which lies in the relative interior of $G$, we observe that the tangent cone of $P$ at $G$ is given by
\begin{align*}
T_G(P)
%&	=\{v\in\RR^d:\bar x+\eps v\in P\text{ for some $\eps>0$}\}\\
%&	=\{v\in\RR^d: \eps v\in P-\bar x\text{ for some $\eps>0$}\}\\
=\pos(x_i-\bar x: i\in\NN).
\end{align*}
But since the positive hull of $x_1-\bar x,\dots,x_{k+1}-\bar x$ equals $A_x-\bar x$, we can write the tangent cone as an orthogonal sum
\begin{align*}
T_G(P)=(A_x-\bar x)\oplus \pos(y_i-y:i\in\NN).
\end{align*}
For convenience we map the points $y_i$ and $y$ to $\RR^{d-k}$ by considering $y_i':=I_{A_x^\perp}(y_i)\in\RR^{d-k}$, $i\in\NN$, and $y':=I_{A_x^\perp}(y)\in\RR^{d-k}$. Using the isometry property of $I_{A_x^\perp}$ and the fact that $(A_x-\bar x)$ is a linear subspace, we obtain that the external angle $\gamma(G,P)$ is given by
\begin{align}\label{eq:ext_angle}
\gamma(G,P)
&	=\measuredangle\big((\pos(y_i'-y':i\in\NN))^\circ\big)
	=\PP\big[\langle y_1'-y',N\rangle \le 0,\langle y_2'-y',N\rangle \le 0,\dots\big],
\end{align}
where $N$ is standard normal random vector in $\RR^{d-k}$ and $\measuredangle (C)$ denotes the solid angle of a cone $C$ (normalized so that the full-space angle is $1$).
The above holds if $G$ is a face of $P$.
In the case where $G$ is \emph{not} a face of $P$ the external angle $\gamma(G,P)$ vanishes by definition, as does the probability on the right-hand side of~\eqref{eq:ext_angle}. To prove the latter claim,  observe that $G$ is not a face of $P$ if and only if $\pi_{A_x^\perp}(x_1)=y$ is not a vertex of $\pi_{A_x^\perp}(P)=\conv(y,y_1,y_2,\dots)$, which holds if and only if $\pos(y_i-y:i\in\NN) = A_x^\perp$  since $y_1-y,y_2-y,\ldots$ are in general position a.s.

\begin{proof}[Proof of Theorem~\ref{thm:exp_ext_anlge_sums}]
Since each $k$-face of $P_{d,\alpha,\beta}$ is of the form $\conv(X_1,\dots,X_{k+1})$ for affinely independent  points $X_1,\dots,X_{k+1}$ from the Poisson process $\zeta_{d,\alpha,\beta}$, we can write, using the Mecke formula from Proposition~\ref{prop:Slivnyak-Mecke},
\begin{align}
&\EE\Bigg[\sum_{G\in\cF_k(P_{d,\alpha,\beta})}\gamma(G,P_{d,\alpha,\beta})\Bigg]\notag\\
&	\quad=\frac{1}{(k+1)!}\EE\Bigg[\sum_{(X_1,\dots,X_ {k+1})\in(\zeta_{d,\alpha,\beta})^{k+1}_{\neq}}\gamma\big(\conv(X_1,\dots,X_{k+1}),P_{d,\alpha,\beta}\big)\notag\\
&\notag\hspace{7cm}\times\ind_{\big\{\conv(X_1,\dots,X_{k+1})\in\cF_k\big(P_{d,\alpha,\beta}\big)  \big\}}\Bigg]\notag\\
&	\quad=\frac{1}{(k+1)!}\int\limits_{(\RR^d\backslash\widebar\BB^d)^{k+1}}\EE\Big[\gamma\big(\conv(x_1,\dots,x_{k+1}),P\big)\notag\\
&\hspace{7cm}\times\ind_{\big\{\conv(x_1,\dots,x_{k+1})\in\cF_k(P)  \big\}}\Big]\,\bigg(\prod_{i=1}^{k+1}\Phi_{d,\alpha,\beta}(\dint x_i)\bigg),\label{eq:integra_ext_angles}
\end{align}
where we used the notation $P:=\conv(\zeta_{d,\alpha,\beta}\cup\{x_1,\dots,x_{k+1}\})$ and recall that $\Phi_{d,\alpha,\beta}$ denotes the infinite measure on $\RR^d\backslash\widebar\BB^d$ with Lebesgue density $f_{d,\alpha,\beta}$. Without changing the integral, we can restrict the integration limits to the set $\mathcal S_{k+1,d}=\{(x_1,\dots,x_{k+1})\in(\RR^{d})^{k+1}: A_x\cap\widebar \BB^{d}=\varnothing\}$. Thus, following the arguments from the beginning of Section~\ref{sec:proof_ext_angles}, we can rewrite the expectation inside the integral as follows:
\begin{align*}
&\EE\left[\gamma\big(\conv(x_1,\dots,x_{k+1}),P\big)\,\ind_{\big\{\conv(x_1,\dots,x_{k+1})\in\cF_k(P)  \big\}}\right]\\
&	\qquad=\EE\bigg[\,\int\limits_{\RR^{d-k}}\ind_{\big\{\forall z\in I_{A_x^\perp}(\pi_{A_x^\perp}(\zeta_{d,\alpha,\beta})):\langle z-I_{A_x^\perp}(\pi_{A_x^\perp}(x_1)),g\rangle \le 0\big\}}\, (2\pi)^{-\frac{d-k}{2}} e^{-\|g\|^2/2} \dint g\bigg].
\end{align*}
%where $\text{N}(0,\mathbbm 1)$ denotes the standard normal distribution on $\RR^{d-k}$.
Now, we observe that the Poisson process $I_{A_x^\perp}(\pi_{A_x^\perp}(\zeta_{d,\alpha,\beta}))$ on $\RR^{d-k}$ has the same distribution as $\zeta_{d-k,\alpha,\beta-\frac{k}2}$ due to Lemma~\ref{lem:Projection}. In particular, it is rotationally invariant. Note that the distribution does not depend on the initial subspace $A_x$, and hence, not on the points $x_1,\dots,x_{k+1}$. Also, according to Theorem~\ref{thm:canonical_decomp} (b), we know that the measure $\big(\bigotimes_{i=1}^{k+1}\Phi_{d,\alpha,\beta}\big)\big|_{\mathcal S_{k+1,d}} \circ T_2^{-1}$ on $\RR^{d-k}\backslash\widebar \BB^{d-k}$ is the beta$^*$ measure  $\Phi_{d-k,\alpha^{k+1},(k+1)\beta-\frac{(d+1)k}{2}}$.
Here, we used the notation $T_2(x)=I_{A_x^\perp}(\pi_{A_x^\perp}(x_1))$ from Theorem~\ref{thm:canonical_decomp}. Inserting all of this into~\eqref{eq:integra_ext_angles} yields
\begin{align*}
&\EE\Bigg[\sum_{G\in\cF_k(P_{d,\alpha,\beta})}\gamma(G,P_{d,\alpha,\beta})\Bigg]\\
&	\qquad=\frac{1}{(k+1)!}\int\limits_{\RR^{d-k}\backslash\widebar \BB^{d-k}}\EE\bigg[\,\int\limits_{\RR^{d-k}}\ind_{\Big\{\forall z\in \zeta_{d-k,\alpha,\beta-\frac{k}2}:\,\langle z-y,g\rangle \le 0\Big\}}\, (2\pi)^{-\frac{d-k}{2}} e^{-\|g\|^2/2} \dint g\bigg]\\
&\hspace*{10.5cm}\times\Phi_{d-k,\alpha^{k+1},(k+1)\beta-\frac{(d+1)k}{2}}(\dint y)\\
&	\qquad=\frac{1}{(k+1)!}\int\limits_{\RR^{d-k}\backslash\widebar \BB^{d-k}}\PP\Big[\forall z\in \zeta_{d-k,\alpha,\beta-\frac{k}2}:\langle z-y,e\rangle \le 0\Big]\,\Phi_{d-k,\alpha^{k+1},(k+1)\beta-\frac{(d+1)k}{2}}(\dint y),
\end{align*}
where $e$ is some arbitrary unit vector in $\RR^{d-k}$. In the last step we used the rotational invariance of the Poisson process $\zeta_{d-k,\alpha,\beta-\frac{k}2}$ and of the beta$^*$ measure $\Phi_{d-k,\alpha^{k+1},(k+1)\beta-\frac{(d+1)k}{2}}$.
%The last step follows from the rotational invariance of the standard normal distribution and the Poisson process $\zeta_{d-k,\alpha,\beta-\frac{k}2}$.
%and the $\Phi_{d-k,\alpha^{k+1},(k+1)\beta-\frac{(d+1)k}{2}}$.
Choosing $e=(1,0,\dots,0)\in\RR^{d-k}$, we obtain, using again Lemma~\ref{lem:Projection} in the second step,
\begin{align*}
&\EE\Bigg[\sum_{G\in\cF_k(P_{d,\alpha,\beta})}\gamma(G,P_{d,\alpha,\beta})\Bigg]\\
&	\quad=\frac{1}{(k+1)!}\int\limits_{\RR^{d-k}\backslash\widebar \BB^{d-k}}\PP\big[\forall z\in\zeta_{d-k,\alpha,\beta-{k\over 2}}:\langle z,e\rangle\le\langle y,e\rangle \big]\,\Phi_{d-k,\alpha^{k+1},(k+1)\beta-\frac{(d+1)k}{2}}(\dint y)\\
& \quad	=\frac{1}{(k+1)!}\int\limits_{\RR\backslash [-1,1]}\PP\big[\forall z'\in \pi_1(\zeta_{d-k,\alpha,\beta-{k\over 2}}):z'\le y' \big]\,\Phi_{1,\alpha^{k+1},(k+1)\beta-\frac{d(k+1)-1}{2}}(\dint y'),
\end{align*}
where $\pi_1:\RR^{d-k}\to\RR$ denotes the projection onto the first coordinate. The same lemma also implies that the projected Poisson process $\pi_{1}(\zeta_{d-k,\alpha,\beta-\frac{k}2})$ has the same distribution as $\zeta_{1,\alpha,\beta-\frac{d-1}{2}}$. This process has atoms clustering at $\pm1$. Therefore, the probability under the integral sign vanishes for $y'\leq 1$. Let in the following $y'>1$. Then,
\begin{align}
\PP\big[\forall z'\in \pi_{1}(\zeta_{d-k,\alpha,\beta-\frac{k}2}):z'\le y' \big]
&	=\PP\left[\zeta_{1,\alpha,\beta-\frac{d-1}{2}}\cap (y',\infty)=\varnothing \right]\notag\\
&	=\exp\Bigg\{-\int\limits_{y'}^\infty f_{1,\alpha,\beta-\frac{d-1}{2}}(t)\,\dint t \Bigg\}
\notag\\&	=\exp\Bigg\{-\alpha\,\tilde c_{1,\beta-\frac{d-1}{2}}\int\limits_{y'}^\infty (t^2-1)^{-(\beta-\frac{d-1}{2})} \,\dint t\Bigg\}.\label{eq:exponential_term}
\end{align}
Setting $\gamma:=(k+1)\beta-\frac{d(k+1)-1}{2}$, and using the definitions~\eqref{eq:def_I_star_sum} and~\eqref{eq:def_I_star} in the second step yields
\begin{align*}
&\EE\Bigg[\sum_{G\in\cF_k(P_{d,\alpha,\beta})}\gamma(G,P_{d,\alpha,\beta})\Bigg]\\
&	\qquad=\frac{\alpha^{k+1}}{(k+1)!}\int\limits_{1}^\infty\tilde c_{1,\gamma} \, (y^2-1)^{-\gamma}\exp\bigg\{-\alpha\cdot\tilde c_{1,\beta-\frac{d-1}{2}}\int\limits_{y}^\infty (t^2-1)^{-(\beta-\frac{d-1}{2})} \,\dint t \bigg\}\,\dint y\\
&	\qquad=\mathbb I^*_{\alpha,k+1}(2\beta-d).
\end{align*}
This completes the proof.
\end{proof}

\subsection{Expected \texorpdfstring{$f$}{f}-vector: Proof of Theorem~\ref{thm:f-vector_beta_star}}\label{sec:proof_f-vector}

This section contains the proof of Theorem~\ref{thm:f-vector_beta_star} on the expected $f$-vector of $P_{d,\alpha,\beta}$. We divide the proof into three parts. In the first part, we sketch the idea of the proof and use a formula by Affentranger and Schneider~\cite{AffentrangerSchneider} to reduce the expected $f$-vector to a formula containing external and internal angles. In the second step we use the canonical decomposition from Theorem~\ref{thm:canonical_decomp} to separate the expectations of the said angles. In the third part, we finally compute the formula for $\EE f_k(P_{d,\alpha,\beta})$.

\paragraph{Idea of the proof:}
Suppose that either $\beta>(d+1)/2$ and $\alpha>0$ or $\beta=(d+1)/2$ and $\alpha > (d-1)\pi$. In order to compute the expected $f$-vector of the beta$^*$-polytope $P_{d,\alpha,\beta}$, which is defined as the convex hull of the Poisson process $\zeta_{d,\alpha,\beta}$ in $\RR^d\backslash \widebar\BB^d$, we shall represent this polytope as a random uniform projection of a higher dimensional polytope. The projection property from Lemma~\ref{lem:Projection} implies that, for some $\ell\in\NN$, the projection of the Poisson process $\zeta_{d+\ell,\alpha,\beta+\frac\ell 2}$ in $\RR^{d+\ell}\backslash \widebar \BB^{d+\ell}$ to $\RR^d\backslash \widebar\BB^d$ has the same distribution as the Poisson process $\zeta_{d,\alpha,\beta}$. Now, taking a random and uniformly distributed $d$-dimensional subspace $L_d\in G(d+\ell,d)$ and denoting by $\pi_{L_d}$ the orthogonal projection onto $L_d$, we observe that
\begin{align*}
f_k\big(\pi_{L_d} P_{d+\ell,\alpha,\beta+\frac \ell 2}\big)\overset{d}{=}f_k\big(P_{d,\alpha,\beta}\big),
\end{align*}
where $\overset{d}{=}$ stands for equality in distribution.
Thus, we can reduce the expectation of the right-hand side to the expectation of the left-hand side, which can be computed by using the following formula due to Affentranger and Schneider~\cite{AffentrangerSchneider}. For any polytope $P$ of dimension $\dim P$, $d\in\{1,\dots,\dim P\}$ and $k\in\{0,\dots,d-1\}$ one has that
\begin{align}\label{eq:AffentrangerSchneider}
\EE f_k(\pi_{L_d} P)=2\sum_{s=0}^\infty\:\sum_{G\in\cF_{d-1-2s}(P)}\gamma(G,P)\sum_{F\in\cF_k(G)}\beta(F,G),
\end{align}
where we recall that $\gamma(G,P)$ denotes the external angle of $P$ at its face $G$ while $\beta(F,G)$ denotes the internal angle of $G$ at its face $F$.
Applying formula~\eqref{eq:AffentrangerSchneider} to our case yields
\begin{align*}
\EE\Big[f_k\big(\pi_{L_d} P_{d+\ell,\alpha,\beta+\frac \ell 2}\big)\Big|P_{d+\ell,\alpha,\beta+\frac \ell 2}\Big]
	=2\sum_{s=0}^\infty\:\sum_{G\in\cF_{d-1-2s}(P_{d+\ell,\alpha,\beta+\frac \ell 2})}\gamma\big(G,P_{d+\ell,\alpha,\beta+\frac \ell 2}\big)\sum_{F\in\cF_k(G)}\beta(F,G).
\end{align*}
Hence, we can take the expectation with respect to the random set $P_{d+\ell,\alpha,\beta+\frac \ell 2}$ of both sides to arrive at
\begin{align}
\EE\big[f_k\big(P_{d,\alpha,\beta}\big)\big]
&	=2\sum_{s=0}^\infty \EE\left[\sum_{G\in\cF_{d-1-2s}(P_{d+\ell,\alpha,\beta+\frac \ell 2})}\gamma\big(G,P_{d+\ell,\alpha,\beta+\frac \ell 2}\big)\sum_{F\in\cF_k(G)}\beta(F,G)  \right].\label{eq:sum_int_ext1}
\end{align}
It turns out that, in some sense,  we separate the sum over the external angles and the sum over the internal angles.  Clearly, we cannot argue that both angle sums are independent, since the sum over the internal angles depends on $G$.  The main ingredient in separating the external and internal angles is the canonical decomposition in Theorem~\ref{thm:canonical_decomp}.

\paragraph{Separating the internal and external angles:}
With the same arguments as in the proof of Theorem~\ref{thm:exp_ext_anlge_sums} (with $k$ replaced by $d-2s-1$, $d$ replaced by $d+\ell$ and $\beta$ replaced by $\beta+\frac \ell 2$), we can rewrite the expectation in line~\eqref{eq:sum_int_ext1} as follows:
\begin{multline*}
\frac{1}{(d-2s)!}\EE \Bigg[\sum_{(X_1,\dots,X_{d-2s})\in(\zeta_{d+\ell,\alpha,\beta+\frac\ell 2})^{d-2s}_{\neq}} \gamma\Big(\conv(X_1,\dots,X_{d-2s}),P_{d+\ell,\alpha,\beta+\frac \ell 2}\Big)\\
	\times \ind_{\Big\{\conv(X_1,\dots,X_{d-2s})\in\cF_{d-1-2s}\big(P_{d+\ell,\alpha,\beta+\frac \ell 2}\big)\Big\}} \sum_{F\in\cF_k(\conv(X_1,\dots,X_{d-2s}))}\beta(F,\conv(X_1,\dots,X_{d-2s}))\Bigg].
\end{multline*}
The Mecke formula from Proposition~\ref{prop:Slivnyak-Mecke} applied to the expectation above yields that $\EE[f_k(P_{d,\alpha,\beta})]$ is equal to
\begin{align}
&\sum_{s=0}^\infty \frac{2}{(d-2s)!}\int\limits_{(\RR^{d+\ell})^{d-2s}}\EE\Bigg[\gamma\big(\conv(x_1,\dots,x_{d-2s}),P\big)\ind_{\{\conv(x_1,\dots,x_{d-2s})\in\cF_{d-1-2s}(P) \}}\notag\\
&	\quad\times\sum_{F\in\cF_k(\conv(x_1,\dots,x_{d-2s}))}\beta(F,\conv(x_1,\dots,x_{d-2s}))\Bigg]\bigg(\prod_{i=1}^{d-2s}\Phi_{d+\ell,\alpha,\beta+\frac\ell 2}(\dint x_i)\bigg)\notag\\
&	=\sum_{s=0}^\infty \frac{2}{(d-2s)!}\int\limits_{(\RR^{d+\ell})^{d-2s}}\sum_{F\in\cF_k(\conv(x_1,\dots,x_{d-2s}))}\beta(F,\conv(x_1,\dots,x_{d-2s}))\label{eq:prod_ext_int_angles}\\
&	\quad\times \EE\Bigg[\gamma\Big(\conv(x_1,\dots,x_{d-2s}),P\Big)\ind_{\{\conv(x_1,\dots,x_{d-2s})\in\cF_{d-1-2s}(P) \}}\Bigg]\bigg(\prod_{i=1}^{d-2s}\Phi_{d+\ell,\alpha,\beta+\frac\ell 2}(\dint x_i)\bigg),\notag
\end{align}
where we used the notation $P:=\conv(\zeta_{d+\ell,\alpha,\beta+\frac \ell 2}\cup\{x_1,\dots,x_{d-2s}\})$. Up to the sum of the internal angles, the above summands  already occurred in line~\eqref{eq:integra_ext_angles} and were evaluated in the subsequent proof. Following the same reasoning, we can rewrite~\eqref{eq:prod_ext_int_angles} to obtain
\begin{multline*}
\EE\big[f_k\big(P_{d,\alpha,\beta}\big)\big]
	=2\sum_{s=0}^\infty \frac{1}{(d-2s)!}\int\limits_{\mathcal S_{d-2s,d+\ell}}\sum_{F\in\cF_k(\conv(x_1,\dots,x_{d-2s}))}\beta(F,\conv(x_1,\dots,x_{d-2s}))\\
	\times \exp\Bigg\{-\alpha\int\limits_{\langle T_2(x_1,\dots,x_{d-2s}),e_1\rangle}^\infty \tilde c_{1,\beta-\frac{d-1}{2}}(t^2-1)^{-(\beta-\frac{d-1}{2})}\,\dint t\Bigg\} \bigg(\prod_{i=1}^{d-2s}\Phi_{d+\ell,\alpha,\beta+\frac\ell 2}(\dint x_i)\bigg),
\end{multline*}
with $T_2(x_1,\dots,x_{d-2s}):=I_{A_x^\perp}(\pi_{A_x^\perp}(x_1))=I_{A_x^\perp}(p(A_x))$.
Since the internal angle is invariant under isometries and rescalings, the sum
$
\sum_{F\in\cF_k(\conv(x_1,\dots,x_{d-2s}))}\beta(F,\conv(x_1,\dots,x_{d-2s}))
$
does not change if we replace $(x_1,\dots,x_{d-2s})$ by
\begin{align*}
T_1(x_1,\dots,x_{d-2s})=\left(\frac{I_{A_x}(x_1)}{\sqrt{\dist^2(A_x)-1}},\dots,\frac{I_{A_x}(x_{d-2s})}{\sqrt{\dist^2(A_x)-1}}\right),
\end{align*}
where we recall that the two functions $T_1$ and $T_2$ are defined in the same way as in Theorem~\ref{thm:canonical_decomp}. Recalling that $T=(T_1,T_2)$, we obtain
\begin{align*}
\EE\big[f_k\big(P_{d,\alpha,\beta}\big)\big]
&	=2\sum_{s=0}^\infty \frac{1}{(d-2s)!}\int\limits_{\mathcal S_{d-2s,d+\ell}}\sum_{F\in\cF_k(\conv(T_1(x_1,\dots,x_{d-2s}))}\beta(F,\conv(T_1(x_1,\dots,x_{d-2s}))\\
&	\quad\times \exp\Bigg\{-\alpha\int\limits_{\langle T_2(x_1,\dots,x_{d-2s}),e_1\rangle}^\infty \hspace*{-10pt}\tilde c_{1,\beta-\frac{d-1}{2}}(t^2-1)^{-(\beta-\frac{d-1}{2})}\,\dint t\Bigg\} \bigg(\prod_{i=1}^{d-2s}\Phi_{d+\ell,\alpha,\beta+\frac\ell 2}(\dint x_i)\bigg)\\
&	=2\sum_{s=0}^\infty \frac{1}{(d-2s)!}\int\limits_{(\RR^{d-2s-1})^{d-2s}\times\RR^{\ell+2s+1}} \sum_{F\in\cF_k(\conv(z_1,\dots,z_{d-2s}))}\hspace*{-1pt}\beta(F,\conv(z_1,\dots,z_{d-2s}))\notag\\
&	\quad\times \exp\Bigg\{-\alpha\int\limits_{\langle y,e_1\rangle}^\infty \tilde c_{1,\beta-\frac{d-1}{2}}(t^2-1)^{-(\beta-\frac{d-1}{2})}\,\dint t\Bigg\}\\
&	\quad\times\Bigg(\bigg(\prod_{i=1}^{d-2s}\Phi_{d+\ell,\alpha,\beta+\frac\ell 2}\bigg)\bigg|_{\mathcal S_{d-2s,d+\ell}}\circ T^{-1}\Bigg)(\dint (z_1,\dots,z_{d-2s},y)).
\end{align*}
Hence, part (a) of Theorem~\ref{thm:canonical_decomp} yields
\begin{align}
&\EE\big[f_k\big(P_{d,\alpha,\beta}\big)\big]\notag\\
&	\quad=2\sum_{s=0}^\infty \Bigg(\,\int\limits_{(\RR^{d-2s-1})^{d-2s}}\sum_{F\in\cF_{k}(\conv(z_1,\dots,z_{d-2s}))}\beta(F,\conv(z_1,\dots,z_{d-2s}))\phi_{d-2s-1}(\dint(z_1,\dots,z_{d-2s})\Bigg)\label{eq:integral_internal_angle}\\
&	\qquad\times\Bigg(\frac{1}{(d-2s)!}\int\limits_{\RR^{\ell+2s+1}\backslash\widebar \BB^{\ell+2s+1}}\exp\Bigg\{-\int\limits_{\langle y,e_1\rangle}^{\infty}f_{1,\alpha,\beta-\frac{d-1}{2}}(t)\,\dint t\Bigg\}f_{\ell+2s+1,\alpha^{d-2s},\gamma}(y)\,\dint y\Bigg),\label{eq:integral_external_angle}
\end{align}
where
\begin{align*}
\gamma:=(d-2s)\Big(\beta+\frac \ell 2\Big)-\frac{(d+\ell+1)(d-2s-1)}{2},
\end{align*}
and $\phi_{d-2s-1}$ is the probability measure on $(\RR^{d-2s-1})^{d-2s}$ from Theorem~\ref{thm:canonical_decomp} whose Lebesgue density is a constant multiple of
\begin{align*}
V_{d-2s-1}^{\ell+2s+1}(\conv(z_1,\dots,z_{d-2s}))\bigg(\prod_{i=1}^{d-2s} \tilde f_{d-2s-1,\beta+\frac \ell 2}(z_i)\bigg).
\end{align*}
We have thus achieved our goal to separate the internal and external angles.

\paragraph{Expected $f$-vector:}
Finally, we need to evaluate the integrals in lines~\eqref{eq:integral_internal_angle} and~\eqref{eq:integral_external_angle}.
As already explained in the proof of Theorem~\ref{thm:exp_ext_anlge_sums}, the integral in line~\eqref{eq:integral_external_angle} can be simplified as follows. Since the image measure of $\Phi_{\ell+2s+1,\alpha,\gamma}$ under orthogonal projection onto $\RR\times\{0\}^{\ell+2s}$ is given by $\Phi_{1,\alpha,\beta(d-2s)-\frac{d(d-2s)-1}{2}}$, due to Lemma~\ref{lem:Projection},
we obtain
\begin{align*}
&\frac{1}{(d-2s)!}\int\limits_{\RR^{\ell+2s+1}\backslash \widebar \BB^{\ell+2s+1}}\exp\Bigg\{-\int\limits_{\langle y,e_1\rangle}^{\infty}f_{1,\alpha,\beta-\frac{d-1}{2}}(y)\,\dint t\Bigg\}f_{\ell+2s+1,\alpha^{d-2s},\gamma}(y)\,\dint y\\
&\quad=\frac{1}{(d-2s)!}\int\limits_{\RR\backslash [-1,1]}\exp\Bigg\{-\int\limits_s^\infty f_{1,\alpha,\beta-\frac{d-1}{2}}(t)\,\dint t\Bigg\}f_{1,\alpha^{d-2s},\beta(d-2s)-\frac{d(d-2s)-1}{2}}(s)\,\dint s\\
&	\quad=\frac{\alpha^{d-2s}}{(d-2s)!}\int\limits_{1}^\infty \exp\Bigg\{-\int\limits_s^\infty f_{1,\alpha,\beta-\frac{d-1}{2}}(t)\,\dint t\Bigg\}f_{1,1,\beta(d-2s)-\frac{d(d-2s)-1}{2}}(s)\,\dint s\\
&	\quad=\mathbb I_{\alpha,d-2s}^*(2\beta-d).
\end{align*}
Now, we want to evaluate the integral in line~\eqref{eq:integral_internal_angle}. Since $\phi_{d-2s-1}$ is a probability measure we will denote this integral as an expectation. Let $Z_1,\dots,Z_{d-2s}$ be random vectors in $\RR^{d-2s-1}$ whose joint distribution is $\phi_{d-2s-1}$. Using the exchangeability of $(Z_1,\dots,Z_{d-2s})$ and the fact that each $k$-face of $\conv(Z_1,\dots,Z_{d-2s})$ is almost surely of the form $\conv(Z_{i_1},\dots,Z_{i_{k+1}})$ for some indices $1\le i_1<\ldots<i_{k+1}\le {d-2s}$, we obtain
\begin{align*}
&\int\limits_{(\RR^{d-2s-1})^{d-2s}}\sum_{F\in\cF_k(z_1,\dots,z_{d-2s})}\beta(F,\conv(z_1,\dots,z_{d-2s}))\phi_{d-2s-1}(\dint(z_1,\dots,z_{d-2s}))\\
&	\quad=\binom{d-2s}{k+1}\EE\big[\beta\big(\conv(Z_1,\dots,Z_{k+1}),\conv(Z_1,\dots,Z_{d-2s})\big)\big].
\end{align*}
The analogue to Theorem~\ref{thm:canonical_decomp} for beta'-densities, which was proven in~\cite[Theorem~3.6]{KabluchkoThaeleZaporozhets}, states that $(Z_1,\dots,Z_{d-2s})$ has the same distribution as
\begin{align*}
\left(\frac{I_{A}(X_1)}{\sqrt{1+\dist^2(A)}},\dots,\frac{I_{A}(X_{d-2s})}{\sqrt{1+\dist^2(A)}}\right),
\end{align*}
where $X_1,\dots,X_{d-2s}$ are i.i.d. points in $\RR^{d+\ell}$ with density $\tilde f_{d+\ell,\beta+\frac\ell 2}$ and $A:=\aff(X_1,\dots,X_{d-2s})$. But the internal angles do not change under rescaling and isometry, which yields
\begin{align*}
&\binom{d-2s}{k+1}\EE\big[\beta\big(\conv(Z_1,\dots,Z_{k+1}),\conv(Z_1,\dots,Z_{d-2s})\big)\big]\\
&	\quad=\binom{d-2s}{k+1}\EE\big[\beta\big(\conv(X_1,\dots,X_{k+1}),\conv(X_1,\dots,X_{d-2s})\big)\big]\\
&	\quad=\mathbb{\tilde J}_{d-2s,k+1}\Big(\beta-s-\frac 12\Big).
\end{align*}
The last equation follows from~\cite[Theorem~4.1]{KabluchkoThaeleZaporozhets} (with $d$ replaced by $d-2s-1$). This completes the proof of Theorem~\ref{thm:f-vector_beta_star}. \hfill $\Box$

\section{Asymptotics for large intensities and monotonicity: Proofs} \label{sec:large_intensities_proofs}

%\begin{proof}[Proof of Theorem~\ref{thm:PolytopesConverge}]
\subsection{Proof of Theorem~\ref{thm:PolytopesConverge}}
We shall provide two proofs. The first one has the advantage that it could be applied in a much more general setting, but it requires that we are dealing with polytopes. The second proof uses special monotonicity features of the model and is valid for all $\beta>d/2$.

\begin{proof}[First proof of Theorem~\ref{thm:PolytopesConverge}]
In this proof it is assumed that the intensity $\alpha$ and the parameter $\beta$ are such that $P_{d,\alpha,\beta}$ is almost surely a polytope.

In order to prove Theorem~\ref{thm:PolytopesConverge} we use a Scheff\'e-type lemma addressing the convergence of random polytopes from~\cite[Proposition~2.3]{KabluchkoTemesvariThaeleConvergence}, and in what follows we also use the same notation as in that paper in order to simplify comparison. To this end, let $\cP_m^d$ be the set of $d$-dimensional polytopes with exactly $m\geq d+1$ vertices whose last coordinates are all distinct. Each such polytope $P$ can uniquely be represented as convex hull of $m$ points $x_1,\ldots,x_m\in\RR^d$ with $x_{i,d} < x_{j,d}$ for $1\leq i<j\leq m$ and we write $\iota^{-1}P$ for the vertex representation $(x_1,\ldots,x_m)$ of $P$. Also, put $\widetilde{\cP}_m^d:=\iota^{-1}(\cP_m^d)\subset(\RR^d)^m$ and let $\mu_m^d$ be the measure on $\widetilde{\cP}_m^d$ arising as the restriction of the $m$-fold product of the $d$-dimensional Lebesgue measure. Clearly, with probability one, the vertex representations of the rescaled beta$^*$ polytopes
$$
Q_\alpha:=\iota^{-1}(\alpha^{-{1\over 2\beta-d}}P_{d,\alpha,\beta}),
$$
where $\alpha>0$ if $\beta>(d+1)/2$ and $\alpha > (d-1)\pi$ if $\beta=(d+1)/2$, belong to the disjoint union $\widetilde{\cP}_\infty^d:=\coprod_{m=d+1}^\infty\widetilde{\cP}_m^d$, which we supply with the infinite sum $\mu_\infty^d$ of the measures $\mu_m^d$. For every $m\geq d+1$ the density $\varphi_\alpha$ of $Q_\alpha$ with respect to $\mu_\infty^d$ is given by
\begin{align*}
\varphi_\alpha(x_1,\ldots,x_m) &= \Big(\prod_{i=1}^m\alpha^{d\over 2\beta-d}f_{d,\alpha,\beta}(\alpha^{1\over 2\beta-d}x_i)\ind_{\{\|\alpha^{1/(2\beta-d)}x_i\|>1\}}\Big)\\
&\quad\times\exp\Bigg\{-\int\limits_{\RR^d\backslash\conv(x_1,\ldots,x_m)}\alpha^{d\over 2\beta-d}f_{d,\alpha,\beta}(\alpha^{1\over 2\beta-d}y)\ind_{\{\|\alpha^{1/(2\beta-d)}y\|>1\}}\,\dint y\Bigg\},
\end{align*}
recall~\eqref{eq:convergence_density_beta$^*$}. Here, the first factor reflects the part of the density corresponding to the $m$ vertices $x_1,\ldots,x_m$ of $Q_\alpha$, while the second factor is the probability that all other points of the Poisson process generating $Q_\alpha$ belong to the convex hull of these vertices. Using the definition of $f_{d,\alpha,\beta}$ this can be rewritten as
\begin{align*}
\varphi_\alpha(x_1,\ldots,x_m) &= \Big(\prod_{i=1}^m \alpha^{{d\over 2\beta-d}+1}\tilde{c}_{d,\beta}(\|\alpha^{1\over 2\beta-d}x_i\|^2-1)^{-\beta}\ind_{\{\|\alpha^{1/(2\beta-d)}x_i\|>1\}}\Big)\\
&\quad\times\exp\Bigg\{-\alpha^{{d\over 2\beta-d}+1}\tilde{c}_{d,\beta}\int\limits_{\RR^d\backslash\conv(x_1,\ldots,x_m)}(\|\alpha^{1\over 2\beta-d}y\|^2-1)^{-\beta}\ind_{\{\|\alpha^{1/( 2\beta-d)}y\|>1\}}\,\dint y\Bigg\}\\
&=\Big(\prod_{i=1}^m \tilde{c}_{d,\beta}(\|x_i\|^2-\alpha^{-{2\over 2\beta-d}})^{-\beta}\ind_{\{\|\alpha^{1/(2\beta-d)}x_i\|>1\}}\Big)\\
&\quad\times\exp\Bigg\{-\tilde{c}_{d,\beta}\int\limits_{\RR^d\backslash\conv(x_1,\ldots,x_m)}(\|y\|^2-\alpha^{-{2\over 2\beta-d}})^{-\beta}\ind_{\{\|\alpha^{1/( 2\beta-d)}y\|>1\}}\,\dint y\Bigg\},
\end{align*}
see again~\eqref{eq:convergence_density_beta$^*$}. Letting $\alpha\to\infty$, this converges to
\begin{align*}
\varphi(x_1,\ldots,x_m) &:= \Big(\prod_{i=1}^m \tilde{c}_{d,\beta}\|x_i\|^{-2\beta}\Big)\exp\Bigg\{-\tilde{c}_{d,\beta}\int\limits_{\RR^d\backslash\conv(x_1,\ldots,x_m)}\|y\|^{-2\beta}\,\dint y\Bigg\}
\end{align*}
for $(x_1,\ldots,x_m)\in\cP_m^d$ with $x_i\neq 0$ for all $i\in\{1,\ldots,m\}$, which shows the pointwise convergence of $\varphi_\alpha$ to $\varphi$ on $\widetilde{\cP}_\infty^d$. However, $\varphi$ is precisely the density on $\widetilde{\cP}_\infty^d$ with respect to $\mu_\infty^d$ of the convex hull of a Poisson process with Lebesgue intensity $\tilde{c}_{d,\beta}\|x_i\|^{-2\beta}$, $x\in\RR^d\backslash\{0\}$, under the mapping $\iota^{-1}$. In Section~\ref{sec:Convergence}, we denoted this random convex hull by $\conv \chi_{d,1,\beta}$. Thus, an application of the Scheff\'e-type lemma for random polytopes~\cite[Proposition~2.3]{KabluchkoTemesvariThaeleConvergence} shows that, as $\alpha\to\infty$, $\alpha^{-1/(2\beta-d)}P_{d,\alpha,\beta}$ weakly converges to the Poisson polytope $\conv \chi_{d,1,\beta}$ on the space of compact convex subsets of $\RR^d$ endowed with the Hausdorff metric. This completes the proof of Theorem~\ref{thm:PolytopesConverge}.
\end{proof}
%\hfill$\Box$

\begin{proof}[Second proof of Theorem~\ref{thm:PolytopesConverge}]
The following proof of Theorem \ref{thm:PolytopesConverge} applies in the full range $\beta>d/2$. The convex set $\alpha^{-1/(2\beta-d)}P_{d,\alpha, \beta}$ is generated by a Poisson process with intensity
$$
h_{d,\alpha, \beta}(x) := \tilde c_{d,\beta} (\|x\|^2 - \alpha^{-\frac 2{2\beta-d}})^{-\beta}, \qquad
\|x\|>\alpha^{-\frac{1}{2\beta-d}}.
$$
Let us extend this definition by putting $h_{d,\alpha,\beta}(x) = +\infty$ if $\|x\|\leq \alpha^{-\frac{1}{2\beta-d}}$. Then, for every $x\in \RR^d$ we have
$$
\lim_{\alpha\to\infty} h_{d,\alpha, \beta}(x) = \tilde c_{d,\beta} \|x\|^{-2\beta} =: h_{d,\infty,\beta}(x).
$$
Moreover, observe that the function $\alpha \mapsto h_{d,\alpha,\beta}(x)$ is non-increasing in $\alpha>0$. This feature allows us to construct a coupling of all random sets of interest on a common probability space. To this end, let $(x_i,y_i)$, $i\in \NN$, be the atoms of a Poisson process on $\RR^d \times [0,\infty)$ whose intensity is the Lebesgue measure. Define the following random convex sets in $\RR^d$:
\begin{align*}
Q_{d,\alpha,\beta}
&=
\overline{\conv} (x_i: y_i \leq h_{d,\alpha, \beta}(x_i))\stackrel{d}{=}\alpha^{-1/(2\beta-d)}P_{d,\alpha, \beta},\\
Q_{d,\infty,\beta}
&=
\conv (x_i: y_i \leq h_{d,\infty, \beta}(x_i))\stackrel{d}{=} \conv \chi_{d,1,\beta}.
\end{align*}
We now claim that in our coupling there is an a.s.\ finite random variable $\alpha_0>0$ such that $Q_{d,\alpha,\beta} = Q_{d,\infty,\beta}$ for all $\alpha > \alpha_0$, which is stronger than the claim of Theorem~\ref{thm:PolytopesConverge}. In the rest of this proof, $\omega$ denotes some outcome of our random experiment and we write $\alpha_0=\alpha_0(\omega)$ to indicate that $\alpha_0$ is random. By definition, we have $Q_{d,\alpha,\beta} \supseteq Q_{d,\infty,\beta}$ for every $\alpha>0$. On the other hand, it is known~\cite[Corollary~4.2]{KabluchkoMarynychTemesvariThaele} that, with probability $1$, the random convex set  $Q_{d,\infty,\beta}$ contains a ball $r(\omega) \widebar\BB^d$ of certain random radius $r(\omega)>0$ and is a convex hull of finitely many points $x_i$, $i\in I(\omega)$, located outside  $r(\omega) \widebar\BB^d$. Since $h_{d,\alpha,\beta}(x)\downarrow h_{d,\infty,\beta}(x)$ as $\alpha\to\infty$, the number of atoms $(x_i, y_i)$, $i\in \NN$, that satisfy $\|x_i\| \geq r(\omega)$ and $h_{d,\infty,\beta}(x_i) \leq y_i \leq h_{d,\alpha,\beta}(x_i)$, converges to $0$ a.s. It follows that for sufficiently large $\alpha>\alpha_0(\omega)$, we have $Q_{d,\alpha,\beta} = Q_{d,\infty,\beta}$.
\end{proof}

%\begin{proof}[Proof of Proposition~\ref{prop:asymptotics_I_star}]
\subsection{Proof of Proposition~\ref{prop:asymptotics_I_star}}
We use that $\mathbb I^*_{\alpha,n}(\lambda)=\frac{\alpha^n}{n!} I^*_{\alpha,n}(\lambda)$ and substitute $y=\coth(w)$ in the definition~\eqref{eq:def_I_star} of $I^*_{\alpha,n}(\lambda)$ to obtain
	\begin{align}
		\mathbb I^*_{\alpha,n}(\lambda)
		&	=\frac{\alpha^n}{n!}\int\limits_0^\infty \tilde c_{1,\frac{\lambda n+1}{2}}(\coth^2w-1)^{-\frac{\lambda n+1}{2}+1}\exp\Bigg\{-\alpha\int\limits_{\coth w}^\infty \tilde c_{1,\frac{\lambda+1}{2}}(t^2-1)^{-\frac{\lambda+1}{2}}\,\dint t \Bigg\}\dint w.\notag\\
		&	=\frac{\alpha^n}{n!}\int\limits_0^\infty \tilde c_{1,\frac{\lambda n+1}{2}}(\sinh w)^{\lambda n-1}\exp\Bigg\{-\alpha\int\limits_{\coth w}^\infty \tilde c_{1,\frac{\lambda+1}{2}}(t^2-1)^{-\frac{\lambda+1}{2}}\,\dint t\Bigg\}\dint w\notag\\
		&	=:\frac{\alpha^n\tilde c_{1,\frac{\lambda n+1}{2}}}{n!}\int\limits_0^\infty \varphi(w)\cdot e^{-\alpha\cdot h(w)}\,\dint w.\label{eq:h_and_phi}
	\end{align}
	We are now applying~\cite[Theorem II.1.1]{Wong} to deduce the first two terms of the asymptotic expansion of the integral in the last line. To this end we need to check especially assumptions (II.1.9) -- (II.1.11) in~\cite{Wong}. Since the function $h:[0,\infty)\to\RR$ defined by~\eqref{eq:h_and_phi} attains its minimum $h(0)=0$ at $w=0$, we have to determine the asymptotic expansions of $h(w)$ and $\varphi(w)$, also defined by~\eqref{eq:h_and_phi}, at $w=0$. Using the same notation as in~\cite{Wong} it holds that
	\begin{align*}
		h(w)=\int\limits_{\coth w}^\infty \tilde c_{1,\frac{\lambda+1}{2}}(t^2-1)^{-\frac{\lambda+1}{2}}\,\dint t
		&	=\frac{\tilde c_{1,\frac{\lambda+1}{2}}}{\lambda}\cdot w^\lambda+\frac{(\lambda-1)\tilde c_{1,\frac{\lambda+1}{2}}}{6(\lambda+2)}\cdot w^{\lambda+2}+O\big(w^{\lambda+4}\big)\\
		&	=:a_0 \cdot w^\lambda+a_2\cdot w^{\lambda+2}+O\big(w^{\lambda+4}\big)
	\end{align*}
	and
	\begin{align*}
		\varphi(w)=(\sinh w)^{\lambda n-1}
		&	=w^{\lambda n-1}+\frac{\lambda n-1}{6}\cdot w^{\lambda n+1} + O\big(w^{\lambda n+3}\big)\\
		&	=:b_0\cdot w^{\lambda n-1}+b_2\cdot w^{\lambda +1}+O\big(w^{\lambda n+3}\big),
	\end{align*}
	as $w\to 0$.
	Putting $\mu:=\lambda$ and $\tilde {\alpha}:=\lambda n$  for the parameters $\mu$ and $\tilde \alpha$ defined as in (II.1.9) and (II.1.10) in~\cite{Wong} (our $\tilde \alpha$ corresponds to Wong's $\alpha$), respectively, and denoting
	\begin{align*}
		&c_0:=\frac{b_0}{\mu} \cdot a_0^{-\frac{\tilde\alpha}{\mu}}=\frac{\lambda^{n-1}}{(\tilde c_{1,\frac{\lambda+1}{2}})^n},\\
		&c_2:=a_0^{-\frac{\tilde\alpha+2}{\mu}}\left( \frac{b_2}{\mu}-2\mu a_0 a_2\cdot \frac{(\tilde \alpha +2)b_0}{2\mu^3a_0^2}\right)=(\tilde c_{1,\frac{\lambda+1}{2}})^{-(\frac 2\lambda +n)}\frac{\lambda^{\frac 2\lambda +n}(n-1)}{2(\lambda+2)},
	\end{align*}
	we can now apply~\cite[Theorem II.1.1]{Wong}. From this result it follows that
	\begin{align*}
		\int\limits_0^\infty \varphi(w)\cdot e^{-\alpha\cdot h(w)}\,\dint w =\Gamma(n)\frac{c_0}{\alpha^n}+\Gamma\bigg(\frac{\lambda n+2}{\lambda}\bigg)\frac{c_2}{\alpha^{\frac{\lambda n+2}{\lambda}}}+O\Big(\alpha^{-\frac{\lambda n+4}{\lambda}}\Big),	
	\end{align*}
	as $\alpha\to \infty$.
	In summary, we obtain
	\begin{align*}
		\mathbb I^*_{\alpha,n}(\lambda)=\frac{\lambda^{n-1}}{n}\frac{\tilde c_{1,\frac{\lambda n+1}{2}}}{(\tilde c_{1,\frac{\lambda+1}{2}})^n}+\frac{K_1(\lambda,n)}{\alpha^{\frac 2\lambda}}+O\Big(\alpha^{-\frac 4\lambda}\Big),
	\end{align*}
	where $K_1(\lambda,n)$ is explicitly given by
	\begin{align}
		K_1(\lambda,n)
		&	=\Gamma\bigg(\frac{\lambda n+2}{\lambda}\bigg)\cdot c_2\cdot\frac{\tilde c_{1,\frac{\lambda n+1}{2}}}{n!}\notag\\
		&	=\frac{\lambda^{n+\frac 2\lambda}(n-1)\pi^{\frac{n-1}{2}+\frac 1\lambda}}{2(\lambda+2)n!}\cdot\frac{\Gamma(\frac{\lambda n+1}{2})\Gamma(n+\frac 2\lambda)}{\Gamma(\frac{\lambda n}{2})}\cdot \left(\frac{\Gamma(\frac\lambda 2)}{\Gamma(\frac{\lambda+1}{2})} \right)^{n+\frac 2{\lambda}}.\label{eq:const(l,n)}
	\end{align}
	This completes the proof.
\hfill $\Box$
%\end{proof}

\subsection{Monotonicity: Proof of Theorem \ref{thm:Monotonicity}}
The proof we give is inspired by those of similar monotonicity results in \cite{Bonnet_etal} and especially \cite{KabluchkoThaeleZaporozhets} for the expected $f$-vector of beta and beta' polytopes in $\RR^d$.

\begin{proof}[Proof of Theorem \ref{thm:Monotonicity}]
Recalling the explicit formula for the expected $f$-vector from Theorem~\ref{thm:f-vector_beta_star} it is sufficient to show that the quantities $\mathbb{I}_{\alpha,m}^*(\lambda)$ given by \eqref{eq:def_I_star_sum} for $m\in\NN$ and $\alpha>0$ if $\lambda>1$ or $\alpha>(m-1)\pi$ if $\lambda=1$ are strictly monotone decreasing in $\alpha$. Since constants only depending on $m$ and $\lambda$ do not influence the monotonicity behaviour of $\mathbb{I}_{\alpha,m}^*(\lambda)$, it is sufficient to prove strict monotonicity in $\alpha$ of the function
\begin{align*}
	G(\alpha)
	:=
	{\alpha^m\over m!}\int\limits_1^\infty (y^2-1)^{-\frac{\lambda m+1}{2}} \exp\Bigg\{-\alpha\int\limits_{y}^\infty\tilde c_{1,\frac{\lambda+1}{2}}(t^2-1)^{-\frac{\lambda+1}{2}}\,\dint t\Bigg\}\,\dint y,
\end{align*}
where $\alpha>0$ if $\lambda>1$ or $\alpha>(m-1)\pi$ if $\lambda=1$, recall \eqref{eq:def_I_star}. To simplify our notation, we define
\begin{align*}
	f(y) := \tilde c_{1,\frac{\lambda+1}{2}}(y^2-1)^{-\frac{\lambda+1}{2}}\qquad\text{and}\qquad \bar F(y):=\int\limits_y^\infty f(t)\,\dint t,\qquad y>1,
\end{align*}
which allows us to rewrite $G(\alpha)$ as
$$
G(\alpha) = {\alpha^m\over m!}\int\limits_1^\infty \Big({f(y)\over \tilde c_{1,\frac{\lambda+1}{2}}}\Big)^{\lambda m+1\over\lambda+1}\,e^{-\alpha\bar F(y)}\,\dint y.
$$
Again, since the constant $\tilde c_{1,\frac{\lambda+1}{2}}$ does not influence the monotonicity behaviour of $G(\alpha)$, it is sufficient to prove that
$$
H(\alpha) := {\alpha^m\over m!}\int\limits_1^\infty f(y)^{\lambda m+1\over\lambda+1}\,e^{-\alpha\bar F(y)}\,\dint y
$$
is strictly monotone decreasing in $\alpha$ for $\alpha>0$ if $\lambda>1$ or $\alpha>(m-1)\pi$ if $\lambda=1$. We compute the derivative of $H(\alpha)$:
\begin{align*}
 H'(\alpha) &= {\alpha^{m-1}\over(m-1)!}\int\limits_1^\infty f(y)^{\lambda m+1\over\lambda+1}\,e^{-\alpha\bar F(y)}\,\dint y - {\alpha^m\over m!}\int\limits_1^\infty f(y)^{\lambda m+1\over\lambda+1}\,\bar F(y)\,e^{-\alpha\bar F(y)}\,\dint y\\
 &= {\alpha^{m-1}\over(m-1)!}\int\limits_1^\infty f(y)^{\lambda m+1\over\lambda+1}\,e^{-\alpha\bar F(y)}\,\Big(1-{\alpha\over m}\,\bar F(y)\Big)\,\dint y.
\end{align*}
Next, we substitute $z=\bar F(y)$ and introduce the abbreviations $L(z):=f(\bar F^{-1}(z))^{\lambda\over\lambda+1}$, $h(z):=e^{-\alpha z}$ and $g(z):=1-{\alpha\over m}z$ for $z>0$. This allows us to rewrite $H'(\alpha)$ as
$$
 H'(\alpha) = {\alpha^{m-1}\over(m-1)!}\int\limits_0^\infty L(z)^{m-1}\,h(z)\,g(z)\,\dint z,
$$
where we used that ${\lambda m+1\over\lambda+1}-1={\lambda(m-1)\over\lambda+1}$.

We observe now that $L(z)$ is strictly convex on $(0,\infty)$. Indeed, according to the chain rule the derivative of $L(z)$ equals $L'(z)=-{\lambda\over\lambda+1}f(\bar F^{-1}(z))^{{\lambda\over\lambda+1}-2}f'(\bar F^{-1}(z))$. To see that this function is strictly increasing, we first note that $-{\lambda\over\lambda+1}f(y)^{{\lambda\over\lambda+1}-2}f'(y)$ is strictly decreasing in $y>1$, since its derivative  $-\lambda(\tilde c_{1,\frac{\lambda+1}{2}})^{-1/(\lambda+1)}(y^2-1)^{-3/2}$ is strictly negative. Also, $\bar F^{-1}(z)$ is strictly decreasing in $z>0$, since $\bar F(z)$ itself is strictly decreasing in $z$ as well.

Let $z_0:=m/\alpha>0$ be the unique zero of the function $g(z)$. Using the strict convexity of $L(z)$ we have that $L(z)<{L(z_0)\over z_0}z$ on $(0,z_0)$ and $L(z)>{L(z_0)\over z_0}z$ on $(z_0,\infty)$. Together with the fact that $g$ is positive on $(0,z_0)$ and negative on $(z_0,\infty)$ this yields
\begin{align*}
 H'(\alpha) &= {\alpha^{m-1}\over(m-1)!}\Big(\int\limits_0^{z_0} L(z)^{m-1}\,h(z)\,g(z)\,\dint z+\int\limits_{z_0}^\infty L(z)^{m-1}\,h(z)\,g(z)\,\dint z\Big)\\
 &<{\alpha^{m-1}\over(m-1)!}\Big(\int\limits_0^{z_0} \Big({L(z_0)\over z_0}\,z\Big)^{m-1}\,h(z)\,g(z)\,\dint z + \int\limits_{z_0}^\infty \Big({L(z_0)\over z_0}\,z\Big)^{m-1}\,h(z)\,g(z)\,\dint z \Big)\\
 &={\alpha^{m-1}\over(m-1)!}\Big({L(z_0)\over z_0}\Big)^{m-1}\int\limits_0^\infty z^{m-1}\,h(z)\,g(z)\,\dint z.
\end{align*}
Finally, we observe that
$$
\int\limits_0^\infty z^{m-1}\,h(z)\,g(z)\,\dint z = \int\limits_0^\infty z^{m-1}\,e^{-\alpha z}\,\dint z - {\alpha\over m}\int\limits_0^\infty z^{m}\,e^{-\alpha z}\,\dint z,
$$
which after the substitution $u=\alpha z$  takes the form
$$
{1\over\alpha^m}\int\limits_0^\infty u^{m-1}\,e^{-u}\,\dint u - {1\over m\,\alpha^m}\int\limits_0^\infty u^m\,e^{-u}\,\dint u = {1\over\alpha^m}\Big(\Gamma(m)-{\Gamma(m+1)\over m}\Big)=0.
$$
This yields that $H'(\alpha)<0$ and shows that $H(\alpha)$ and consequently $\alpha\mapsto\EE f_k(P_{d,\alpha,\beta})$ is strictly decreasing in the respective range of $\alpha$. % for $\alpha>0$ if $\beta>(d+1)/2$ or $\alpha>(d-1)\pi$ if $\beta=(d+1)/2$.
\end{proof}

\section{Special cases and small dimensions: Proofs}

This section contains the proofs of the remaining results from Section~\ref{sec:SpecialCases}.

\subsection{Typical Poisson-Voronoi cells in small dimensions: Proof of Corollary~\ref{cor:k-faces_beta=d=234}}
%\begin{proof}[Proof of Corollary \ref{cor:k-faces_beta=d=234}]
The formula for $d=2$ can either be concluded from Theorem~\ref{thm:Tfunctional} by putting $a=b=0$ and $\beta=d=2$ there or from Theorem~\ref{thm:f-vector_beta_star}. From the latter we get
$$
\EE f_0(P_{2,\alpha,2})=\EE f_1(P_{2,\alpha,2}) = 2\,\mathbb{I}_{\alpha,2}^*(2),
$$
where we used the trivial values $\tilde{\mathbb{J}}_{2,1}(3/2)=\tilde{\mathbb{J}}_{2,2}(3/2)=1$.
From the definition of $\mathbb{I}_{\alpha,2}^*(2)$, see~\eqref{eq:I*_alt_expression_sinh_negative_powers},  we have
\begin{align*}
\mathbb{I}_{\alpha,2}^*(2)
=
{\alpha^2\over 2}\cdot{3\over 4}\int\limits_{0}^\infty (\sinh\varphi)^{-4}\exp\Bigg\{-{\alpha\over 2}\int\limits_\varphi^\infty(\sinh\theta)^{-2}\,\dint\theta\Bigg\}\dint\varphi
=
3\Big(1+{2\over \alpha}\Big),
\end{align*}
which proves the case $d=2$.
%The formula for $\EE f_0(P_{2,\alpha,2})$ follows.
Applying Theorem~\ref{thm:f-vector_beta_star} to the case $\beta=d=3$ yields
\begin{align*}
	\EE f_k(P_{3,\alpha,3})
%	&	=2\sum_{s=0}^\infty\mathbb I^*_{\alpha,3-2s}(3)\cdot \tilde{\mathbb J}_{3-2s,k+1}\Big(\frac 52 -s\Big)\\
		=2\mathbb I_{\alpha,3}^*(3)\cdot\tilde{\mathbb J}_{3,k+1}\Big(\frac 52\Big)+2\mathbb{I}^*_{\alpha,1}(3)\cdot \tilde{\mathbb J}_{1,k+1}\Big(\frac 32\Big).
\end{align*}
Now, we use that $\tilde{\mathbb J}_{1,k+1}(3/2)$ vanishes for $k\in\{1,2\}$ and is equal to $1$ for $k=0$. Together with the explicit values $\tilde{\mathbb J}_{3,1}(5/2)=1/2$ and $\tilde{\mathbb J}_{3,2}(5/2)=3/2$ (which are just angle sums in a triangle), this gives
\begin{align*}
	\EE f_k(P_{3,\alpha,3})
	&	=
	\begin{cases}
		\mathbb I^*_{\alpha,3}(3)+2 \mathbb I^*_{\alpha,1}(3),	&:k=0,\\
		 3 \mathbb I^*_{\alpha,3}(3),	&:k=1,\\
		 2 \mathbb I^*_{\alpha,3}(3) ,	&:k=2.
	\end{cases}
\end{align*}
However, applying Euler's relation $f_0(P)-f_1(P)+f_2(P)=2$, which holds for every polytope $P\subset\RR^3$, yields $\mathbb I^*_{\alpha,1}(3)=1$, and thus, completes the proof for $d=3$. The proof for $d=4$ is similar and uses the non-trivial values $\tilde {\mathbb J}_{4,1}(7/2) = 27/143$ and $\tilde {\mathbb J}_{4,2} (7/2) = 170/143$ that follow from~\eqref{eq:internal_angle_sum_J_formula}.
\hfill $\Box$
%\end{proof}

\subsection{The case \texorpdfstring{$\beta=(d+1)/2$}{beta=(d+1)/2}: Proof of Theorem~\ref{thm:f-vector,beta=(d+1)/2}}
We turn to the proof of the formula for the expected $f$-vector of $P_{d,\alpha,(d+1)/2}$ stated in Theorem~\ref{thm:f-vector,beta=(d+1)/2}. The next lemma provides an explicit formula for the expected external angle sums $\mathbb I^*_{\alpha,m}(1)$.

\begin{lemma}\label{lem:I_star_(1)}
	In the special case $\lambda=1$ and for $\alpha>\pi(m-1)$, it holds that
	\begin{align}\label{eq:I*(1)}
		\mathbb I^*_{\alpha,m}(1)
		&	=\frac{\alpha^m\Gamma(\frac{m+1}{2})}{m2^{m}\sqrt{\pi}\Gamma(\frac m2)}\cdot\frac{\Gamma(\frac{\alpha}{2\pi}-\frac{m-1}{2})}{\Gamma(\frac{\alpha}{2\pi}+\frac{m+1}{2})},\qquad m\in\NN.
	\end{align}
\end{lemma}
%%ZK: Numerisch korrekt.

\begin{proof}
	By definition of $\mathbb I^*_{\alpha,m}(\lambda)$ given in~\eqref{eq:def_I_star} and~\eqref{eq:def_I_star_sum}, we have
	\begin{align*}
		\mathbb I^*_{\alpha,m}(1)
		&	=\frac{\alpha^{m}}{m!}\int\limits_1^\infty \tilde c_{1,\frac{m+1}{2}}(y^2-1)^{-\frac{m+1}{2}} \exp\Bigg\{-\alpha\int\limits_{y}^\infty\tilde c_{1,1}(t^2-1)^{-1}\,\dint t\Bigg\}\,\dint y.
	\end{align*}
	Now, $\tilde{c}_{1,1}=1/\pi$ and $\int_y^\infty{\dint t\over t^2-1}={1\over 2}\log{y+1\over y-1}$ for $y>1$.
	Inserting this and the value of $\tilde c_{1,(m+1)/2}$ leads to
	\begin{align*}
		\mathbb I^*_{\alpha,m}(1)
		&	=\frac{\alpha^m\Gamma(\frac{m+1}{2})}{m!\sqrt{\pi}\Gamma(\frac m2)}\int\limits_1^\infty \bigg(\frac{y+1}{y-1}\bigg)^{-\frac\alpha{2\pi}}(y^2-1)^{-\frac{m+1}{2}}\,\dint y.
	\end{align*}
	Substituting $z=\frac{y+1}{y-1}-1 = \frac{2}{y-1}$  gives
	\begin{align*}
		\mathbb I^*_{\alpha,m}(1)
%		=\frac{\alpha^m\Gamma(\frac{m+1}{2})}{m!2^m\sqrt{\pi}\Gamma(\frac m2)}\int\limits_1^\infty (u-1)^{m-1}u^{-(\frac{\alpha}{2\pi}+\frac{m+1}{2})}\,\dint u
		=\frac{\alpha^m\Gamma(\frac{m+1}{2})}{m!2^m\sqrt{\pi}\Gamma(\frac m2)}\int\limits_0^\infty z^{m-1}(z+1)^{-(\frac{\alpha}{2\pi}+\frac{m+1}{2})}\,\dint z.
	\end{align*}
	Note that this integral converges since ${\alpha\over 2\pi}+{m+1\over 2}>m$ by our assumption $\alpha>(m-1)\pi$. Using the integral representation for the beta function and expressing the result through gamma functions we arrive at
	\begin{align*}
		\mathbb I^*_{\alpha,m}(1)
		=\frac{\alpha^m\Gamma(\frac{m+1}{2})}{m!2^m\sqrt{\pi}\Gamma(\frac m2)}\cdot \frac{\Gamma(m)\Gamma(\frac{\alpha}{2\pi}-\frac{m-1}{2})}{\Gamma(\frac{\alpha}{2\pi}+\frac{m+1}{2})}
		=\frac{\alpha^m \Gamma(\frac{m+1}{2})}{m2^{m}\sqrt{\pi}\Gamma(\frac m2)}\cdot\frac{\Gamma(\frac{\alpha}{2\pi}-\frac{m-1}{2})}{\Gamma(\frac{\alpha}{2\pi}+\frac{m+1}{2})}.
	\end{align*}
The proof is thus complete.
\end{proof}

\begin{proof}[Proof of Theorem~\ref{thm:f-vector,beta=(d+1)/2}]
We apply Theorem~\ref{thm:f-vector_beta_star} with $\beta= (d+1)/2$ to obtain
\begin{align}\label{eq:proof_exp_f_vect_beta=(d+1)/2}
\EE f_{\ell-1}(P_ {d,\alpha,\frac{d+1}{2}})
=
2\sum_{\substack{m \in \{\ell,\ldots, d\}\\ m \equiv d \Mod 2}}  \mathbb I^*_{\alpha,m}(1)\cdot\mathbb{\tilde {J}}_{m,\ell}\Big(\frac {m}2\Big).
\end{align}
The expected internal angle sums appearing in this formula are given by Theorem~4.1 from~\cite{KabluchkoZeroPolytope} which states that
	\begin{align}
		\mathbb{\tilde J}_{m,\ell}\Big(\frac m2\Big)
		&	=\frac{\pi^{\ell-m}}{\ell!}\cdot\frac{m (A[m,\ell]-A[m-2,\ell])}{2\tilde c_{1,\frac{m+1}{2}}}\label{eq:J_tilde_m/2},
	\end{align}
where $A[\,\cdot\,,\,\cdot\,]$ has been defined in~\eqref{eq:A-Terme}.
	Inserting~\eqref{eq:I*(1)} and~\eqref{eq:J_tilde_m/2} into~\eqref{eq:proof_exp_f_vect_beta=(d+1)/2} yields
	\begin{multline*}
		\EE f_{\ell-1}(P_ {d,\alpha,\frac{d+1}{2}})
		=2\sum_{\substack{m \in \{\ell,\ldots, d\}\\ m \equiv d \Mod 2}}
  \frac{\alpha^{m}\Gamma(\frac{m+1}{2})}{m 2^{m}\sqrt{\pi}\Gamma(\frac {m}2)}\cdot\frac{\Gamma(\frac{\alpha}{2\pi}-\frac{m-1}{2})}{\Gamma(\frac{\alpha}{2\pi}+\frac{m+1}{2})}
		\cdot \frac{\pi^{\ell-m}}{\ell!}\frac{m(A[m,\ell]-A[m-2,\ell]) }{2\tilde c_{1,\frac{m+1}{2}}}.
	\end{multline*}
	Using the definition of $\tilde c_{1,(m+1)/2}$ and simplifying the resulting expression completes the proof of the formula for $\EE f_{\ell-1}(P_ {d,\alpha,(d+1)/2})$.	
%The formula for $\EE f_{d-1}(P_ {d,\alpha,\frac{d+1}{2}})$ can be obtained as a special case.
\end{proof}

\subsection{The case \texorpdfstring{$\beta=(d+2)/2$}{beta= (d+2)/2}: Proof of Theorem~\ref{thm:f-vector,beta=(d+2)/2}}
To prove Theorem~\ref{thm:f-vector,beta=(d+2)/2}, we apply Theorem~\ref{thm:f-vector_beta_star} with $\beta= (d+2)/2$, which results in
\begin{equation}\label{eq:proof_exp_f_vect_beta=(d+2)/2}
\EE f_{\ell-1}(P_ {d,\alpha,\frac{d+2}{2}})
=
2\sum_{\substack{m \in \{\ell,\ldots, d\}\\ m \equiv d \Mod 2}}  \mathbb I^*_{\alpha,m}(2)\cdot\mathbb{\tilde {J}}_{m,\ell}\Big(\frac {m+1}2\Big).
\end{equation}
The internal angle sums appearing here were already evaluated in~\cite[Theorem~2.6]{KabluchkoAngles}:
\begin{align}\label{eq:J_tilde_(m+1)/2}
	\mathbb{\tilde J}_{m,\ell}\Big(\frac{m+1}{2}\Big)=\frac{1}{\binom{2m}m}\left(\binom m\ell\binom{m+\ell} \ell-\binom{m-2}\ell\binom{m-2+\ell}\ell\right).
\end{align}
The next lemma relates the external angle sums $\mathbb I^*_{\alpha,m}(2)$ to the modified Bessel function  $K_\nu(z)$ of the second kind as defined in~\eqref{eq:def_Bessel_2nd_kind}.

\begin{lemma}\label{lem:I_star_(2)}
	In the special case $\lambda=2$ and for all $\alpha>0$, we have
	\begin{align}\label{eq:I*(2)}
		\mathbb I^*_{\alpha,m}(2)
		&=\sqrt{\alpha e^{\alpha} \over\pi}\,{2m-1\choose m}\,K_{m-\frac 12}\left(\frac{\alpha}{2}\right),
\qquad m\in\NN.
		%&	=\frac{\sqrt{\alpha}2^{2m-1}e^{\frac\alpha 2}\Gamma(m+\frac 12)}{m!\pi}\cdot K\Big[m-\frac 12,\frac{\alpha}{2}\Big]
	\end{align}
\end{lemma}
%ZK: Numerisch korrekt.

\begin{proof}
	For $\mathbb I^*_{\alpha,m}(2)$, we use the definition~\eqref{eq:def_I_star} and insert the values of $\tilde c_{1,m+\frac 12}$ and $\tilde c_{1,\frac 32}$ to obtain
	\begin{align*}
		\mathbb I^*_{\alpha,m}(2)
		=\frac{\alpha^m\Gamma(m+\frac 12)}{m!\sqrt{\pi}\Gamma(m)}\int\limits_1^\infty(y^2-1)^{-(m+\frac 12)}\exp\Bigg\{-\frac\alpha 2\int\limits_y^\infty (t^2-1)^{-\frac 32}\,\dint t\Bigg\}\,\dint y.
	\end{align*}
	Since
	\begin{align*}
		\int\limits_y^\infty (t^2-1)^{-\frac 32}\,\dint t=\frac{y}{\sqrt{y^2-1}}-1,\qquad y>1,
	\end{align*}
	we have
	\begin{align*}
		\mathbb I^*_{\alpha,m}(2)=\frac{\alpha^m e^{\frac \alpha 2}\Gamma(m+\frac 12)}{m!\sqrt{\pi}(m-1)!}\int\limits_1^\infty(y^2-1)^{-(m+\frac 12)}\exp\left\{-\frac\alpha 2\cdot\frac{y}{\sqrt{y^2-1}}\right\}\,\dint y.
	\end{align*}
	Substituting $u=y/\sqrt{y^2-1}$ and recalling a formula for $K_\nu(z)$ given in~\eqref{eq:def_Bessel_2nd_kind} yields
	\begin{align*}
		\mathbb I^*_{\alpha,m}(2)
		=\frac{\alpha^m e^{\frac \alpha 2}\Gamma(m+\frac 12)}{m!\sqrt{\pi}(m-1)!}\int\limits_1^\infty(z^2-1)^{m-1}e^{-\frac\alpha 2\cdot u}\,\dint u
		=\frac{\sqrt{\alpha}2^{2m-1}e^{\frac\alpha 2}\Gamma(m+\frac 12)}{m!\pi}\cdot K_{m-\frac 12}\left(\frac{\alpha}{2}\right).
	\end{align*}
	Applying Legendre's duplication formula to the last expression yields that
	\begin{align*}
		\frac{\sqrt{\alpha}2^{2m-1}e^{\frac\alpha 2}\Gamma(m+\frac 12)}{m!\pi} = {\sqrt{\alpha}\,e^{\alpha/2}\over m!\,\sqrt{\pi}}\,{\Gamma(2m)\over\Gamma(m)} = \sqrt{\alpha e^{\alpha}\over\pi}\,{2m-1\choose m}.
	\end{align*}
	The proof is thus complete.
\end{proof}

\begin{proof}[Proof of Theorem~\ref{thm:f-vector,beta=(d+2)/2}]
Insert~\eqref{eq:J_tilde_(m+1)/2} and~\eqref{eq:I*(2)} into~\eqref{eq:proof_exp_f_vect_beta=(d+2)/2} and simplify the resulting expression.
%More details:
%	This follows from Theorem~\ref{thm:f-vector_beta_star} together with~\eqref{eq:I*(2)} and~\eqref{eq:J_tilde_(m+1)/2}. In fact, %we obtain
%	\begin{align*}
%		\EE f_k(P_{d,\alpha,{d+2\over 2}}) &= 2\sqrt{\alpha\over\pi}\,e^{\alpha/2}\,\sum_{s=0}^\infty {{2(d-2s)-1\choose %d-2s}\over{2(d-2s)\choose d-2s}}\Bigg({d-2s\choose k+1}{d-2s+k+1\choose k+1}\\
%		&\hspace{6cm}-{d-2s-2\choose k+1}{d-2s+k-1\choose k+1}\Bigg).
%	\end{align*}
%	The result follows by taking into account that ${{2(d-2s)-1\choose d-2s}/{2(d-2s)\choose d-2s}}=1/2$.
\end{proof}

\subsection*{Acknowledgement}

The authors are grateful to Ben Hansen and Tobias M\"uller for providing us with the simulation of the hyperbolic Poisson-Voronoi tessellation shown in the left panel of Figure~\ref{fig:Voronoi} and to the unknown referee for a careful reading of the manuscript.
TG and ZK were supported by the German Research Foundation under Germany's Excellence Strategy  EXC 2044 -- 390685587, \textit{Mathematics M\"unster: Dynamics - Geometry - Structure}  and TG, ZK and CT by the DFG priority program SPP 2265 \textit{Random Geometric Systems}.

\end{document}